\documentclass{article}
\usepackage[utf8]{inputenc}
\usepackage{bm,amsfonts,amsthm,amsmath,amssymb}
\usepackage{graphicx,color}
\usepackage{mathrsfs}
\usepackage{indentfirst}
\usepackage{bbm}
\usepackage{appendix}

\DeclareMathAlphabet{\mathpzc}{OT1}{pzc}{m}{it}

\def\eqdefa{\buildrel\hbox{\footnotesize def}\over =}
\newcommand{\ve}{\varepsilon}
\newcommand{\ud}{\mathrm{d}}

\newcommand{\vv}{\mathbf{v}}

\newcommand{\aaa}{\mathbf{a}}
\newcommand{\xx}{\mathbf{x}}

\newcommand{\nn}{\mathbf{n}}
\newcommand{\ee}{\mathbf{e}}

\newcommand{\mm}{\mathbf{m}}

\newcommand{\sss}{\mathbf{s}}

\newcommand{\A}{\mathbf{A}}

\newcommand{\BB}{\mathbf{B}}
\newcommand{\GG}{\mathbf{G}}
\newcommand{\HH}{\mathbf{H}}
\newcommand{\UU}{\mathbf{U}}

\newcommand{\MM}{\mathbf{M}}
\newcommand{\NN}{\mathbf{N}}

\newcommand{\VV}{\mathbf{V}}
\newcommand{\RR}{\mathbf{R}}
\newcommand{\QQ}{\mathbf{Q}}

\newcommand{\FF}{\mathbf{F}}

\newcommand{\JJ}{\mathbf{J}}

\newcommand{\KK}{\mathbf{K}}

\newcommand{\CR}{\mathcal{R}}

\newcommand{\CS}{\mathcal{S}}

\newcommand{\CV}{\mathcal{V}}
\newcommand{\CE}{\mathcal{E}}
\newcommand{\CN}{\mathcal{N}}
\newcommand{\CF}{\mathcal{F}}
\newcommand{\CG}{\mathcal{G}}

\newcommand{\CP}{\mathcal{P}}
\newcommand{\CH}{\mathcal{H}}
\newcommand{\CI}{\mathcal{I}}

\newcommand{\CY}{\mathcal{Y}}

\newcommand{\CM}{\mathcal{M}}

\newcommand{\CJ}{\mathcal{J}}

\newcommand{\Ff}{\mathfrak{F}}
\newcommand{\Ef}{\mathfrak{E}}
\newcommand{\bxi}{\bm{\xi}}

\newcommand{\ML}{\mathscr{L}}
\newcommand{\MA}{\mathscr{A}}
\newcommand{\MB}{\mathscr{B}}

\newcommand{\MP}{\mathscr{P}}

\newcommand{\MS}{\mathscr{S}}

\newcommand{\Fi}{\mathfrak{i}}
\newcommand{\Fp}{\mathfrak{p}}
\newcommand{\Fq}{\mathfrak{q}}

\newcommand{\BOm}{\mathbf{\Omega}}

\newtheorem{theorem}{Theorem}
\newtheorem{proposition}[theorem]{Proposition}
\newtheorem{lemma}[theorem]{Lemma}
\newtheorem{corollary}[theorem]{Corollary}

\newtheorem{assumption}{Assumption}
\numberwithin{theorem}{section}
\numberwithin{equation}{section}

\allowdisplaybreaks[4]
\title{Rigorous biaxial limit of a molecular-theory-based two-tensor hydrodynamics}
\author{Sirui Li\footnote{ School of Mathematics and Statistics, Guizhou University, Guiyang 550025, China (srli@gzu.edu.cn) }, Jie Xu\footnote{LSEC and NCMIS, Institute of Computational Mathematics and Scientific/Engineering Computing (ICMSEC), Academy of Mathematics and Systems Science (AMSS), Chinese Academy of Sciences, Beijing, China (xujie@lsec.cc.ac.cn)}}

\date{}

\begin{document}
\maketitle

\begin{abstract}
We consider a two-tensor hydrodynamics derived from the molecular model, where high-order tensors are determined by closure approximation through the maximum entropy state or the quasi-entropy. We prove the existence and uniqueness of local in time smooth solutions to the two-tensor system. Then, we rigorously justify the connection between the molecular-theory-based two-tensor hydrodynamics and the biaxial frame hydrodynamics. More specifically, in the framework of Hilbert
expansion, we show the convergence of the solution to the two-tensor hydrodynamics to the solution to the frame hydrodynamics.
\end{abstract}

\tableofcontents

\section{Introduction}

Liquid-crystal theories are classified by how the local anisotropy is described.
For the uniaxial nematic phase formed by rod-like molecules, the local anisotropy can be described by an orientational density function, a second-order symmetric traceless tensor, or a unit vector.
The resulting hydrodynamics are molecular models (such as Doi--Onsager \cite{DE}), tensor models (such as Beris--Edwards \cite{BE} and Qian--Sheng \cite{QS}) and vector models (such as Ericksen--Leslie \cite{E-61,E-91,Les}), respectively.
The connections between theories at these three levels are studied extensively, both in the sense of formal expansion \cite{KD,EZ,HLWZZ} and rigorous limit \cite{WZZ2,LWZ,WZZ3,LW,Liu-Wang, Xin-Zhang}.

When the molecular shape becomes more complex, other types of nematic phases are also observed.
For these phases, however, the connections between different level of theories are merely revealed.
Let us focus on the biaxial nematic phase on which we will focus throughout this article.
At the coarsest level, it needs to be described by an orthonormal frame field $\Fp(\xx)\in SO(3)$.
Under this description, the form of orientational elasticity \cite{S-A,GV1,SV} and hydrodynamics \cite{S-A,Liu-M,BP,S-W-M,GV2,LLC} have been proposed, which we would call frame models.
On the other hand, most molecular models and tensor models (typically with two tensors) are only built for spatially homogeneous cases, no matter for equilibria \cite{S-J-P,MDN,APK,BVG,SVD} and hydrodynamics \cite{SLW,SW1,SW2}.
Under this circumstance, it is not the timing to study the connections between these models.
%
%
%
%
%
%
%

A couple of recent works construct the spatially inhomogeneous free energy \cite{XYZ} and hydrodynamics \cite{XZ} for bent-core molecules that exhibit the biaxial nematic phase. The free energy is a functional of three tensors and is derived from molecular interactions, based on which the molecular model and tensor model for hydrodynamics are built.
When we restrain our interest to the biaxial nematic phase, the order parameters in these models can be reduced to two second-order tensors.
%
%
%
%
With these models, the very recent work \cite{LX} formally derive a biaxial limit of a molecular-theory-based hydrodynamics \cite{XZ} using the Hilbert expansion.
The biaxial limit model, written in the coordinates of the frame, is just the frame hydrodynamics given previously.
In the frame model, the numerous coefficients are expressed by those in the tensor model that have been derived as functions of molecular parameters, and the energy dissipation law is maintained in the frame model.
%
%
%
%
The Ericksen--Leslie model can also be recovered as a special case.
Furthermore, armed with the form of frame hydrodynamics in \cite{LX}, its well-posedness of smooth solutions in $\mathbb{R}^d(d=2,3)$ and the global existence of weak solutions in $\mathbb{R}^2$ are shown \cite{LWX}.
The uniqueness of global weak solutions is also established using the Littlewood--Paley theory \cite{LWX2}.

The main goal of this paper is to prove the local well-posedness of smooth solutions to the molecular-theory-based two-tensor hydrodynamics, and to rigorously justify the connection between the two-tensor hydrodynamics and the frame hydrodynamics in the sense of smooth solutions.
The main framework of our proof is constructing approximate solutions near the solution to frame hydrodynamics using the Hilbert expansion, then deriving the uniform estimates for the difference between the exact and the approximate solution.
Some new issues arise in the estimates compared with the uniaxial limit for one-tensor hydrodynamics, which we outline below.

The molecular-theory-based tensor model requires a stabilizing entropy term in the bulk energy.
The model also involves many high-order tensors.
They need to be expressed by the two order parameter tensors in some way, which is called the closure approximation.
In this paper, we consider two different approaches, that is, using the maximum entropy state \cite{XYZ} and the quasi-entropy \cite{Xu3} to give the entropy term and the closure approximation.
The maximum entropy state is a somewhat standard approach, but is expensive in computation.
The quasi-entropy is proposed as a elementary function substitution for the entropy term derived from the maximum entropy state, which we will explain in the following section.
In several simple cases discussed in \cite{Xu3}, the physics given by the quasi-entropy is consistent with that given by the maximum entropy state.
In the light of proposing a model with adequate computational complexity, it is significant to discuss closure approximation using the quasi-entropy.
Concerning the estimates for the rigorous biaxial limit, we shall see these two approaches will lead to the same arguments, so that the derivations afterwards can be handled in a unified manner.

To recognize the biaxial limit, it is necessary to comprehend the biaxial minimizer of the bulk energy.
The rotational invariance of the bulk energy leads to the fact that the minimizer is actually a three-dimensional manifold.
The tangent vectors of this manifold are zero-eigenvectors of the Hessian that is usually called in terms of `the linearized operator' in previous works.
A basic assumption is that the tangent vectors exactly span the kernel of the Hessian.
This assumption implies that the biaxial minimizer is nondegenerate, which physically means that the coefficients in the bulk energy are not at the critical values of phase transitions.
We will give some numerical evidences to this assumption.

The kernel of the linearized operator plays a key role in the analysis of the Hilbert expansion.
In particular, we can use this kernel to decouple the Hilbert expansion at different orders by projecting the tensors according to the kernel and its orthogonal complement.
%
As for the estimates, an essential difference from the previous works is that the projection and high-order tensors are noncommutative.
To deal with it, we will give a new and general way to obtain the desired estimates.
This is also the case for the singular term in the remainder equation, where we shall utilize the eigen-decomposition of the linearized operator.
We will clarify these two points after we write down the model and the main results.

To our knowledge, this is the first time a rigorous biaxial limit of hydrodynamics from a molecular-based model is established.
Moreover, it is possible that the minimizer of the bulk energy is uniaxial.
In this case, the limit model will be the Ericksen--Leslie model (see Section 5 in \cite{LX}).
When attempting to establish the uniaxial limit of two-tensor hydrodynamics, one will also face the issues described in the paragraphs above.
In other words, those approaches in the previous works on the uniaxial limit of the one-tensor hydrodynamics are special cases that cannot be extended naturally.
Instead, we need to follow the approach for the biaxial case.
We shall explain this at the end of this article.

In the rest of this section, let us introduce the notations, followed by the tensor and frame hydrodynamics.
Then, we state the main results.

\subsection{Preliminaries}

We begin with notations of orthonormal frames and tensors.
In a system composed of identical non-spherical rigid molecules, we need to describe its orientational distribution.
To this end, on each molecule we anchor a right-handed orthonormal frame $(\mm_1,\mm_2,\mm_3)$ to express its orientation.
If we write out the coordinates of the body-fixed frame in the right-handed reference frame, it gives a rotation matrix $\Fq\in SO(3)$, where $\Fq_{ij}$ is the $i$th coordinate of $\mm_j$.
%
We need another orthonormal frame to represent the local orientation of the nematic phase.
To distinguish, we shall use the notation $\Fp=(\nn_1,\nn_2,\nn_3)$.

For an $n$-th order tensor $U$ in $\mathbb{R}^3$, its coordinates in the reference frame is denoted by $U_{i_1\ldots i_n}$.
The dot product $U\cdot V$ of two tensors of the same order is defined by summing up the product of the corresponding coordinates,
\begin{align*}
U\cdot V=U_{i_1\cdots i_n}V_{i_1\cdots i_n},\quad|U|^2=U\cdot U,
\end{align*}
where we have adopted the Einstein summation convention on repeated indices and will assume it throughout the article.

An $n$-th order tensor $U$ is said to be symmetric, if its coordinates satisfy $U_{i_{\sigma(1)}\ldots i_{\sigma(n)}}=U_{i_1\cdots i_n}$ for arbitrary permutation $\sigma$ of $\{1,\ldots,n\}$.
A tensor $U$ can be symmetrized as
\begin{align*}
(U_{\rm sym})_{i_1\cdots i_n}=\frac{1}{n!}\sum_{\sigma}U_{i_{\sigma(1)}\cdots i_{\sigma(n)}},
\end{align*}
where the sum is taken over all the permutations. For any $n$-th order symmetric tensor $U$, its trace is an $(n-2)$th-order tensor defined by the contraction of two coordinates,
\begin{align*}
({\rm tr}U)_{i_1\ldots i_{n-2}}=U_{i_1\ldots i_{n-2}kk}.
\end{align*}
If a symmetric tensor $U$ satisfies ${\rm tr}U=0$, we say that $U$ is symmetric traceless.
If we express a tensor by its coordinates in another right-handed orthonormal frame, the symmetric and traceless properties are maintained, i.e. they are intrinsic properties of a tensor.

To express the symmetric tensors conveniently, we introduce the monomial notation,
\begin{align}\nonumber
\mm^{k_1}_1\mm^{k_2}_2\mm^{k_3}_3=\Big(\underbrace{\mm_1\otimes\cdots\otimes\mm_1}_{k_1}\otimes\underbrace{\mm_2\otimes\cdots\otimes\mm_2}_{k_2}\otimes\underbrace{\mm_3\otimes\cdots\otimes\mm_3}_{k_3}\Big)_{\rm sym}.
\end{align}
In this way, any homogeneous polynomical of $\mm_i$ represents a symmetric tensor.
The $3\times 3$ identity matrix $\Fi$ can be written as a polynomial $\Fi=\mm^2_1+\mm^2_2+\mm^2_3$.
For clarity, several simple examples are given below:
\begin{align*}
\mm_1\mm_2=&\,\frac{1}{2}(\mm_1\otimes\mm_2+\mm_2\otimes\mm_1),\quad\mm^2_1=\mm_1\otimes\mm_1,\\
\mm_1\mm_2\mm_3=&\,\frac{1}{6}\big(\mm_1\otimes\mm_2\otimes\mm_3+\mm_2\otimes\mm_3\otimes\mm_1+\mm_3\otimes\mm_1\otimes\mm_2\\
&\,+\mm_1\otimes\mm_3\otimes\mm_2+\mm_2\otimes\mm_1\otimes\mm_3+\mm_3\otimes\mm_2\otimes\mm_1\big),\\
\mm_1\mm^2_2=&\,\frac{1}{3}\big(\mm_1\otimes\mm_2\otimes\mm_2+\mm_2\otimes\mm_1\otimes\mm_2+\mm_2\otimes\mm_2\otimes\mm_1\big).
\end{align*}
The definition above also works for $\Fp=(\nn_1,\nn_2,\nn_3)$.

The order parameters to depict the local anisotropy are defined from the moments of $\mm_i$ of the density function $\rho(\Fq,\xx)$,
\begin{align}\nonumber
\big\langle\mm_{i_1}\otimes\cdots\otimes\mm_{i_n}\big\rangle=\int_{SO(3)}\mm_{i_1}(\Fq)\otimes\cdots\otimes\mm_{i_n}(\Fq)\rho(\Fq,\xx)\ud\Fq,\quad i_1,\ldots,i_n=1,2,3.
\end{align}
where $\langle\cdot\rangle$ is a short notation for the average under $\rho(\Fq,\xx)$. To extract the linearly independent components from these moments, it turns out that order parameters shall be chosen from averaged symmetric traceless tensors (see \cite{Xu1,XC,Xu4} for details).

The hydrodynamics involves differential operators on $SO(3)$.
For any frame $\Fp=(\nn_1,\nn_2,\nn_3)\in SO(3)$, the tangential space $T_{\Fp}SO(3)$ at a point $\Fp$ can be spanned by the orthogonal basis
\begin{align*}
V_1=(0,\nn_3,-\nn_2),\quad V_2=(-\nn_3,0,\nn_1),\quad V_3=(\nn_2,-\nn_1,0).
\end{align*}
The differential operators $\ML_k(k=1,2,3)$ are defined by taking the dot products of $V_k$ and
$\partial/\partial\Fp=(\partial/\partial\nn_1,\partial/\partial\nn_2,\partial/\partial\nn_3)$,
\begin{align}\label{diff-ML1-3}
\left\{
    \begin{aligned}
&\ML_1\eqdefa V_1\cdot\frac{\partial}{\partial\Fp}=\nn_3\cdot\frac{\partial}{\partial\nn_2}-\nn_2\cdot\frac{\partial}{\partial\nn_3},\\
&\ML_2\eqdefa V_2\cdot\frac{\partial}{\partial\Fp}=\nn_1\cdot\frac{\partial}{\partial\nn_3}-\nn_3\cdot\frac{\partial}{\partial\nn_1},\\
&\ML_3\eqdefa V_3\cdot\frac{\partial}{\partial\Fp}=\nn_2\cdot\frac{\partial}{\partial\nn_1}-\nn_1\cdot\frac{\partial}{\partial\nn_2}.
\end{aligned}
    \right.
\end{align}
The operator $\ML_k$ actually gives the derivative along the infinitesimal rotation about $\nn_k$. Acting the operators $\ML_k(k=1,2,3)$ on $\nn_i$, we can verify that $\ML_k\nn_i=\epsilon^{ijk}\nn_j$ with $\epsilon^{ijk}$ being the Levi-Civita symbol.
The operator $\ML_k$ can also be acted on a functional, where $\partial/\partial\Fp$ shall be replaced by the variational derivative $\delta/\delta\Fp$.

\subsection{Two-tensor hydrodynamics}\label{tensor-models}

In this paper, we focus on a two-tensor hydrodynamics considered in \cite{LX} for biaxial nematics, which is reduced from the model proposed in \cite{XZ}.

The order parameters are given by two second-order symmetric traceless tensors,
\begin{align*}
Q_1=\langle\mm^2_1-\Fi/3\rangle,\quad Q_2=\langle\mm^2_2-\Fi/3\rangle.
\end{align*}
We denote $\QQ=(Q_1,Q_2)^T$ in short.
A projection map $\MS$ is defined for second-order tensor,
\begin{align*}
    (\MS R)_{ij}=\frac 12(R_{ij}+R_{ji})-\frac 13 R_{kk}\delta_{ij},
\end{align*}
giving a symmetric traceless one.
It can also be imposed on an array of second-order tensors, i.e.,
\begin{align*}
\MS(R_1,\cdots,R_k)=(\MS R_1,\cdots,\MS R_k).
\end{align*}
Denote by $\mathbb{Q}$ the linear space formed by a pair of second-order symmetric traceless tensors
\begin{align*}
    \mathbb{Q}=\big\{\QQ=(Q_1,Q_2)^T:Q_i\text{ second-order symmetric traceless}\big\}.
\end{align*}
Its dimension is ten, since each second-order symmetric traceless tensor contributes five.

Now we are in the position of writing down the tensor hydrodynamics.
We shall adopt the form in \cite{LX} that would clearly reflect its structure.
To simplify the presentation, throughout the paper, we assume that the concentration of rigid molecules at each point $\xx$, the product of the Boltzmann constant and the absolute temperature, and the density of the fluid $\rho_s$ are all equal to one, so that they will not appear in the model compared with those given previously.
Such simplifications make no difference in the structure of the model and the estimates to be established.

We begin with the free energy
\begin{align}\label{free-energy}
{\CF[\QQ,\nabla\QQ]}=\int_{\mathbb{R}^3}\Big(\frac{1}{\ve}F_b(\QQ)+F_e(\nabla\QQ)\Big)\ud\xx.
\end{align}
The energy density is divided into the bulk part $F_b$ and the elastic part $F_e$.
The bulk energy density $F_b$ contains an entropy term and quadratic terms of $\QQ$,
\begin{align}
    F_b(\QQ)=&\,F_{\rm entropy}+\frac{1}{2}\big(c_{02}|Q_1|^2+c_{03}|Q_2|^2+2c_{04}Q_1\cdot Q_2\big).\label{free-energy-bulk}
\end{align}
The entropy term requires detailed discussion, which will be presented in Section \ref{key-section}.
The elastic energy density $F_e$, penalizing spatial inhomogeneity, consists of quadratic terms of $\nabla\QQ$,
\begin{align}
    F_e(\nabla\QQ)=&\,\frac{1}{2}\Big(c_{22}|\nabla Q_1|^2+c_{23}|\nabla Q_2|^2 +2c_{24}\partial_iQ_{1jk}\partial_iQ_{2jk}\nonumber\\
    &\,\quad+c_{28}\partial_iQ_{1ik}\partial_jQ_{1jk}+c_{29}\partial_iQ_{2ik}\partial_jQ_{2jk} +2c_{2,10}\partial_iQ_{1ik}\partial_jQ_{2jk}\Big). \label{free-energy-elastic}
\end{align}
To ensure the lower-boundedness of the free energy, we assume that the coefficients of the elastic energy satisfy positive definite conditions
\begin{align*}
    &\,c_{22}, c_{23}, c_{28}, c_{29}>0,\\
    &\,c_{24}^2<c_{22}c_{23},\quad c_{28}^2<c_{29}c_{2,10}.
\end{align*}
In the above, the coefficients $c_{ij}$ of the quadratic terms in (\ref{free-energy-bulk}) and (\ref{free-energy-elastic}) can be derived as functions of the molecular parameters (see \cite{XZ1,XYZ} for details).
In particular, it has been verified that these derived coefficients indeed satisfy the positive definite conditions given above (cf. \cite{XYZ}).
The small parameter $\ve$ in \eqref{free-energy} can be viewed as the squared relative scale $\tilde{L}$ between the rigid molecule and the domain of observation by a change of variable $\tilde{\xx}=\xx/\tilde{L}$.

We write the variational derivative of the free energy  \eqref{free-energy} as
\begin{align}\label{vari-deriv-Q}
\mu_{\QQ}=&\,~\frac{\delta \CF(\QQ,\nabla\QQ)}{\delta\QQ}=\MS\Big(\frac{1}{\ve}\frac{\partial F_b(\QQ)}{\partial\QQ}-\partial_i\Big(\frac{\partial F_e(\nabla\QQ)}{\partial (\partial_i\QQ)}\Big)\Big)\nonumber\\
\eqdefa&\,~\frac{1}{\ve}\CJ(\QQ)+\CG(\QQ),
\end{align}
where $\mu_{\QQ}=(\mu_{Q_1}, \mu_{Q_2})^T$, $\CJ(\QQ)=\big(\CJ_1(\QQ),\CJ_2(\QQ)\big)^T$ and $\CG(\QQ)=\big(\CG_1(\QQ),\CG_2(\QQ)\big)^T$ are calculated as
\begin{align}
\mu_{Q_1}=&\,\frac{1}{\ve}\CJ_1(\QQ)+\CG_1(\QQ)\nonumber\\
=&\,\frac{1}{\ve}\left(\MS\frac{\partial F_{\rm entropy}}{\partial Q_1}+c_{02}Q_1+c_{04}Q_2\right)\nonumber\\
&\,-c_{22}\Delta Q_{1jk}-c_{24}\Delta Q_{2jk}
-\MS(c_{28}\partial_j\partial_iQ_{1ik}+c_{2,10}\partial_j\partial_iQ_{2ik}),\label{mu-Q1}\\
\mu_{Q_2}=&\,\frac{1}{\ve}\CJ_2(\QQ)+\CG_2(\QQ)\nonumber\\
=&\,\frac{1}{\ve}\left(\MS\frac{\partial F_{\rm entropy}}{\partial Q_2}+c_{04}Q_1+c_{03}Q_2\right)\nonumber\\
&\,-c_{24}\Delta Q_{1jk}-c_{23}\Delta Q_{2jk}
-\MS(c_{2,10}\partial_j\partial_iQ_{1ik}+c_{29}\partial_j\partial_iQ_{2ik}).\label{mu-Q2}
\end{align}

Before we present the hydrodynamics, we explain some physical parameters that will appear below, for which we refer to \cite{XZ,LX}:
$I_{ii}(i=1,2,3)$ are diagonal elements of the moment of inertia for a molecule;
$\Gamma_i=\frac{m_0}{\zeta I_{ii}}(i=1,2,3)$ are the diffusion coefficients, where $m_0$ is a mass unit and $\zeta$ is a friction constant;
$e_i(i=1,2)$ are defined as $e_1=1-e_2=\frac{I_{22}}{I_{11}+I_{22}}$.

In addition, we define some fourth-order tensors:
\begin{align}\label{seven-tensors}
\left\{
    \begin{aligned}
&\CR_1=\big\langle\big(\mm^2_1-\Fi/3\big)\otimes\big(\mm^2_1-\Fi/3\big)\big\rangle,\quad
\CR_2=\big\langle\big(\mm^2_2-\Fi/3\big)\otimes\big(\mm^2_2-\Fi/3\big)\big\rangle, \\
&\CR_3=4\langle\mm_1\mm_2\otimes\mm_1\mm_2\rangle,\quad
\CR_4=4\langle\mm_1\mm_3\otimes\mm_1\mm_3\rangle,\quad
\CR_5=4\langle\mm_2\mm_3\otimes\mm_2\mm_3\rangle,\\
&\CV_{Q_1}=2\Big(\langle\mm_1\mm_3\otimes(\mm_1\otimes\mm_3)\rangle+e_1\langle\mm_1\mm_2\otimes(\mm_1\otimes\mm_2)\rangle-e_2\langle\mm_1\mm_2\otimes(\mm_2\otimes\mm_1)\rangle\Big), \\
&\CV_{Q_2}=2\Big(\langle\mm_2\mm_3\otimes(\mm_2\otimes\mm_3)\rangle-e_1\langle\mm_1\mm_2\otimes(\mm_1\otimes\mm_2)\rangle+e_2\langle\mm_1\mm_2\otimes(\mm_2\otimes\mm_1)\rangle\Big).
\end{aligned}
\right.
\end{align}
These tensors need to be specified as functions of $\QQ$, which is the so-called closure approximation to be discussed in Section \ref{key-section}.
The closure approximation is the origin of strong nonlinearity in the hydrodynamics.
We then define some operators from these tensors,
\begin{align}
&\,\CM_{\QQ}=\left(
  \begin{array}{cc}
    \CM_{11} &\, \CM_{12}\vspace{0.5ex} \\
    \CM_{12} &\, \CM_{22} \\
  \end{array}
\right)\eqdefa
\left(
  \begin{array}{cc}
    \Gamma_2\CR_4+\Gamma_3\CR_3 &\, -\Gamma_3\CR_3 \vspace{0.5ex}\\
    -\Gamma_3\CR_3 &\, \Gamma_1\CR_5+\Gamma_3\CR_3 \\
  \end{array}
\right), \label{diss-operator}\\
&\,\CV_{\QQ}\eqdefa\left(\begin{array}{c}
    \CV_{Q_1}  \vspace{0.5ex}\\
    \CV_{Q_2}
\end{array}\right),
\quad
\CN_{\QQ}\eqdefa(\CN_{Q_1},\CN_{Q_2})=(\CV^{T}_{Q_1}, \CV^T_{Q_2}),\label{CVQ-CNQ}\\
&\,\CP_{\QQ}\eqdefa \zeta\big(I_{22}\CR_{1}+I_{11}\CR_2+e_1I_{11}\CR_{3}\big). \label{CPQ}
\end{align}
Let us explain some short notations concerning the tensor contraction.
We regard fourth-order and second-order tensors as matrices and vectors, respectively.
Then, we make use of matrix-matrix and matrix-vector multiplications, such as
\begin{align*}
(\CV_{Q_1})_{ijkl}\kappa_{kl}=(\CV_{Q_1}\kappa)_{ij},
\end{align*}
where $\kappa_{ij}=\partial_jv_i$ is the gradient of the fluid velocity field $\vv$.
Thus, the transpose of a fourth-order tensor can be defined, such as $(\CV^T_{Q_1})_{ijkl}=(\CV_{Q_1})_{klij}$.
Again, $\CM_{\QQ}\mu_{\QQ}$ is carried out by matrix-vector multiplication,
\begin{align*}
    \CM_{\QQ}\mu_{\QQ}=\left(
    \begin{array}{c}
        \CM_{11}\mu_{Q_1}+\CM_{12}\mu_{Q_2} \vspace{0.5ex}\\
        \CM_{12}\mu_{Q_1}+\CM_{22}\mu_{Q_2}
    \end{array}\right).
\end{align*}

Now, let us write down the two-tensor hydrodynamics,
\begin{align}
\frac{\partial\QQ}{\partial t}+\vv\cdot\nabla\QQ=&\,-\CM_{\QQ}\mu_{\QQ}+\CV_{\QQ}\kappa,\label{Re-MB-Q-tensor-1}\\
\Big(\frac{\partial\vv}{\partial t}+\vv\cdot\nabla\vv\Big)_i=&\,-\partial_ip+\eta\Delta v_i+\partial_j(\CP_{\QQ}\kappa+\CN_{\QQ}\mu_{\QQ})_{ij}+\mu_{\QQ}\cdot\partial_i\QQ,\label{Re-MB-Q-tensor-2}\\
\nabla\cdot\vv=&\,0. \label{Re-MB-Q-tensor-3}
\end{align}
Here, $p$ is the pressure imposing the incompressibility (\ref{yuan-incompressible-v}) on $\vv$; $\eta$ is the viscous coefficient of the fluid; $\mu_{\QQ}=\ve^{-1}\CJ(\QQ)+\CG(\QQ)$; $\CJ(\QQ)$ and $\CG(\QQ)$ can be rewritten as
\begin{align}
\CJ(\QQ)=&\MS\frac{\partial F_{\rm entropy}}{\partial\QQ}+D_0\QQ,\label{CJ-QQ}\\
\CG(\QQ)_{jk}=&-D_1(\Delta\QQ)_{jk}-D_2\MS(\partial_j\partial_i\QQ_{ik}),\label{CG-QQ}
\end{align}
where the constant coefficient matrices $D_i(i=1,2,3)$ are expressed as, respectively,
\begin{align*}
D_0=\left(
  \begin{array}{cc}
    c_{02} & c_{04} \\
    c_{04} & c_{03}
  \end{array}
\right),\quad
D_1=\left(
  \begin{array}{cc}
    c_{22} & c_{24} \\
    c_{24} & c_{23}
  \end{array}
\right),\quad
D_2=\left(
  \begin{array}{cc}
    c_{28} & c_{2.10} \\
    c_{2.10} & c_{29}
  \end{array}
\right).
\end{align*}
To comprehend the model \eqref{Re-MB-Q-tensor-1}--\eqref{Re-MB-Q-tensor-3},  $\CM_{\QQ}\mu_{\QQ}$ is the rotational diffusion term; $\CV_{\QQ}\kappa$ is the rotational convection term, where the corresponding elastic stress is $\CN_{\QQ}\mu_{\QQ}$; $\CP_{\QQ}\kappa$ is the extra viscous stress induced by rigid molecules; and $\mu_{\QQ}\cdot\partial_i\QQ$ is the external force due to the presence of rigid molecules.


It is crucial that the fourth-order tensors $\CR_i(i=1,\cdots,5)$ are all positive definite, in the sense that for any second-order symmetric traceless tensor $Y$, we have $Y\cdot \CR_i Y\ge 0$ and the equality implies $Y=0$.
This can be guaranteed by the closure approximation to be introduced in Section \ref{key-section}, which has been discussed in \cite{LX}.
As a result, the operators $\CM_{\QQ}$ and $\CP_{\QQ}$ are positive definite.

The energy dissipation is given by
\begin{align}
    \frac{\ud}{\ud t} \left(\frac{1}{2}\|\vv\|_{L^2}^2+\CF[\QQ,\nabla\QQ]\right)
    =-(\mu_{\QQ},\CM_{\QQ}\mu_{\QQ})-\eta\|\kappa\|_{L^2}^2-(\kappa,\CP_{\QQ}\kappa).
\end{align}

\subsection{Biaxial frame hydrodynamics}

The local orientation of the biaxial nematic phase needs to be described by an orthonormal frame $\Fp=(\nn_1,\nn_2,\nn_3)\in SO(3)$.
The biaxial frame hydrodynamics can be derived formally from the tensor hydrodynamics above.
In this case, the frame $\Fp$ is that in the minimizer of the bulk energy $F_b$, written as
\begin{align}
    Q_i=s_i\Big(\nn_1\otimes\nn_1-\frac{\Fi}{3}\Big)+b_i(\nn_2\otimes\nn_2-\nn_3\otimes\nn_3),~~i=1,2. \label{mnmzform}
\end{align}
We will discuss it more in Section \ref{key-section}. Here, we focus on writing down the biaxial hydrodynamics.

We write the biaxial orientational elasticity in the form
\begin{align}\label{elastic-energy}
\CF_{Bi}[\Fp]=\int_{\mathbb{R}^3}f_{Bi}(\Fp,\nabla\Fp)\ud\xx.
\end{align}
The elastic energy density $f_{Bi}$ can be given by
(see \cite{SV,GV1,Xu2}, where other equivalent forms are provided):
\begin{align*}
    f_{Bi}(\Fp,\nabla\Fp)=&\,\frac{1}{2}\Big(\sum_{i=1}^3 K_i(\nabla\cdot\nn_i)^2+\sum_{i,j=1}^3K_{ij}(\nn_i\cdot\nabla\times\nn_j)^2
    +\sum_{i=1}^3\gamma_i\nabla\cdot[(\nn_i\cdot\nabla)\nn_i-(\nabla\cdot\nn_i)\nn_i]
    \Big),
\end{align*}
which consists of twelve bulk terms and three surface terms.
The coefficients $K_i$ and $K_{ij}$ of bulk terms shall all be positive. They can be derived from $c_{2j}$ in the tensor model (cf. \cite{Xu2}).

To present the frame hydrodynamics more conveniently, we introduce a set of local basis generated by nine second-order tensors: the identity tensor $\Fi$, five symmetric traceless tensors,
\begin{align*}
    \sss_1=\nn^2_1-\frac13\Fi,\quad \sss_2=\nn^2_2-\nn^2_3,\quad \sss_3=\nn_1\nn_2,\quad \sss_4=\nn_1\nn_3,\quad \sss_5=\nn_2\nn_3,
\end{align*}
and three asymmetric traceless tensors,
\begin{align*}
    \aaa_1=\nn_1\otimes\nn_2-\nn_2\otimes\nn_1,\quad \aaa_2=\nn_3\otimes\nn_1-\nn_1\otimes\nn_3,\quad \aaa_3=\nn_2\otimes\nn_3-\nn_3\otimes\nn_2.
\end{align*}
%
%
The frame hydrodynamics is a coupled system between the evolution equation of the orthonormal frame field $\Fp=(\nn_1,\nn_2,\nn_3)\in SO(3)$ and the Navier-Stokes equations. We write down the form given in \cite{LX}:
\begin{align}
&\,\chi_1\dot{\nn}_2\cdot\nn_3-\frac{1}{2}\chi_1\BOm\cdot\aaa_3-\eta_1\A\cdot\sss_5+\ML_1\CF_{Bi}=0,\label{frame-equation-n1}\\
&\,\chi_2\dot{\nn}_3\cdot\nn_1-\frac{1}{2}\chi_2\BOm\cdot\aaa_2-\eta_2\A\cdot\sss_4+\ML_2\CF_{Bi}=0,\label{frame-equation-n2}\\
&\,\chi_3\dot{\nn}_1\cdot\nn_2-\frac{1}{2}\chi_3\BOm\cdot\aaa_1-\eta_3\A\cdot\sss_3+\ML_3\CF_{Bi}=0,\label{frame-equation-n3}\\
&\,\Fp=(\nn_1,\nn_2,\nn_3)\in SO(3),\label{frame-SO3}\\
&\,\dot{\vv}=-\nabla p+\eta\Delta\vv+\nabla\cdot\sigma+\mathfrak{F},\label{yuan-frame-equation-v}\\
&\,\nabla\cdot\vv=0,\label{yuan-incompressible-v}
\end{align}
where we use the notation $\dot{f}=\partial_tf+\vv\cdot\nabla f$ to represent the material derivative, and recall that $\ML_k\CF_{Bi}$ is the variational derivative along the infinitesimal rotation round $\nn_k$.
The divergence of the viscous stress $\sigma$ is defined by $(\nabla\cdot\sigma)_i=\partial_j\sigma_{ij}$.
To express the stress, let us denote by $\A$ and $\BOm$ the
symmetric and skew-symmetric parts of the velocity gradient $\kappa_{ij}=\partial_jv_i$, respectively, i.e.,
\begin{align*}
\A=\frac{1}{2}(\kappa+\kappa^T),\quad\BOm=\frac{1}{2}(\kappa-\kappa^T).
\end{align*}
The viscous stress $\sigma$
is given by
\begin{align}
\sigma(\Fp,\vv)=&\,\beta_1(\A\cdot\sss_1)\sss_1+\beta_0(\A\cdot\sss_2)\sss_1+\beta_0(\A\cdot\sss_1)\sss_2+\beta_2(\A\cdot\sss_2)\sss_2\nonumber\\
&\,+\beta_3(\A\cdot\sss_3)\sss_3-\eta_3\Big(\dot{\nn}_1\cdot\nn_2-\frac{1}{2}\BOm\cdot\aaa_1\Big)\sss_3\nonumber\\
&\,+\beta_4(\A\cdot\sss_4)\sss_4-\eta_2\Big(\dot{\nn}_3\cdot\nn_1-\frac{1}{2}\BOm\cdot\aaa_2\Big)\sss_4\nonumber\\
&\,+\beta_5(\A\cdot\sss_5)\sss_5-\eta_1\Big(\dot{\nn}_2\cdot\nn_3-\frac{1}{2}\BOm\cdot\aaa_3\Big)\sss_5\nonumber\\
&\,+\frac{1}{2}\eta_3(\A\cdot\sss_3)\aaa_1-\frac{1}{2}\chi_3\Big(\dot{\nn}_1\cdot\nn_2-\frac{1}{2}\BOm\cdot\aaa_1\Big)\aaa_1\nonumber\\
&\,+\frac{1}{2}\eta_2(\A\cdot\sss_4)\aaa_2-\frac{1}{2}\chi_2\Big(\dot{\nn}_3\cdot\nn_1-\frac{1}{2}\BOm\cdot\aaa_2\Big)\aaa_2\nonumber\\
&\,+\frac{1}{2}\eta_1(\A\cdot\sss_5)\aaa_3-\frac{1}{2}\chi_1\Big(\dot{\nn}_2\cdot\nn_3-\frac{1}{2}\BOm\cdot\aaa_3\Big)\aaa_3,
\end{align}
where the viscous coefficients  satisfy the following nonnegative definiteness conditions:
\begin{align}\label{coefficient-conditions}
\left\{
    \begin{aligned}
&\beta_i\geq0,~i=1,\cdots,5,\quad \chi_j>0,~j=1,2,3,\quad \eta>0,\\
&\beta^2_0\leq\beta_1\beta_2,~~\eta^2_1\leq\beta_5\chi_1,~~\eta^2_2\leq\beta_4\chi_2,~~\eta^2_3\leq\beta_3\chi_3.
\end{aligned}
  \right.
\end{align}
Again, such conditions are indeed met when the coefficients are derived from the tensor hydrodynamics as functions of molecular parameters (see {\rm \cite{LX}} for details).
The external force $\mathfrak{F}$ is defined by
\begin{align}\label{external-force-F}
\mathfrak{F}_i=\partial_i\nn_1\cdot\nn_2\ML_3\CF_{Bi}+\partial_i\nn_3\cdot\nn_1\ML_2\CF_{Bi}+\partial_i\nn_2\cdot\nn_3\ML_1\CF_{Bi}.
\end{align}
It can also be regarded as a elastic stress (see Eq. (1.17) in \cite{LWX2}).

The relations between coefficients (\ref{coefficient-conditions}) guarantee that the frame hydrodynamics (\ref{frame-equation-n1})--(\ref{yuan-incompressible-v}) fulfils the following energy dissipation law \cite{LWX}:
\begin{align}\label{energy-law}
&\frac{\ud}{\ud t}\Big(\frac{1}{2}\int_{\mathbb{R}^3}|\vv|^2\ud\xx+\CF_{Bi}[\Fp]\Big)
=-\eta\|\nabla\vv\|^2_{L^2}-\sum^3_{k=1}\frac{1}{\chi_k}\|\ML_k\CF_{Bi}\|^2_{L^2}\nonumber\\
&\quad-\bigg(\beta_1\|\A\cdot\sss_1\|^2_{L^2}+2\beta_0\int_{\mathbb{R}^3}(\A\cdot\sss_1)(\A\cdot\sss_2)\ud\xx+\beta_2\|\A\cdot\sss_2\|^2_{L^2}\bigg)\nonumber\\
&\quad-\Big(\beta_3-\frac{\eta^2_3}{\chi_3}\Big)\|\A\cdot\sss_3\|^2_{L^2}
-\Big(\beta_4-\frac{\eta^2_2}{\chi_2}\Big)\|\A\cdot\sss_4\|^2_{L^2}
-\Big(\beta_5-\frac{\eta^2_1}{\chi_1}\Big)\|\A\cdot\sss_5\|^2_{L^2}.
\end{align}
Concerning the well-posedness results of the system (\ref{frame-equation-n1})--(\ref{yuan-incompressible-v}), we refer to \cite{LWX} for details.

\subsection{Main results}

The first result is the local well-posedness of the tensor hydrodynamics. Let $\mathbb{Q}_{\delta}$ be defined in \eqref{Q-delta}, which is related to the physical range of two tensors.
\begin{theorem}\label{locall-posedness-theorem}
Let $s\geq 2$ be an integer. Assume that $\Fp^*=(\nn^*_1,\nn^*_2,\nn^*_3)\in SO(3)$ is a constant orthonormal frame, and that $\QQ^*=(Q^*_1,Q^*_2)^T$ takes the biaxial minimizer of the bulk energy $F_b$ in the form
\begin{align*}
 Q^*_i=s_i\Big(\nn^{*2}_1-\frac{1}{3}\Fi\Big)+b_i(\nn^{*2}_2-\nn^{*2}_3),\quad i=1,2.
\end{align*}
If the initial data satisfies
\begin{align*}
\QQ_I(\xx)-\QQ^*\in H^{s+1}(\mathbb{R}^3),\quad\vv_I(\xx)\in H^s(\mathbb{R}^3),
\end{align*}
with $\QQ_I(\xx)\in\mathbb{Q}_{\delta}$ for all $\xx\in\mathbb{R}^3$, then there exists $T>0$ and a unique solution $(\QQ,\vv)$ to the two-tensor system \eqref{Re-MB-Q-tensor-1}--\eqref{Re-MB-Q-tensor-3} on $[0,T]$, such that $\QQ(\xx,0)=\QQ_I(\xx),\, \vv(\xx,0)=\vv_I(\xx)$ and
\begin{align*}
&\QQ(\xx,t)-\QQ^*\in C([0,T];H^{s+1}(\mathbb{R}^3)),\\
&\vv(\xx,t)\in C([0,T];H^{s}(\mathbb{R}^3))\cap L^2(0,T;H^{s+1}(\mathbb{R}^3)),
\end{align*}
with $\QQ(\xx,t)\in\mathbb{Q}_{\delta/2}$.
\end{theorem}
This result mainly depends on the estimates involving fourth-order tensors given by closure approximation.
Once the basic estimates are established (see Section \ref{key-section}), the proof is quite standard, which we present in Section \ref{loc-wellp}.

We now state the main result.
To emphasize the dependence of the solution on $\ve$, we use the notations such as $\QQ^{\ve}$, $\vv^{\ve}$ and $\CM_{\QQ^{\ve}}$.
Let $\CH_{\QQ}$ be the Hessian of the bulk energy $F_b$ at $\QQ$, and the projections $\MP^{\rm in}$ and $\MP^{\rm out}$ be defined in \eqref{kerproj} concerning the kernel of $\CH_{\QQ}$.

\begin{theorem}\label{biaixal-limit-theorem}
Assume that $\QQ^{(0)}(\xx,t)=(Q^{(0)}_1,Q^{(0)}_2)^T$ with $Q^{(0)}_i=s_i\big(\nn^2_1-\Fi/3\big)+b_i\big(\nn^2_2-\nn^2_3\big) (i=1,2)$ is the biaxial minimizer of the bulk energy $F_b(\QQ)$.
Let $(\Fp(\xx,t), \vv(\xx,t))$ be a smooth solution to the frame hydrodynamics \eqref{frame-equation-n1}--\eqref{yuan-incompressible-v} on $[0,T]$, 
which satisfies
\begin{align*}
 (\nabla\Fp,\vv)\in C([0,T];H^{\ell}), \quad \textrm{for~the~integer}~ \ell\ge 20.
\end{align*}
Suppose that the initial data $(\QQ^{\ve}_I, \vv^\ve_I)$ takes the form
\begin{align*}
\QQ_I^{\ve}(\xx)=\sum^3_{k=0}\ve^k\QQ^{(k)}(\xx,0)+\ve^3\QQ^{\ve}_{I,R}(\xx),\quad
\vv_I^\ve(\xx)=\sum^2_{k=0}\ve^k\vv^{(k)}(\xx,0)+\ve^3\vv^{\ve}_{I,R}(\xx),
\end{align*}
where the functions $\big(\QQ^{(1)},\QQ^{(2)},\QQ^{(3)}, \vv^{(1)},\vv^{(2)}\big)$ are determined by Proposition \ref{existence-Hilbert-expansion-prop}, and $(\QQ_{I,R}^\ve, \vv_{I,R}^\ve)$ fulfils
\begin{align*}
\|\vv_{I,R}^\ve\|_{H^2}+\|\QQ_{I,R}^\ve\|_{H^3}+\ve^{-1}\|\MP^{\rm out}(\QQ^\ve_{I,R})\|_{L^2}\le E_0.
\end{align*}
Then, there exists $\ve_0>0$ and $E_1>0$ such that for all $\ve<\ve_0$, the system \eqref{Re-MB-Q-tensor-1}--\eqref{Re-MB-Q-tensor-3} has a unique solution
$(\QQ^\ve(\xx,t), \vv^\ve(\xx,t))$ on $[0,T]$ that possesses the following Hilbert expansion:
\begin{align*}
\QQ^\ve(\xx,t)=\sum^3_{k=0}\ve^k\QQ^{(k)}(\xx,t)+\ve^3\QQ_R(\xx,t),\quad
\vv^\ve(\xx,t)=\sum^2_{k=0}\ve^k\vv^{(k)}(\xx,t)+\ve^3\vv_R(\xx,t),
\end{align*}
where, for any $t\in[0,T]$, the remainder $(\QQ_R,\vv_R)$ satisfies
\begin{align*}
\mathfrak{E}\big(\QQ_R(t),\vv_R(t)\big)\leq E_1.
\end{align*}
Here, $\mathfrak{E}\big(\QQ,\vv\big)$ is defined by
\begin{align*}
&\mathfrak{E}\big(\QQ,\vv\big)\eqdefa\frac{1}{2}\int_{\mathbb{R}^3}\Big[\Big(|\vv|^2+(\CM^{-1}_{\QQ^{(0)}}\QQ)\cdot\QQ+\frac{1}{\ve}\big(\CH^{\ve}_{\QQ^{(0)}}\QQ\big)\cdot\QQ\Big)\\
&\qquad\qquad\quad+\ve^2\Big(|\nabla\vv|^2+\frac{1}{\ve}\big(\CH^{\ve}_{\QQ^{(0)}}\partial_i\QQ\big)\cdot\partial_i\QQ\Big)+\ve^4\Big(|\Delta\vv|^2+\frac{1}{\ve}\big(\CH^{\ve}_{\QQ^{(0)}}\Delta\QQ\big)\cdot\Delta\QQ\Big)\Big]\ud\xx,
\end{align*}
where $\CH^{\ve}_{\QQ^{(0)}}\QQ=\CH_{\QQ^{(0)}}\QQ+\ve\CG(\QQ)$, and the constant $E_1$ is independent of $\ve$.
\end{theorem}

To illustrate the idea of the proof Theorem \ref{biaixal-limit-theorem}, we give a short overview. To begin with, we make the Hilbert expansion of the solutions $(\QQ^{\ve},\vv^{\ve})$ with respect to the small parameter $\ve$:
\begin{align*}
\QQ^{\ve}(\xx,t)=&\QQ^{(0)}(\xx,t)+\ve\QQ^{(1)}(\xx,t)+\ve^2\QQ^{(2)}(\xx,t)+\ve^3\QQ^{(3)}(\xx,t)+\ve^3\QQ_R(\xx,t),\\
\vv^{\ve}(\xx,t)=&\vv^{(0)}(\xx,t)+\ve\vv^{(1)}(\xx,t)+\ve^2\vv^{(2)}(\xx,t)+\ve^3\vv_R(\xx,t).
\end{align*}
Substituting the above expansions into the system \eqref{Re-MB-Q-tensor-1}--\eqref{Re-MB-Q-tensor-3} and collecting the terms with the same order of $\ve$, we obtain a series of equations for $(\QQ^{(k)},\vv^{(k)};\QQ^{(3)})(0\leq k\leq2)$ (see Subsection \ref{Hilbert-subsection}).
The $O(\ve^{-1})$ equation requires that $\CJ(\QQ^{(0)})=0$, indicating that $\QQ^{(0)}$ is the stationary point of the bulk energy $F_b(\QQ)$, which is assumed to be the biaxial minimizer.
The $O(1)$ system (\ref{O-0-epsion-system-1})--(\ref{O-0-epsion-system-3}) gives the biaxial frame hydrodynamics \eqref{frame-equation-n1}--\eqref{yuan-incompressible-v}.

To show Theorem \ref{biaixal-limit-theorem}, we need to take care of two points below.
\begin{itemize}
\item The existence of smooth solutions to the equations of $(\QQ^{(k)},\vv^{(k)};\QQ^{(3)})(0\leq k\leq2)$.

\item Uniform boundedness of the remainder.
\end{itemize}

The first point relies on the local existence of smooth solutions to the frame hydrodynamics \eqref{frame-equation-n1}--\eqref{yuan-incompressible-v} on $[0,T]$, which has been established in a recent work \cite{LWX}.
It then requires to solve $\QQ^{(1)}$.
From the work for rod-like molecules \cite{WZZ3,LWZ,LW}, the basic approach is to decompose $\QQ^{(1)}$ into two parts: one in the kernel space ${\rm Ker}\CH_{\QQ^{(0)}}$ and the other in its orthogonal complement.
The latter can be obtained from the equation of $\QQ^{(0)}$, so that we need to derive an equation for the former from the $O(\ve)$ system.
The $O(\ve)$ system has the following abstract form:
\begin{align*}
\frac{\partial\QQ^{(1)}}{\partial t}+\vv^{(0)}\cdot\nabla\QQ^{(1)}=&-\CM^{(0)}\big(\CH_{\QQ^{(0)}}\QQ^{(2)}+\CG(\QQ^{(1)})+\JJ_1\big)+\cdots,\\
\frac{\partial\vv^{(1)}}{\partial t}+\vv^{(0)}\cdot\nabla\vv^{(0)}=&\nabla\cdot\CN^{(0)}\big(\CH_{\QQ^{(0)}}\QQ^{(2)}+\CG(\QQ^{(1)})+\JJ_1\big)+\cdots,
\end{align*}
where the 
the operator $\CM^{(0)}$ is a short notations for $\CM_{\QQ^{(0)}}$, and $\JJ_1$ is a nonlinear term.
%
%
%
Here comes a major difficulty: we can not directly project the nonlinear system of $(\QQ^{(1)},\vv^{(1)})$ into the kernel space ${\rm Ker}\CH_{\QQ^{(0)}}$, since $\MP^{\rm in}$ and $\CM^{(0)}$ are noncommutative, i.e., $\MP^{\rm in}\CM^{(0)}\neq\CM^{(0)}\MP^{\rm in}$.
However, by carefully investigating the intrinsic structure of the system, we could eventually derive a linear system (\ref{QQ1-top-equation-1})--(\ref{vv1-incomp-equation-3}) for the ${\rm Ker}\CH_{\QQ^{(0)}}$-projection of $\QQ^{(1)}$,
which guarantees the existing time to be $[0,T]$.
It requires a clear characterization of the space ${\rm Ker}\CH_{\QQ^{(0)}}$, which will be discussed in Section \ref{key-section}.
The $O(\ve^2)$ system for $(\QQ^{(2)},\vv^{(2)};\QQ^{(3)})$ is solved similarly.


The system for $(\QQ_R,\vv_R)$ takes the form
\begin{align*}
\partial_t\QQ_R=&-\CM_{\QQ^{(0)}}\Big(\frac{1}{\ve}\CH_{\QQ^{(0)}}\QQ_R+\CG(\QQ_R)\Big)+\CV_{\QQ^{(0)}}\kappa_R+\cdots,\\
\partial_t\vv_R=&-\nabla p_R+\nabla\cdot\Big(\CN_{\QQ^{(0)}}\Big(\frac{1}{\ve}\CH_{\QQ^{(0)}}\QQ_R+\CG(\QQ_R)\Big)\Big)+\cdots.
\end{align*}
The main obstacle towards the uniform estimate comes from the singular term $\frac{1}{\ve}\CH_{\QQ^{(0)}}\QQ_R$.
We introduce a suitable energy functional $\mathfrak{E}(t)$ in (\ref{Ef-t-functional}) to handle it. Since the operator $\CH_{\QQ^{(0)}}$ depends on the time $t$, we have to control the singular term $\frac{1}{\ve}\big\langle\QQ_R,\partial_t(\CH_{\QQ^{(0)}})\QQ_R\big\rangle$.
This takes tedious calculations for rod-like molecules \cite{WZZ3}.
However, we point out in this work that it can be done by the eigen-decomposition of the Hessian (see \eqref{eigen-decomp}) straightforwardly.

The proof of Theorem \ref{biaixal-limit-theorem} is presented in Section \ref{bipr}.
Throughout the proof, the estimates on fourth-order tensors still play a key role.

The bulk minimizer of the two-tensor model might also be uniaxial, for which we briefly discuss in Section \ref{genappr}.
As we have discussed in \cite{LX}, the limit model is the Ericksen--Leslie model.
The whole procedure is the same as how we deal with the biaxial nematic phase, only to notice a slight difference in the structure of $\CH_{\QQ^{(0)}}$.
Despite the limit model is the same as that of one-tensor models, it shall be clear that the derivation in the previous works are special and not suitable for the uniaxial limit of the two-tensor model.

\section{Entropy, stationary points and closure approximation} \label{key-section}
In this section, we introduce the entropy term and some crucial results relevant to it, as well as how to use the entropy to define the closure approximation.

We consider two types of the entropy terms, which we call the original entropy and the quasi-entropy.
The original entropy is defined implicitly by minimizing $\int_{SO(3)}\rho\ln\rho\ud\Fq$ with $\QQ$ fixed, which is a standard approach that has been utilized in different cases \cite{KKLS,BM,HLWZZ,XYZ,Taylor,Xu1}.
The quasi-entropy refers to a class of elementary functions as a substitution of the original entropy \cite{Xu3}.
The quasi-entropy significantly simplifies the hydrodynamics,  especially in the view of numerical simulations.
Yet it has been proved that the original entropy and the quasi-entropy share several essential properties and display very close results on homogeneous phase transitions for various molecules \cite{Xu3}.
Therefore, it is of great worthiness to discuss both of them.
In the formal derivation of the biaxial hydrodynamics \cite{LX}, it has been seen that the original entropy and the quasi-entropy lead to the models of the identical form, with differences only in the specific values of coefficients.

In what follows, apart from specifying their definitions and the corresponding closure approximations, we shall discuss the stationary points of the bulk energy and the structure of the Hessian at these points.
In addition, we write down some basic estimates to be utilized in the establishment of the rigorous biaxial limit.
As it turns out, on the above aspects, the original entropy and the quasi-entropy lead to similar results, which will be convenient for the derivations in the forthcoming sections.

\subsection{Original entropy and quasi-entropy}
Let us begin with defining the original entropy for $\QQ$ (for general cases, see Section 5 of \cite{Xu1}).
On minimizing $\int_{SO(3)}\rho\ln\rho\ud\Fq$, we obtain the maximum entropy state
\begin{align}
    &\,\rho(\Fq)=\frac 1Z \exp(B_1\cdot \mm_1^2+B_2\cdot\mm_2^2), \label{maxentstate}
\end{align}
where $Z$ is the normalizing constant, and two second-order symmetric traceless tensors $B_1$ and $B_2$ are Lagrange multipliers for the constraints on $\QQ$,
\begin{align}
    &\,Z=\int_{SO(3)}\exp(B_1\cdot \mm_1^2+B_2\cdot\mm_2^2)\ud\Fq, \nonumber\\
    &\,Q_i=\frac 1 Z \int_{SO(3)}\left(\mm_i^2-\frac 13 \Fi\right)\exp(B_1\cdot \mm_1^2+B_2\cdot\mm_2^2)\ud\Fq. \label{maxent-cnstrt}
\end{align}
The original entropy $F_{\rm orig}$ is obtained by taking \eqref{maxentstate} into $\int_{SO(3)} \rho\ln\rho\ud\Fp$,
\begin{align}
    &\,F_{\rm orig}=B_1\cdot Q_1 +B_2\cdot Q_2-\ln Z. \label{orig-ent}
\end{align}

\begin{proposition}
  The original entropy has the following properties:
  \begin{itemize}
  \item $\BB$ is uniquely determined by $\QQ$, where $\BB=(B_1,B_2)^T$.
  \item $\partial F_{\rm orig}/\partial \QQ=\BB$.
  \item $\partial\QQ/\partial\BB$ is positive definite.
  \end{itemize}
\end{proposition}
\begin{proof}
The uniqueness of $\BB$ is shown in Theorem 5.1 in \cite{Xu1}.

For the second property, from (\ref{orig-ent}) we have
\begin{align*}
  \frac{\partial F_{\rm orig}}{\partial (Q_1)_{ij}}=&\,(B_1)_{ij}+\frac{\partial (B_1)_{kl}}{\partial (Q_1)_{ij}}(Q_1)_{kl}+\frac{\partial (B_2)_{kl}}{\partial (Q_1)_{ij}}(Q_2)_{kl}-\frac{\partial \ln Z}{\partial (Q_1)_{ij}}\\
  =&\,(B_1)_{ij}+\frac{\partial (B_1)_{kl}}{\partial (Q_1)_{ij}}\left(\frac{\partial \ln Z}{\partial (B_1)_{kl}}-\frac{\delta_{kl}}{3}\right)+\frac{\partial (B_2)_{kl}}{\partial (Q_1)_{ij}}\left(\frac{\partial \ln Z}{\partial (B_2)_{kl}}-\frac{\delta_{kl}}{3}\right)-\frac{\partial \ln Z}{\partial (Q_1)_{ij}}\\
  =&\,(B_1)_{ij}.
\end{align*}
Similar arguments hold for $\partial F_{\rm orig}/\partial Q_2$.

For the third property, direct calculation gives
\begin{align*}
  \frac{\partial\QQ}{\partial\BB}=&\,\left(
  \begin{array}{cc}
    \langle(\mm_1^2-\Fi/3)\otimes(\mm_1^2-\Fi/3)\rangle & \langle(\mm_1^2-\Fi/3)\otimes(\mm_2^2-\Fi/3)\rangle\\
    \langle(\mm_2^2-\Fi/3)\otimes(\mm_1^2-\Fi/3)\rangle & \langle(\mm_2^2-\Fi/3)\otimes(\mm_2^2-\Fi/3)\rangle
  \end{array}
  \right)\\
  &\,-\Big(\langle \mm_1^2-\Fi/3 \rangle,\langle \mm_2^2-\Fi/3 \rangle\Big)^T\Big(\langle \mm_1^2-\Fi/3 \rangle,\langle \mm_2^2-\Fi/3 \rangle\Big),
\end{align*}
which is a covariance matrix.
\end{proof}

The quasi-entropy is defined by the log-determinant of a covariance matrix.
It depends on the truncation order of the tensors, which we need to specify (see \cite{Xu3} for details).
For $\QQ$, we truncate at second order to arrive at
\begin{align}
  \Xi_2\big(\QQ\big)=-\ln\det \Big(Q_1+\frac{\Fi}{3}\Big)-\ln\det\Big( Q_2+\frac{\Fi}{3}\Big)-\ln\det\Big(\frac{\Fi}{3}-Q_1-Q_2\Big),\label{qent-2nd}
\end{align}
where $\Fi$ denotes the second-order identity tensor ($3\times 3$ matrix).
The domain of $\Xi_2$ is restrained in those $\QQ$ such that the three matrices are positive definite:
\begin{align}
  {\rm dom}(\Xi_2)=\left\{\QQ:Q_1+\frac{\Fi}{3},~ Q_2+\frac{\Fi}{3},~\frac{\Fi}{3}-Q_1-Q_2~\text{are positive definite}\right\}. \nonumber
\end{align}

To carry out closure approximation by the quasi-entropy, we need the one truncated at fourth order.
For the high-order tensors appearing in the tensor hydrodynamics, they can be expressed linearly by $Q_1$, $Q_2$ and several symmetric traceless tensors of third- and fourth-order.
The relevant expressions can be found in \cite{Xu1,Xu3,LX}. Here, since these expressions are unused, we choose not to write them down explicitly.
Instead, we denote in short these third- and fourth-order symmetric traceless tensors as $\HH$.


Now, let us write down the quasi-entropy truncated at fourth order.
We introduce a few short notations:
\begin{itemize}
\item For a second-order tensor $U$, we define a $1\times 5$ row vector as
\begin{align*}
    \Phi_2(U)_j
    =(U\cdot\sss_j).
\end{align*}
\item For a third-order tensor $U$, we define a $3\times 5$ matrix,
\begin{align*}
  \Psi_3(U)_{ij}
  =(U\cdot \nn_i\otimes\sss_j).
\end{align*}
\item For a fourth order tensor $U$, we define a $5\times 5$ matrix,
\begin{align*}
  \Psi_4(U)_{ij}
  =(U\cdot \sss_i\otimes\sss_j).
\end{align*}
\end{itemize}
The quasi-entropy at fourth-order is given by
\begin{align}
    &\,\Xi_4(\QQ,\HH)=\nonumber\\
    &\,-\ln\det\left(
    \begin{array}{ccc}
        1 &\, \Phi_2(\langle\mm_1^2-\frac\Fi 3\rangle) &\, \Phi_2(\langle\mm_2^2-\mm_3^2\rangle)
        \vspace{1ex}\\
        \Phi_2(\langle\mm_1^2-\frac\Fi 3\rangle)^T &\,
        \Psi_4\big(\langle(\mm_1^2-\frac\Fi 3)\otimes(\mm_1^2-\frac\Fi 3)\rangle\big) &\,  \Psi_4\big(\langle(\mm_2^2-\mm_3^2)\otimes(\mm_1^2-\frac{\Fi}{3})\rangle\big)
        \vspace{1ex}\\
        \Phi_2(\langle\mm_2^2-\mm_3^2\rangle)^T &\,
        \Psi_4\big(\langle(\mm_2^2-\mm_3^2)\otimes(\mm_1^2-\frac\Fi 3)\rangle\big)^T &\,
        \Psi_4\big(\langle(\mm_2^2-\mm_3^2)\otimes(\mm_2^2-\mm_3^2)\rangle\big)
    \end{array}
    \right)\nonumber\\
    &\,-\ln\det\left(
    \begin{array}{cc}
        \Psi_2(\langle\mm_1^2\rangle) &\, \Psi_3(\langle\mm_1\otimes\mm_2\mm_3\rangle)
        \vspace{1ex}\\
        \Psi_3(\langle\mm_1\otimes\mm_2\mm_3\rangle)^T &\, \Psi_4(\langle\mm_2\mm_3\otimes\mm_2\mm_3\rangle)
    \end{array}
    \right)\nonumber\\
    &\,-\ln\det\left(
    \begin{array}{cc}
        \Psi_2(\langle\mm_2^2\rangle) &\, \Psi_3(\langle\mm_2\otimes\mm_1\mm_3\rangle)
        \vspace{1ex}\\
        \Psi_3(\langle\mm_2\otimes\mm_1\mm_3\rangle)^T &\, \Psi_4(\langle\mm_1\mm_3\otimes\mm_1\mm_3\rangle)
    \end{array}
    \right)\nonumber\\
    &\,-\ln\det\left(
    \begin{array}{cc}
        \Psi_2(\langle\mm_3^2\rangle) &\, \Psi_3(\langle\mm_3\otimes\mm_1\mm_2\rangle)
        \vspace{1ex}\\
        \Psi_3(\langle\mm_3\otimes\mm_1\mm_2\rangle)^T &\, \Psi_4(\langle\mm_1\mm_2\otimes\mm_1\mm_2\rangle)
    \end{array}
    \right).\label{qent-4th}
\end{align}
The domain of $\Xi_4$ is given by those $(\QQ,\HH)$ such that the matrices above are all positive definite.

\begin{proposition}\label{cvx}
The domains of $F_{\rm orig}$, $\Xi_2$, $\Xi_4$ are bounded, convex open sets.
Each of the three functions is strictly convex on the corresponding domain.
\end{proposition}
\begin{proof}
They are special cases of Section 3 and 4 in \cite{Xu3}.
\end{proof}
Furthermore, if the tensors in $\Xi_2$ (resp. $\Xi_4$) are given by some density in \eqref{maxentstate}, they must lie within the domain of $\Xi_2$ (resp. $\Xi_4$), since the covariance matrix of a density function is positive definite.
A short remark is given here that it is still open whether the inverse of the above claim holds.
However, in the current work we could keep our discussion within the domain of $F_{\rm orig}$, which we will explain later in this section.

\begin{proposition}\label{entropy-smooth-prop}
The functions $F_{\rm orig}$ and $\Xi_2$ are $C^{\infty}$ with respect to $\QQ$ in the corresponding domain. The function $\Xi_4$ is $C^{\infty}$ with respect to $(\QQ,\HH)$ in its domain.
\end{proposition}
\begin{proof}
  For $\Xi_2$ and $\Xi_4$, the result is obvious because they are elementary functions.
  For the original entropy $F_{\rm orig}$, notice that $\partial F_{\rm orig}/\partial\QQ=\BB$ and that $\partial\QQ/\partial\BB$ is positive definite.
  It is easy to verify by direct calculation that $\QQ$ is smooth with respect to $\BB$, so that $\BB$ is smooth with respect to $\QQ$, which already gives the smoothness of $F_{\rm orig}$.
\end{proof}

\subsection{Stationary points and the linearized operator}
Let us look into the bulk energy \eqref{free-energy-bulk}.
Now, we can specify the entropy term $F_{\rm entropy}$.
It may take the original entropy $F_{\rm orig}$ or the quasi-entropy $\nu\Xi_2$.
The free parameter $\nu$ is introduced to attain a match between the original entropy and the quasi-entropy.
We choose $\nu=5/9$ here, which is proposed in Section 6.1.2 in \cite{Xu3}.

To study the limit $\ve\to 0$, we need to characterize the minimizer of the bulk energy. The up-to-date theoretical result \cite{XZ2,Xu3} is given below.
\begin{proposition}\label{bulk-minimum}
Assume that the matrix
$\displaystyle \left(\begin{array}{cc}
    c_{02} &\, c_{04} \\
    c_{04} &\, c_{03}
\end{array}\right)
$
is not negative definite, or is negative but $c_{04}^2/c_{03}-c_{02}\le 2$.
For the bulk energy $F_b$ given in \eqref{free-energy-bulk}, where the entropy term $F_{\rm entropy}$ takes either $F_{\rm orig}$ or $\nu\Xi_2$ with $\nu=5/9$, at the stationary points $Q_1$ and $Q_2$ have a shared eigenframe.
\end{proposition}

In the cases where the eigenframe of $Q_1$ and $Q_2$ are identical, they can be written in the form
\begin{align}\label{Qi-biaxial}
Q_i=s_i\Big(\nn^2_1-\frac{\Fi}{3}\Big)+b_i(\nn^2_2-\nn^2_3),\quad i=1,2.
\end{align}
Numerical studies \cite{S-J-P,LGRA,XZ1,XYZ,Xu3} indicate that even if the conditions of Proposition \ref{bulk-minimum} do not hold, at each local energy minimizer (i.e. not saddle points) $Q_1$ and $Q_2$ still have an identical eigenframe.
Furthermore, depending on the coefficients, the global energy minimum could be either uniaxial (where $b_i=0$) or biaxial (where at least one $b_i\ne 0$).
To fix the idea, we assume that under certain coefficients $c_{02}$, $c_{03}$, $c_{04}$, we have a biaxial global minimum $\QQ^{(0)}=(Q_1^{(0)},Q_2^{(0)})^T$ of the form \eqref{Qi-biaxial}.

With the form \eqref{Qi-biaxial}, the scalars $s_i$ and $b_i$ shall satisfy
\begin{align}
    \frac 23 s_i+\frac 13>0,\quad \frac 13-\frac 13 s_i\pm b_i>0,\qquad i=1,2,3, \label{range-2nd}
\end{align}
where we define $s_3=-s_1-s_2$ and $b_3=-b_1-b_2$.
This requirement originates from the domain of the entropy term, which has been discussed previously \cite{XYZ}.

It is significant to notice that the bulk energy is rotational invariant.
Thus, when changing the frame $\Fp=(\nn_1,\nn_2,\nn_3)$ in the biaxial minimizer $\QQ^{(0)}$, it still gives a minimizer.
Therefore, the biaxial equilibrium state is not a single point, but a three-dimensional manifold.
We would like to write this manifold as $\QQ^{(0)}(\Fp)$.
The tangential space of this manifold at certin $\Fp$ is spanned by $\bxi_k\triangleq \ML_k\QQ^{(0)}(\Fp),\,k=1,2,3$.
The three $\bxi_k$ lie within the kernel of the Hessian \begin{align}\label{Hessian}
    \CH_{\QQ^{(0)}}=\partial_{\QQ}^2F_b|_{\QQ^{(0)}}.
\end{align}
This can be recognized by $\CJ(\QQ^{(0)}(\Fp))=0$, because $\QQ^{(0)}(\Fp)$ is an minimizer of $F_b$. Acting the operator $\ML_k$ on it, we obtain
\begin{align}
  0=\ML_k\CJ(\QQ^{(0)}(\Fp))=\CH_{\QQ^{(0)}}\bxi_k. \label{xiker}
\end{align}

On the other hand, since $\QQ^{(0)}$ is a minimizer, the eigenvalues of the Hessian $\CH_{\QQ^{(0)}}$ are nonnegative. Therefore, we have the following results.
\begin{proposition}\label{prop-linearized-operator}
  For a given stationary point $\QQ^{(0)}=\QQ^{(0)}(\Fp)$, the linearized operator $\CH_{\QQ^{(0)}}$ satisfies the following properties:
  \begin{itemize}
  \item It holds $\CH_{\QQ^{(0)}}\QQ\in ({\rm Ker}\CH_{\QQ^{(0)}})^{\perp}$;
  \item There exists a constant $C_0>0$ such that for any $\QQ\in({\rm Ker}\CH_{\QQ^{(0)}})^{\perp}$,
    \begin{align*}
      (\CH_{\QQ^{(0)}}\QQ)\cdot\QQ\geq C_0|\QQ|^2.
    \end{align*}
  \item $\CH_{\QQ^{(0)}}$ is a one to one map on $({\rm Ker}\CH_{\QQ^{(0)}})^{\perp}$ and its inverse $\CH^{-1}_{\QQ^{(0)}}$ exists.
  \end{itemize}
\end{proposition}
\begin{proof}
  We may choose a basis of $({\rm Ker}\CH_{\QQ^{(0)}})^{\perp}$ as $\ee_1,\cdots,\ee_l$. The operator $\CH_{\QQ^{(0)}}$ can be written as
  \begin{equation}
    \CH_{\QQ^{(0)}}=\sum_{j=1}^l\lambda_j\ee_j\otimes \ee_j,\quad \lambda_j>0. \label{eigen-decomp}
  \end{equation}
  The three statements can be deduced easily using the above representation.
\end{proof}

To fully characterize the Hessian, we adopt the assumption below.
\begin{assumption}\label{HQ-3}
  ${\rm Ker}\CH_{\QQ^{(0)}}={\rm span}\{\bxi_1,\bxi_2,\bxi_3\}$.
\end{assumption}
Actually, we have proved that $\bxi_k\in{\rm Ker}\CH_{\QQ^{(0)}}$.
The meaning of this assumption is that the three-dimensional biaxial minimizer manifold will not lose its stability when the coefficients in the bulk energy are perturbed.
This assumption certainly does not always hold, as it will be broken in the case of phase transitions.
In other words, by adopting this assumption we consider the coefficients far from the critical values that give rise to phase transitions.
In Appendix, we provide a simple numerical example as an evidence for this assumption to hold for biaxial minimizers.

Under Assumption \ref{HQ-3}, we have $l=\dim({\rm Ker}\CH_{\QQ^{(0)}})^{\perp}=\dim\mathbb{Q}-\dim({\rm Ker}\CH_{\QQ^{(0)}})=7$ in \eqref{eigen-decomp}.
Another thing to be noticed is that $\bxi_k$ and $\ee_k$ depend on the frame $\Fp$. Later, we will need their derivatives.

Define
\begin{align}
    \MP^{\rm in}\QQ=\sum_{j=1}^3\QQ\cdot\bxi_j\frac{\bxi_j}{|\bxi_j|^2},\quad
    \MP^{\rm out}\QQ=\sum_{j=1}^7\QQ\cdot\ee_j\frac{\ee_j}{|\ee_j|^2},\label{kerproj}
\end{align}
which are the projections onto the space ${\rm Ker}\CH_{\QQ^{(0)}}$ and $({\rm Ker}\CH_{\QQ^{(0)}})^{\perp}$, respectively.

\subsection{Closure by entropy minimization}

Corresponding to the choice of the entropy term $F_{\rm entropy}$, the closure approximation can also be done in two ways.
To avoid ambiguity, we assume that $\QQ$ lies within ${\rm dom}F_{\rm orig}$. The closure approximation aims to calculate the high-order tensors $\HH$ from $\QQ$.

If the entropy term is given by $F_{\rm orig}$, one could use the maximum entropy state \eqref{maxentstate} to calculate $\HH$.

\begin{proposition}
If $\HH$ is calculated from \eqref{maxentstate}, it is $C^{\infty}$ with respect to $\QQ$.
\end{proposition}
\begin{proof}
  It can be verify easily that $\HH$ is $C^{\infty}$ with respect to $\BB$.
  Proposition \ref{entropy-smooth-prop} tells us $\BB$ is $C^{\infty}$ with respect to $\QQ$.
\end{proof}

The closure approximation by \eqref{maxentstate} can be equivalently formulated as a constrained minimization \cite{Xu3}.
Based on this formulation, one can naturally define the closure approximation by the quasi-entropy.
Specifically, when $F_{\rm entropy}$ takes $\nu\Xi_2$, we use $\Xi_4$ for closure approximation.
The high-order tensors are solved through
\begin{equation}
    \min \Xi_4(\QQ,\HH),\quad s.t.~ \QQ \text{ given}. \label{closure-qe}
\end{equation}
From Proposition \ref{cvx} and the discussion below it, the solution exists and is unique.
Because of strict convexity, it equivalently demands
\begin{equation}
    \partial_{\HH}\Xi_4=0.\label{closure-qe-d}
\end{equation}
%
%
%
%
%
The smoothness statement still holds.
\begin{proposition}
When $\HH$ is solved from \eqref{closure-qe-d}, it is $C^{\infty}$ with respect to $\QQ$.
\end{proposition}
\begin{proof}
Taking derivative with respect to $\QQ$ on \eqref{closure-qe}, we obtain
\begin{align*}
\frac{\partial\HH}{\partial\QQ}=-(\partial_{\HH}^2\Xi_4)^{-1}\partial_{\QQ}\partial_{\HH}\Xi_4.
\end{align*}
Since $\Xi_4$ is $C^{\infty}$ and $\partial_{\HH}^2\Xi_4$ is positive definite, the smoothness is immediately obtained.
\end{proof}

We define $\mathbb{Q}_{\delta}$ as all $\QQ$ in the domain of $F_{\rm orig}$ whose distance to the boundary is at least $\delta$.
\begin{align}
    \mathbb{Q}_{\delta}=\big\{\QQ\in\mathbb{Q}:{\rm d}(\QQ,\partial{\rm dom}F_{\rm orig})\ge \delta\big\}. \label{Q-delta}
\end{align}
It gives a bounded closed set, which is compact.
Therefore, for any $\delta$, uniform estimates hold for derivatives of $\HH$.
\begin{proposition}\label{high-tensor-smooth-prop}
For arbitrary small enough $\delta>0$ and nonnegative integer $k$, there exists $C_{\delta,k}>0$ such that $|\partial_{\QQ}^k\HH|\le C_{\delta,k}$ for $\QQ\in\mathbb{Q}_{\delta}$.
\end{proposition}
\begin{proof}
It is deduced immediately from smoothness and the fact that $\mathbb{Q}_{\delta}$ is compact.
\end{proof}

Below, we list some basic estimates derived from the properties we stated above. We denote the derivative of the entropy term with respect to $\QQ$ by $\CS(\QQ)=\partial_{\QQ}F_{\rm entropy}$, in which $F_{\rm entropy}$ can be taken as the original entropy or the quasi-entropy. Lemma \ref{SS-defference-lemma} and Lemma \ref{SS-Lip-lemma} are direct consequences of Proposition \ref{high-tensor-smooth-prop}, Lemma \ref{lem:composition} and Lemma \ref{lem:difference}.
\begin{lemma}\label{SS-defference-lemma}
For any $\delta>0, k\in\mathbb{N}$ and constant tensor $\QQ^*=(Q^*_1,Q^*_2)^T$, there exists a positive constant $C_{\delta}$ depending on $\delta$ such that if $\QQ(\xx)\in\mathbb{Q}_{\delta}$, then
\begin{align*}
\|\CS(\QQ)-\CS(\QQ^*)\|_{H^k}\leq C_{\delta}\|\QQ-\QQ^*\|_{H^k}.
\end{align*}
\end{lemma}

\begin{lemma}\label{SS-Lip-lemma}
For any $\delta>0$, there exists a positive constant $C_{\delta}$ depending on $\delta$ such that
\begin{align*}
|\CS(\QQ')-\CS(\QQ'')|\leq C_{\delta}|\QQ'-\QQ''|,
\end{align*}
where $\QQ'=(Q'_1,Q'_2)^T,\QQ''=(Q''_1,Q''_2)^T\in\mathbb{Q}_{\delta}$.
Consequently, it follows that
\begin{align*}
|\partial_i\CS(\QQ)|\leq C_{\delta}|\partial_i\QQ|.
\end{align*}
Further, for any $k\in\mathbb{N}$, there exists a constant $C=C(\delta,\|\QQ'-\QQ^*\|_{H^k},\|\QQ''-\QQ^*\|_{H^k})$, such that
\begin{align*}
\|\CS(\QQ')-\CS(\QQ'')\|_{H^k}\leq C(\delta,\|\QQ'-\QQ^*\|_{H^k},\|\QQ''-\QQ^*\|_{H^k})\|\QQ'-\QQ''\|_{H^k}.
\end{align*}
\end{lemma}


Lemma \ref{commutor-YY-QQ-lemma} and Lemma \ref{defference-YY-UU-lemma} are direct corollaries of Lemma \ref{SS-Lip-lemma}, Lemma \ref{lem:product} and Lemma \ref{lem:commutator}.

\begin{lemma}\label{commutor-YY-QQ-lemma}
For any $\delta>0$, there exists a positive constant $C_{\delta}$ depending on $\delta$ such that if $\QQ(\xx)\in\mathbb{Q}_{\delta}$ and $\UU\in\mathbb{R}^{3\times3}$, then for any multiple index $\alpha$, it follows that
\begin{align*}
\|\partial^{\alpha}(\CY_{\QQ}\UU)-\CY_{\QQ}\partial^{\alpha}\UU\|_{L^2}\leq C_{\delta}(\|\nabla\QQ\|_{L^{\infty}}\|\UU\|_{H^{|\alpha|-1}}+\|\nabla\QQ\|_{H^{|\alpha|-1}}\|\UU\|_{L^{\infty}}).
\end{align*}
Furthermore, if $|\alpha|\geq2$, it holds
\begin{align*}
\|\partial^{\alpha}(\CY_{\QQ}\UU)-\CY_{\QQ}\partial^{\alpha}\UU\|_{L^2}\leq C_{\delta}\|\nabla\QQ\|_{H^{|\alpha|}}\|\UU\|_{H^{|\alpha|-1}}.
\end{align*}
Here, the operator $\CY_{\QQ}$ can be taken as $\CM_{\QQ},\CV_{\QQ},\CN_{\QQ}$ and $\CP_{\QQ}$, respectively.
\end{lemma}

\begin{lemma}\label{defference-YY-UU-lemma}
For any $\delta>0$ and $k\in\mathbb{N}$, there exists positive constants $C_1=C_1(\delta)$ and $C_2=C_2(\delta,\|\QQ'-\QQ^*\|_{H^k},\|\QQ''-\QQ^*\|_{H^k})$ such that
\begin{align*}
\|\CY_{\QQ'}\UU-\CY_{\QQ''}\UU\|_{H^k}\leq C_1\|\UU\|_{H^k}\|\QQ'-\QQ''\|_{L^{\infty}}+C_2\|\UU\|_{L^{\infty}}\|\QQ'-\QQ''\|_{H^k}.
\end{align*}
Furthermore, if $0\leq k\leq 2$, there exists a positive constant $C=C(\delta,\|\QQ'-\QQ^*\|_{H^2},\|\QQ''-\QQ^*\|_{H^2})$ such that
\begin{align*}
\|\CY_{\QQ'}\UU-\CY_{\QQ''}\UU\|_{H^k}\leq &C\|\UU\|_{H^2}\|\QQ'-\QQ''\|_{H^k},\\
\|\CY_{\QQ'}\UU-\CY_{\QQ''}\UU\|_{H^k}\leq &C\|\UU\|_{H^k}\|\QQ'-\QQ''\|_{H^2}.
\end{align*}
Here, the operator $\CY_{\QQ}$ can be taken as $\CM_{\QQ},\CV_{\QQ},\CN_{\QQ}$ and $\CP_{\QQ}$, respectively.
\end{lemma}


\section{Local well-posedness of smooth solutions}\label{loc-wellp}

This section is devoted to studying the local well-posedness of smooth solutions to the system (\ref{Re-MB-Q-tensor-1})--(\ref{Re-MB-Q-tensor-3}).
The major estimates come directly from the basic estimates in Section \ref{key-section}.

For the integer $s\geq2$, we define the space $\mathbb{X}$ as follows:
\begin{align*}
\mathbb{X}(\delta,T,C_0)\eqdefa&\Big\{(\QQ,\vv):\QQ\in\mathbb{Q}_{\delta/2}, ~\|\QQ-\QQ^*\|_{H^{s+1}}+\|\CG(\QQ)\|_{L^2_tH^s_x}\\
&\qquad\qquad\quad+\|\vv\|_{H^s}+\|\nabla\vv\|_{L^2_tH^s_x}\leq C_0,\quad\text{a.e.}~t\in[0,T]\Big\}.
\end{align*}
If the solution $(\QQ,\vv)\in\mathbb{X}$, then by the Sobolev imbedding theorem, it follows that
\begin{align*}
\|\QQ\|_{L^{\infty}}+\|\nabla\QQ\|_{L^{\infty}}+\|\vv\|_{L^{\infty}}\leq C(C_0).
\end{align*}
The proof of Theorem \ref{locall-posedness-theorem} is mainly based on the iterative argument and the closed energy estimate.

\subsection{Linearized system and iteration scheme}
In order to define a sequence $\{(\QQ^{(n)},\vv^{(n)})\}_{n\in\mathbb{N}}$ of approximate solutions to the system (\ref{Re-MB-Q-tensor-1})--(\ref{Re-MB-Q-tensor-3}), we follow an iterative scheme.
First, we set
\begin{align*}
\big(\QQ^{(0)}(\xx,t),\vv^{(0)}(\xx,t)\big)=(\QQ_I(\xx),\vv_I(\xx))\in\mathbb{X}(\delta,T,C_0).
\end{align*}
Assume that $(\QQ^{(n)},\vv^{(n)})\in\mathbb{X}(\delta,T,C_0)$.
We construct $(\QQ^{(n+1)},\vv^{(n+1)})$ by solving the following linearized system:
\begin{align}
&\frac{\partial\QQ^{n+1}}{\partial t}+\vv^{(n)}\cdot\nabla\QQ^{(n+1)}=-\CM_{\QQ^{(n)}}\Big(\frac{1}{\ve}\CJ(\QQ^{(n)})+\CG(\QQ^{(n+1)})\Big)+\CV_{\QQ^{(n)}}\kappa^{(n+1)},\label{approximate-eq-1}\\
&\Big(\frac{\partial\vv^{(n+1)}}{\partial t}+\vv^{(n)}\cdot\nabla\vv^{(n+1)}\Big)_i=-\partial_i p^{(n+1)}+\eta\Delta v^{(n+1)}_i+\partial_j\big(\CP_{\QQ^{(n)}}\kappa^{(n+1)}\big)_{ij}\nonumber\\
&\qquad\qquad+\partial_j\CN_{\QQ^{(n)}}\Big(\frac{1}{\ve}\CJ(\QQ^{(n)})+\CG(\QQ^{(n+1)})\Big)_{ij}+\Big(\frac{1}{\ve}\CJ(\QQ^{(n)})+\CG(\QQ^{(n+1)})\Big)\cdot\partial_i\QQ^{(n)},\label{approximate-eq-2}\\
&\nabla\cdot\vv^{(n+1)}=0,\label{approximate-eq-3}
\end{align}
where $\CJ(\QQ^{(n)})$ and $\CG(\QQ^{(n+1)})_{jk}$ are given by, respectively,
\begin{align*}
\CJ(\QQ^{(n)})=&\MS\CS(\QQ^{(n)})+D_0\QQ^{(n)},\\
\CG(\QQ^{(n+1)})_{jk}=&-D_1(\Delta\QQ^{(n+1)})_{jk}-D_2\MS(\partial_j\partial_i\QQ^{(n+1)}_{ik}).
\end{align*}
Furthermore, 
the initial data satisfies
\begin{align*}
\big(\QQ^{(n+1)}(\xx,t),\vv^{(n+1)}(\xx,t)\big)=(\QQ_I(\xx),\vv_I(\xx)).
\end{align*}
The existence of the solution $(\QQ^{(n+1)},\vv^{(n+1)})$ to the linearized system (\ref{approximate-eq-1})--(\ref{approximate-eq-3}) can be guaranteed by the classical parabolic theory (see \cite{BCD} for instance). The next task is to prove local a priori estimates, that is, for some suitably chosen $T>0$, it holds  $(\QQ^{(n+1)},\vv^{(n+1)})\in\mathbb{X}$.

For the integer $s\geq2$, we define the following two energy functionals:
\begin{align*}
E_s(\QQ,\vv)\eqdefa&~\frac{1}{2}\|\QQ-\QQ^*\|^2_{L^2}+\CF_e[\nabla\QQ]+\frac{1}{2}\|\vv\|^2_{L^2}+\CF_e[\nabla^{s+1}\QQ]+\frac{1}{2}\|\nabla^s\vv\|^2_{L^2},\\
F_s(\QQ,\vv)\eqdefa&~\|\CG(\QQ)\|^2_{L^2}+\|\nabla^s\CG(\QQ)\|^2_{L^2}+\|\nabla\vv\|^2_{L^2}+\|\nabla^{s+1}\vv\|^2_{L^2},
\end{align*}
where $
\CF_e[\nabla\QQ]=\int_{\mathbb{R}^3}F_e(\nabla\QQ)\ud\xx$.
It follows from Sobolev's interpolation theorem that
\begin{align*}
E_s\sim\|\QQ-\QQ^*\|^2_{L^2}+\|\nabla\QQ\|^2_{H^s}+\|\vv\|^2_{H^s},\quad F_s\sim\|\nabla\QQ\|^2_{H^{s+1}}+\|\nabla\vv\|^2_{H^s}.
\end{align*}
Assume that $E^{(n)}_s=E_s(\QQ^{(n)},\vv^{(n)})$. In order to prove $(\QQ^{(n+1)},\vv^{(n+1)})\in\mathbb{X}$, we need to establish the closed energy estimate
\begin{align*}
\frac{\ud}{\ud t}E^{(n+1)}_s+\nu F^{(n+1)}_s\leq C(\delta,C_0,\nu)(1+E^{(n+1)}_s),
\end{align*}
for some small $\nu>0$. The proof is split into three steps.

{\it Step 1. $L^2$-estimate for $\QQ^{(n+1)}-\QQ^*$}. First of all, using the definition of $\CJ(\QQ)$ in (\ref{CJ-QQ}) and Lemma \ref{SS-Lip-lemma}, we deduce that
\begin{align}\label{CJ-QQ-n-estimate-L2}
\|\CJ(\QQ^{(n)})\|_{L^2}=&\|\MS\big(\CS(\QQ^{(n)})-\CS(\QQ^*)\big)+D_0(\QQ^{(n)}-\QQ^*)\|_{L^2}\nonumber\\
\leq&C_{\delta}\|\QQ^{(n)}-\QQ^*\|_{L^2}\leq C(\delta,C_0).
\end{align}
Taking the $L^2$-inner product on the equation (\ref{approximate-eq-1}) with $\QQ^{(n+1)}-\QQ^*$ and using  (\ref{CJ-QQ-n-estimate-L2}), we arrive at
\begin{align}\label{QQ-n+1-estimate-L2}
&\frac{1}{2}\frac{\ud }{\ud t}\|\QQ^{(n+1)}-\QQ^*\|^2_{L^2}=\langle\partial_t\QQ^{(n+1)},\QQ^{(n+1)}-\QQ^*\rangle\nonumber\\
&\quad=-\frac{1}{\ve}\big\langle\CM_{\QQ^{(n)}}\CJ(\QQ^{(n)}),\QQ^{(n+1)}-\QQ^*\big\rangle\nonumber\\
&\qquad-\big\langle\CM_{\QQ^{(n)}}\CG(\QQ^{(n+1)}),\QQ^{(n+1)}-\QQ^*\big\rangle+\big\langle\CV_{\QQ^{(n)}}\kappa^{(n+1)},\QQ^{(n+1)}-\QQ^*\big\rangle\nonumber\\
&\quad\leq \Big(C(\delta,C_0)+C_{\delta}\|\CG(\QQ^{(n+1)})\|_{L^2}+C_{\delta}\|\nabla\vv^{(n+1)}\|_{L^2}\Big)\|\QQ^{(n+1)}-\QQ^*\|_{L^2}\nonumber\\
&\quad\leq C(\delta,C_0)\big((E^{(n+1)}_s)^{1/2}+E^{(n+1)}_s\big).
\end{align}

{\it Step 2. $L^2$-estimate for $(\nabla\QQ^{(n+1)},\vv^{(n+1)})$}. Taking the $L^2$-inner product on the equation (\ref{approximate-eq-1}) with $\CG(\QQ^{(n+1)})$, and noticing the fact that the operator $\CM_{\QQ^{(n)}}$ is positive definite, we obtain
\begin{align}\label{CE-n+1-QQ-estimate-L2}
&\frac{\ud}{\ud t}\CF_e[\nabla\QQ^{(n+1)}]=\langle\partial_t\QQ^{(n+1)},\CG(\QQ^{(n+1)})\rangle\nonumber\\
&\quad=-\langle\vv^{(n)}\cdot\nabla\QQ^{(n+1)},\CG(\QQ^{(n+1)})\rangle-\frac{1}{\ve}\big\langle\CM_{\QQ^{(n)}}\CJ(\QQ^{(n)}),\CG(\QQ^{n+1})\big\rangle\nonumber\\
&\qquad-\big\langle\CM_{\QQ^{(n)}}\CG(\QQ^{(n)}),\CG(\QQ^{n+1})\big\rangle+\big\langle\CV_{\QQ^{(n)}}\kappa^{(n+1)},\CG(\QQ^{n+1})\big\rangle\nonumber\\
&\quad\leq\|\vv^{(n)}\|_{L^{\infty}}\|\nabla\QQ^{(n+1)}\|_{L^2}\|\CG(\QQ^{n+1})\|_{L^2}+C(\delta,C_0)\|\CG(\QQ^{(n+1)})\|_{L^2}\nonumber\\
&\qquad-\nu\|\CG(\QQ^{(n+1)})\|^2_{L^2}+\big\langle\CV_{\QQ^{(n)}}\kappa^{(n+1)},\CG(\QQ^{n+1})\big\rangle\nonumber\\
&\quad\leq C(\delta,C_0,\nu)(1+E^{(n+1)}_s)-\nu\|\CG(\QQ^{(n+1)})\|^2_{L^2}+\big\langle\CV_{\QQ^{(n)}}\kappa^{(n+1)},\CG(\QQ^{n+1})\big\rangle.
\end{align}
On the other hand, taking the inner product on the equation (\ref{approximate-eq-2}) with $\vv^{(n+1)}$, noticing that $\CM_{\QQ^{(n)}}$ is positive definite, we deduce that
\begin{align}\label{vv-n+1-estimate-L2}
&\frac{1}{2}\frac{\ud}{\ud t}\|\vv^{(n+1)}\|^2_{L^2}+\eta\|\nabla\vv^{(n+1)}\|^2_{L^2}\nonumber\\
&\quad=-\big\langle\CP_{\QQ^{(n)}}\kappa^{(n+1)},\nabla\vv^{(n+1)}\big\rangle-\frac{1}{\ve}\big\langle\CN_{\QQ^{(n)}}\CJ(\QQ^{(n)}),\nabla\vv^{(n+1)}\big\rangle\nonumber\\
&\qquad-\big\langle\CN_{\QQ^{(n)}}\CG(\QQ^{(n+1)}),\nabla\vv^{(n+1)}\big\rangle+\frac{1}{\ve}\big\langle\CJ(\QQ^{(n)})\cdot\partial_i\QQ^{(n)},(\vv^{(n+1)})_i\big\rangle\nonumber\\
&\qquad+\big\langle\CG(\QQ^{(n+1)})\cdot\partial_i\QQ^{(n)},(\vv^{(n+1)})_i\big\rangle\nonumber\\
&\quad\leq C(\delta,C_0)\|\nabla\vv^{(n+1)}\|_{L^2}-\big\langle\CN_{\QQ^{(n)}}\CG(\QQ^{(n+1)}),\nabla\vv^{(n+1)}\big\rangle\nonumber\\
&\qquad+C(\delta,C_0)\|\nabla\QQ^{(n)}\|_{H^2}\|\vv^{n+1}\|_{L^2}+C\|\nabla\QQ^{(n)}\|_{H^2}\|\nabla\QQ^{(n+1)}\|_{L^2}\|\nabla\vv^{(n+1)}\|_{L^2}\nonumber\\
&\quad\leq C(\delta,C_0,\nu)(1+E^{(n+1)}_s)-\big\langle\CN_{\QQ^{(n)}}\CG(\QQ^{(n+1)}),\nabla\vv^{(n+1)}\big\rangle
+\delta\|\nabla\vv^{n+1}\|^2_{L^2},
\end{align}
where $\delta$ represents a small positive constant to be determined later.
Consequently, noticing $\CN_{\QQ^{(n)}}=\CV^T_{\QQ^{(n)}}$ and combining (\ref{CE-n+1-QQ-estimate-L2}) with (\ref{vv-n+1-estimate-L2}), we arrive at
\begin{align}\label{QQ-vv-estimate-L2}
&\frac{\ud}{\ud t}\Big(\CF_e[\nabla\QQ^{(n+1)}]+\frac{1}{2}\|\vv^{(n+1)}\|^2_{L^2}\Big)+\nu\|\CG(\QQ^{(n+1)})\|^2_{L^2}\nonumber\\
&\qquad+(\eta-\delta)\|\nabla\vv^{(n+1)}\|^2_{L^2}\leq C(\delta,C_0,\nu)(1+E^{(n+1)}_s).
\end{align}

{\it Step 3. $L^2$-estimate for $(\nabla^{s+1}\QQ^{(n+1)},\nabla^s\vv^{(n+1)})$}. We first consider the estimate of the higher order derivative for $\QQ^{(n+1)}$. Acting the differential operator $\nabla^s$ on the equation (\ref{approximate-eq-1}) and taking the inner product with $\nabla^s\CG(\QQ^{(n+1)})$, we get
\begin{align}\label{CE-high-estimate-L2}
&\frac{\ud}{\ud t}\CF_e[\nabla^{s+1}\QQ^{(n+1)}]=\langle\nabla^s\partial_t\QQ^{(n+1)},\nabla^s\CG(\QQ^{n+1})\rangle\nonumber\\
&\quad=-\big\langle\nabla^s(\vv^{(n)}\cdot\nabla\QQ^{(n+1)}),\nabla^s\CG(\QQ^{(n+1)})\big\rangle
-\frac{1}{\ve}\big\langle\nabla^s\big(\CM_{\QQ^{(n)}}\CJ(\QQ^{(n)})\big),\nabla^s\CG(\QQ^{(n+1)})\big\rangle\nonumber\\
&\qquad-\big\langle\nabla^s\big(\CM_{\QQ^{(n)}}\CG(\QQ^{(n+1)})\big),\nabla^s\CG(\QQ^{(n+1)})\big\rangle+\big\langle\nabla^s\big(\CV_{\QQ^{(n)}}\kappa^{(n+1)}\big),\nabla^s\CG(\QQ^{(n+1)})\big\rangle\nonumber\\
&\quad\eqdefa I_1+I_2+I_3+I_4.
\end{align}
Now we estimate (\ref{CE-high-estimate-L2}) term by term as follows. Using Lemma \ref{SS-Lip-lemma}
and Lemma \ref{lem:product}, the terms $I_1$ and $I_2$ can be handled as
\begin{align*}
I_1\leq& C\|\vv^{(n)}\|_{H^s}\|\nabla\QQ^{(n+1)}\|_{H^s}\|\nabla^s\CG(\QQ^{(n+1)})\|_{L^2}\\
\leq& C(\delta,C_0)\big(E^{(n+1)}_sF^{(n+1)}_s\big)^{1/2},\\
I_2\leq&C_{\delta}\|\QQ^{(n)}-\QQ^*\|_{H^{s+1}}\|\nabla^{s-1}\CG(\QQ^{n+1})\|_{L^2}\leq C(\delta,C_0)\big(E^{(n+1)}_s\big)^{1/2}.
\end{align*}
Moreover, taking advantage of Lemma \ref{commutor-YY-QQ-lemma}, we derive that
\begin{align*}
I_3=&-\big\langle\CM_{\QQ^{(n)}}\big(\nabla^s\CG(\QQ^{(n+1)})\big),\nabla^s\CG(\QQ^{(n+1)})\big\rangle\\
&-\big\langle[\nabla^s,\CM_{\QQ^{(n)}}]\CG(\QQ^{(n+1)}),\nabla^s\CG(\QQ^{(n+1)})\big\rangle\\
\leq&-\nu\|\nabla^s\CG(\QQ^{(n+1)})\|^2_{L^2}+C_{\delta}\|\QQ^{(n)}-\QQ^*\|_{H^s}\|\CG(\QQ^{(n+1)})\|_{H^{s-1}}\|\nabla^s\CG(\QQ^{(n+1)})\|_{L^2}\\
\leq&-\nu\|\nabla^s\CG(\QQ^{(n+1)})\|^2_{L^2}+C(\delta,C_0)\big(E^{(n+1)}_sF^{(n+1)}_s\big)^{1/2},\\
I_4=&\big\langle\CV_{\QQ^{(n)}}\big(\nabla^s\kappa^{(n+1)}\big),\nabla^s\CG(\QQ^{(n+1)})\big\rangle\\
&+\big\langle[\nabla^s,\CV_{\QQ^{(n)}}]\kappa^{(n+1)},\nabla^s\CG(\QQ^{(n+1)})\big\rangle\\
\leq&\big\langle\CV_{\QQ^{(n)}}\big(\nabla^s\kappa^{(n+1)}\big),\nabla^s\CG(\QQ^{(n+1)})\big\rangle\\
&+C_{\delta}\|\QQ^{(n)}-\QQ^*\|_{H^s}\|\nabla\vv^{(n+1)}\|_{H^{s-1}}\|\nabla^s\CG(\QQ^{(n+1)})\|_{L^2}\\
\leq&\big\langle\CV_{\QQ^{(n)}}\big(\nabla^s\kappa^{(n+1)}\big),\nabla^s\CG(\QQ^{(n+1)})\big\rangle+C(\delta,C_0)\big(E^{(n+1)}_sF^{(n+1)}_s\big)^{1/2}.
\end{align*}
Plugging the above estimates into (\ref{CE-high-estimate-L2}) leads to
\begin{align}\label{CE-high-estimate-L2-final}
&\frac{\ud}{\ud t}\CF_e[\nabla^{s+1}\QQ^{(n+1)}]+\nu\|\nabla^s\CG(\QQ^{(n+1)})\|^2_{L^2}\nonumber\\
&\quad\leq \big\langle\CV_{\QQ^{(n)}}\big(\nabla^s\kappa^{(n+1)}\big),\nabla^s\CG(\QQ^{(n+1)})\big\rangle+C(\delta,C_0)\big(E^{(n+1)}_sF^{(n+1)}_s\big)^{1/2}.
\end{align}

We now turn to the estimate of the higher order derivative of $\vv^{(n+1)}$. Similarly, applying the equation (\ref{approximate-eq-2}), we deduce that
\begin{align}\label{vv-high-estimate-L2}
&\frac{1}{2}\frac{\ud}{\ud t}\|\nabla^s\vv^{(n+1)}\|^2_{L^2}+\eta\|\nabla^{s+1}\vv^{n+1}\|^2_{L^2}\nonumber\\
&\quad=-\big\langle\nabla^s(\vv^{(n)}\cdot\nabla\vv^{(n+1)}),\nabla^s\vv^{(s+1)}\big\rangle-\big\langle\nabla^s(\CP_{\QQ^{(n)}}\kappa^{(n+1)}),\nabla^{s+1}\vv^{(n+1)}\big\rangle\nonumber\\
&\qquad-\frac{1}{\ve}\big\langle\nabla^s\big(\CN_{\QQ^{(n)}}\CJ(\QQ^{(n)})\big),\nabla^{s+1}\vv^{(n+1)}\big\rangle-\big\langle\nabla^s\big(\CN_{\QQ^{(n)}}\CG(\QQ^{(n+1)})\big),\nabla^{s+1}\vv^{(n+1)}\big\rangle\nonumber\\
&\qquad+\frac{1}{\ve}\big\langle\nabla^s\big(\CJ(\QQ^{(n)})\cdot\partial_i\QQ^{(n)}\big),\nabla^s(\vv^{(n+1)})_i\big\rangle+\big\langle\nabla^s\big(\CG(\QQ^{(n+1)})\cdot\partial_i\QQ^{(n)}\big),\nabla^s(\vv^{(n+1)})_i\big\rangle\nonumber\\
&\quad\eqdefa I'_1+I'_2+I'_3+I'_4+I'_5+I'_6.
\end{align}
In virtue of Lemma \ref{SS-Lip-lemma}
and Lemma \ref{lem:product}, the terms $I'_1, I'_3$ and $I'_5$ can be estimated as follows:
\begin{align*}
I'_1\leq&C\|\vv^{(n)}\|_{H^s}\|\nabla\vv^{(n+1)}\|_{H^s}\|\nabla^s\vv^{(n+1)}\|_{L^2}\\
\leq&C(\delta,C_0)\big(E^{(n+1)}_s+(E^{(n+1)}_sF^{(n+1)}_s)^{1/2}\big),\\
I'_3\leq&C_{\delta}\|\nabla\QQ^{(n)}\|_{H^s}\|\nabla^s\vv^{(n+1)}\|_{L^2}\leq C(\delta,C_0)(E^{(n+1)}_s)^{1/2},\\
I'_5\leq&C\|\CJ(\QQ^{(n)})\|_{H^s}\|\nabla\QQ^{(n)}\|_{H^s}\|\nabla^s\vv^{(n+1)}\|_{L^2}\leq C(\delta,C_0)(E^{(n+1)}_s)^{1/2}.
\end{align*}
Applying Lemma \ref{commutor-YY-QQ-lemma} and Lemma \ref{lem:commutator}, along with the fact that $\CP_{\QQ^{(n)}}$ is positive definite, we have
\begin{align*}
I'_2=&-\big\langle\CP_{\QQ^{(n)}}(\nabla^s\kappa^{(n+1)}),\nabla^{s+1}\vv^{(n+1)}\big\rangle\\
&-\big\langle[\nabla^s,\CP_{\QQ^{(n)}}]\kappa^{(n+1)},\nabla^{s+1}\vv^{n+1}\big\rangle\\
\leq&C(\delta,C_0)\|\nabla\QQ^{(n)}\|_{H^s}\|\nabla\vv^{(n+1)}\|_{H^{s-1}}\|\nabla^{s+1}\vv^{(n+1)}\|_{L^2}\\
\leq&C(\delta,C_0)(E^{(n+1)}_s)^{1/2},\\
I'_4=&-\big\langle\CN_{\QQ^{(n)}}\big(\nabla^s\CG(\QQ^{(n+1)})\big),\nabla^{s+1}\vv^{(n+1)}\big\rangle\\
&-\big\langle[\nabla^s,\CN_{\QQ^{(n)}}]\CG(\QQ^{(n+1)}),\nabla^{s+1}\vv^{(n+1)}\big\rangle\\
\leq&-\big\langle\CN_{\QQ^{(n)}}\big(\nabla^s\CG(\QQ^{(n+1)})\big),\nabla^{s+1}\vv^{(n+1)}\big\rangle\\
&+C_{\delta}\|\QQ^{(n)}-\QQ^*\|_{H^s}\|\CG(\QQ^{(n+1)})\|_{H^{s-1}}\|\nabla^{s+1}\vv^{(n+1)}\|_{L^2}\\
\leq&-\big\langle\CN_{\QQ^{(n)}}\big(\nabla^s\CG(\QQ^{(n+1)})\big),\nabla^{s+1}\vv^{(n+1)}\big\rangle+C(\delta,C_0)(E^{(n+1)}_sF^{(n+1)}_s)^{1/2}\\
I'_6=&\Big\langle\nabla^{s}\partial_j\Big(\frac{\partial F_e(\nabla\QQ^{(n+1)})}{\partial(\partial_j\QQ^{(n+1)})}\cdot\partial_i\QQ^{(n)}\Big),\nabla^{s}v^{(n+1)}_i\Big\rangle\\
=&-\Big\langle\nabla^{s}\Big(\frac{\partial F_e(\nabla\QQ^{(n+1)})}{\partial(\partial_j\QQ^{(n+1)})}\cdot\partial_i\QQ^{(n)}\Big),\partial_j\nabla^{s}v^{(n+1)}_i\Big\rangle\\
\leq&C\|\nabla\QQ^{(n)}\|_{H^s}\|\nabla\QQ^{(n+1)}\|_{H^s}\|\nabla\vv^{(n+1)}\|_{H^s}\leq C(C_0)(E^{(n+1)}_sF^{(n+1)}_s)^{1/2},
\end{align*}
where for the estimate of $I'_6$ we have used the following identity:
\begin{align}\label{additional-pressure-term}
\CG(\QQ)\cdot\partial_i\QQ=-\partial_j\Big(\frac{\partial F_e(\nabla\QQ)}{\partial(\partial_j\QQ)}\cdot\partial_i\QQ\Big)-\partial_i\widetilde{p}
\end{align}
with $\widetilde{p}$ being an additional pressure term that can be adsorbed into the pressure $p$.
Then, from the above estimates of $I'_i(i=1,\cdots,6)$ and (\ref{vv-high-estimate-L2}), it follows that
\begin{align}\label{vv-high-estimate-L2-final}
&\frac{1}{2}\frac{\ud}{\ud t}\|\nabla^s\vv^{(n+1)}\|^2_{L^2}+\eta\|\nabla^{s+1}\vv^{n+1}\|^2_{L^2}\nonumber\\
&\quad\leq-\big\langle\CN_{\QQ^{(n)}}\big(\nabla^s\CG(\QQ^{(n+1)})\big),\nabla^{s+1}\vv^{(n+1)}\big\rangle\nonumber\\
&\qquad+C(\delta,C_0)(E^{(n+1)}_s)^{1/2}\Big(1+(F^{(n+1)}_s)^{1/2}\Big).
\end{align}
Therefore, combining (\ref{QQ-n+1-estimate-L2}), (\ref{QQ-vv-estimate-L2}), (\ref{CE-high-estimate-L2-final}) with (\ref{vv-high-estimate-L2-final}), we arrive at
\begin{align}\label{Es-n+1-estimate}
\frac{1}{2}\frac{\ud}{\ud t}E^{(n+1)}_s+\nu F^{(n+1)}_s\leq C(\delta,C_0)(1+E^{(n+1)}_s),
\end{align}
for sufficiently small $\nu>0$.
The Gronwall's inequality implies that for any $t\in[0,T]$, there holds
\begin{align*}
E^{(n+1)}_s(t)\leq&\big(1+E^{(n+1)}_s(0)\big)\exp\big(C(\delta,C_0)t\big)-1\\
=&\big(1+E_s(\QQ_I,\vv_I)\big)\exp\big(C(\delta,C_0)t\big)-1.
\end{align*}
Consequently, if we take $T_0>0$ such that
\begin{align*}
C(\delta,C_0)T_0\leq \ln(1+C_0)-\ln\big(1+E_s(\QQ_I,\vv_I)\big),
\end{align*}
then it follows that $\sup\limits_{0\leq t\leq T_0}E^{(n+1)}_s(t)\leq C_0$.

On the other hand, using the equation (\ref{approximate-eq-1}) we derive
\begin{align*}
&\Big\|\int^t_0\partial_t\QQ^{(n+1)}(\xx,t)\ud t\Big\|_{L^{\infty}}\leq\int^t_0\|\partial_t\QQ^{(n+1)}(\xx,t)\|_{H^2}\ud t\\
&\quad\leq C(\delta,C_0)\int^t_0\Big(\|\CG(\QQ^{(n+1)})\|_{H^2}+\|\nabla\vv^{(n+1)}\|_{H^2}+\|\QQ^{(n+1)}-\QQ^*\|_{H^3}+1\Big)\ud t\\
&\quad\leq C(\delta,C_0)t,
\end{align*}
which together with $\QQ_I\in\mathbb{Q}_{\delta}$ implies that $\QQ^{(n+1)}\in\mathbb{Q}_{\delta/2}$ for $t\in[0,T_0]$, if taking $T_0>0$ sufficiently small. Thus, we arrive at $(\QQ^{(n+1)},\vv^{(n+1)})\in\mathbb{X}(\delta,T,C_0)$ for $T\leq T_0$.

\subsection{Convergence of the sequence}
The subsection will be devoted to showing that the approximate solution sequence $\{(\QQ^{(\ell)},\vv^{(\ell)})\}_{\ell\in\mathbb{N}}$ is a Cauchy sequence, and to finishing the proof of Theorem \ref{locall-posedness-theorem}.

For this purpose, we define
\begin{align*}
&\delta^{\ell+1}_{\QQ}=\QQ^{(\ell+1)}-\QQ^{(\ell)},\quad \delta^{\ell}_{\CM}=\CM_{\QQ^{(\ell)}}-\CM_{\QQ^{(\ell-1)}},\quad\delta^{\ell}_{\CV}=\CV_{\QQ^{(\ell)}}-\CV_{\QQ^{(\ell-1)}},\\
&\delta^{\ell}_{\CN}=\CN_{\QQ^{(\ell)}}-\CN_{\QQ^{(\ell-1)}},\quad
\delta^{\ell}_{\CP}=\CP_{\QQ^{(\ell)}}-\CP_{\QQ^{(\ell-1)}},\quad
\delta^{\ell}_{\CJ}=\CJ(\QQ^{(\ell)})-\CJ(\QQ^{(\ell-1)}),\\ &\delta^{\ell+1}_{\vv}=\vv^{(\ell+1)}-\vv^{(\ell)},\quad \delta^{\ell+1}_{\kappa}=\kappa^{(\ell+1)}-\kappa^{(\ell)},\quad \delta^{\ell+1}_p=p^{(\ell+1)}-p^{(\ell)}.
\end{align*}
Assume that $(\QQ^{(\ell+1)},\vv^{(\ell+1)})$ and $(\QQ^{(\ell)},\vv^{(\ell)})$ are two solutions to the linearized system (\ref{approximate-eq-1})--(\ref{approximate-eq-3}) with the same initial data.
Taking the difference between the equations for $(\QQ^{(\ell+1)},\vv^{(\ell+1)})$ and $(\QQ^{(\ell)},\vv^{(\ell)})$, we deduce that
\begin{align}
\frac{\partial\delta^{\ell+1}_{\QQ}}{\partial t}+\vv^{(\ell)}\cdot\nabla\delta^{\ell+1}_{\QQ}=&~\CM_{\QQ^{(\ell)}}\CG(\delta^{\ell+1}_{\QQ})+\CV_{\QQ^{(\ell)}}\delta^{\ell+1}_{\kappa}+\delta\FF^{\ell}_1,\label{difference-equation-1}\\
\Big(\frac{\partial\delta^{\ell+1}_{\vv}}{\partial t}+\vv^{(\ell)}\cdot\nabla\delta^{\ell+1}_{\vv}\Big)_i=&~-\partial_i\delta^{\ell+1}_p+\eta\Delta\delta^{\ell+1}_{v_i}+\partial_j\big(\CP_{\QQ^{(\ell)}}\delta^{\ell+1}_{\kappa}\big)_{ij}\nonumber\\
&+\partial_j\big(\CN_{\QQ^{(\ell)}}\CG(\delta^{\ell+1}_{\QQ})\big)_{ij}+\CG(\delta^{\ell+1}_{\QQ})\cdot\partial_i\QQ^{(\ell)}+\partial_j(\delta\FF^{\ell}_2)_{ij},\label{difference-equation-2}\\
\nabla\cdot\delta^{\ell+1}_{\vv}=&~0,\label{difference-equation-3}
\end{align}
where $\delta\FF^{\ell}_1$ and $\delta\FF^{\ell}_2$ are given by
\begin{align*}
\delta\FF^{\ell}_1=&-\delta^{\ell}_{\vv}\cdot\nabla\QQ^{(\ell)}+\frac{1}{\ve}\big(\CM_{\QQ^{(\ell)}}\delta^{\ell}_{\CJ}+\delta^{\ell}_{\CM}\CJ(\QQ^{(\ell-1)})\big)+\delta^{\ell}_{\CM}\CG(\QQ^{(\ell)})+\delta^{\ell}_{\CV}\kappa^{(\ell)},\\
\partial_j(\delta\FF^{\ell}_2)_{ij}=&~\partial_j\Big(-\delta^{\ell}_{\vv}\otimes\vv^{(\ell)}+\delta^{\ell}_{\CP}\kappa^{(\ell)}+\frac{1}{\ve}\big(\CN_{\QQ^{(\ell)}}\delta^{\ell}_{\CJ}+\delta^{\ell}_{\CN}\CJ(\QQ^{(\ell-1)})\big)\Big)+\delta^{\ell}_{\CN}\CG(\QQ^{(\ell)})\\
&+\frac{1}{\ve}\Big(\delta^{\ell}\cdot\partial\QQ^{\ell}+\CJ(\QQ^{(\ell-1)})\cdot\partial_i\delta^{(\ell)}_{\QQ}\Big).
\end{align*}
Using Lemma \ref{defference-YY-UU-lemma} and integrating by parts, we obtain
\begin{align*}
\|(\delta\FF^{\ell}_1,\delta\FF^{\ell}_2)\|_{L^2}\leq&C(\delta,C_0)(\|\delta^{\ell}_{\QQ}\|_{H^1}+\|\delta^{\ell}_{\vv}\|_{L^2}).
\end{align*}

Similar to the argument in (\ref{QQ-vv-estimate-L2}), we can prove that there exist a sufficiently small $\nu>0$ and $C(\delta,C_0,\nu)>0$, such that
\begin{align}\label{tilde-E0l-estimate}
&\frac{\ud}{\ud t}\widetilde{E}^{(\ell+1)}_0(t)+\frac{\eta}{2}\|\nabla\delta^{\ell+1}_{\vv}\|^2_{L^2}+\nu\|\CG(\delta^{\ell+1}_{\QQ})\|^2_{L^2}\nonumber\\
&\quad\leq C(\delta,C_0,\nu)\big(\|\delta^{\ell+1}_{\vv}\|^2_{L^2}+\|\delta^{\ell+1}_{\QQ}\|^2_{H^1}+\|\delta^{\ell}_{\vv}\|^2_{L^2}+\|\delta^{\ell}_{\QQ}\|^2_{H^1}\big),
\end{align}
where
\begin{align*}
\widetilde{E}^{(\ell)}_0(t)\eqdefa\frac{1}{2}\|\delta^{\ell}_{\QQ}\|^2_{L^2}+\frac{1}{2}\|\delta^{\ell}_{\vv}\|^2_{L^2}+\CF_e(\nabla\delta^{\ell}_{\QQ}).
\end{align*}
Then, from (\ref{tilde-E0l-estimate}) we know
\begin{align*}
\frac{\ud}{\ud t}\widetilde{E}^{(\ell+1)}_0(t)\leq C\Big(\widetilde{E}^{(\ell)}_0(t)+\widetilde{E}^{(\ell+1)}_0(t)\Big),
\end{align*}
which further implies that
\begin{align*}
\widetilde{E}^{(\ell+1)}_0(t)\leq C\int^t_0\exp\big(C(t-\tau)\big)\widetilde{E}^{(\ell)}_0(\tau)\ud\tau\leq C\int^T_0\exp\big(C(T-\tau)\big)\ud\tau\sup_{t\in(0,T]}\widetilde{E}^{(\ell)}_0(t).
\end{align*}
Thus, taking $T<T_0$ small enough such that $C\int^T_0\exp\big(C(T-\tau)\big)\ud\tau\leq\frac{1}{2}$, we arrive at
\begin{align*}
\sup_{t\in(0,T]}\widetilde{E}^{(\ell+1)}_0(t)\leq\frac{1}{2}\sup_{t\in(0,T]}\widetilde{E}^{(\ell)}_0(t).
\end{align*}
This implies that $\{(\QQ^{(\ell)},\vv^{(\ell)})\}_{\ell\in\mathbb{N}}$ is a Cauchy sequence. More precisely, there exists the limits $\QQ-\QQ^*\in C([0,T];H^1(\mathbb{R}^3))$ and $\vv\in C([0,T];L^2(\mathbb{R}^3))$ such that
\begin{align*}
&\QQ^{(n)}-\QQ^*\rightarrow \QQ-\QQ^*\in C([0,T];H^1(\mathbb{R}^3)),\\
&\vv^{(n)}\rightarrow \vv\in C([0,T];L^2(\mathbb{R}^3)).
\end{align*}
Applying the uniform bounds and the Sobolev's interpolation theorem, there holds
\begin{align*}
&\QQ^{(n)}-\QQ^*\rightarrow \QQ-\QQ^*\in C([0,T];H^{s'+1}(\mathbb{R}^3)),\\
&\vv^{(n)}\rightarrow \vv\in C([0,T];H^{s'}(\mathbb{R}^3)),
\end{align*}
for any $s'\in(0,s)$. Therefore, the limit $(\QQ,\vv)$ is just the classical solution to the system (\ref{Re-MB-Q-tensor-1})--(\ref{Re-MB-Q-tensor-3}). Following the proof of convergence for the sequence $\{(\QQ^{(n)},\vv^{(n)})\}_{n\in\mathbb{N}}$, the uniqueness of the limit $(\QQ,\vv)$ can be obtained by the similar energy estimate. Furthermore, by the standard regularity argument for parabolic system, we obtain
\begin{align*}
\QQ-\QQ^*\in C([0,T];H^{s+1}(\mathbb{R}^3)),\quad \vv\in C([0,T];H^s(\mathbb{R}^3))\cap L^2([0,T];H^{s+1}(\mathbb{R}^3)).
\end{align*}
We omit the details here. This completes the proof of Theorem \ref{locall-posedness-theorem}.

\section{Rigorous biaxial limit of two-tensor hydrodynamics}\label{bipr}
In this section, based on the Hilbert expansion of solutions with respect to $\ve$, we rigorously derive the biaxial frame hydrodynamics from the two-tensor hydrodynamics.

\subsection{The Hilbert expansion}\label{Hilbert-subsection}
Let $(\QQ^{\ve},\vv^{\ve})$ be a solution to the system (\ref{Re-MB-Q-tensor-1})--(\ref{Re-MB-Q-tensor-3}). We make the following Hilbert expansion:
\begin{align}
\QQ^{\ve}=&\sum^3_{k=0}\ve^k\QQ^{(k)}+\ve^3\QQ_R\eqdefa\widetilde{\QQ}+\ve^3\QQ_R,\label{QQ-Hilbert-expansion}\\
\vv^{\ve}=&\sum^2_{k=0}\ve^k\vv^{(k)}+\ve^3\vv_R\eqdefa\widetilde{\vv}+\ve^3\vv_R,\label{vv-Hilbert-expansion}
\end{align}
where $\QQ^{(k)}(0\leq k\leq3)$ and $\vv^{(l)}(0\leq l\leq2)$ are independent of $\ve$, and $(\QQ_R,\vv_R)$ represent the remainder term depending upon $\ve$.

By the Taylor expansion, we obtain
\begin{align}
\CJ(\QQ^{\ve})=&\CJ(\QQ^{(0)})+\ve\CH_{\QQ^{(0)}}\QQ^{(1)}+\ve^2\big(\CH_{\QQ^{(0)}}\QQ^{(2)}+\JJ_1\big)+\ve^3\big(\CH_{\QQ^{(0)}}\QQ^{(3)}+\JJ_2\big)\nonumber\\
&+\ve^3\CH_{\QQ^{(0)}}\QQ_R+\ve^4\CJ^{\ve}_R,
\end{align}
where $\JJ_1,\JJ_2$, $\CJ^{\ve}_R$ are given by
\begin{align*}
\JJ_1=&\frac{1}{2}\big(\CJ''(\QQ^{(0)})\QQ^{(1)}\big)\cdot\QQ^{(1)},\\
\JJ_2=&\frac{1}{2}\big(\CJ''(\QQ^{(0)})\QQ^{(1)}\big)\cdot\QQ^{(2)}+\frac{1}{2}\big(\CJ''(\QQ^{(0)})\QQ^{(2)}\big)\cdot\QQ^{(1)}\\
&+\frac{1}{3!}\big(\CJ'''(\QQ^{(0)})\QQ^{(1)}\QQ^{(1)}\big)\cdot\QQ^{(1)},\\
\CJ^{\ve}_R=&\frac{1}{2}\sum\limits_{\mbox{\tiny$\begin{array}{c}
1\leq i,j\leq 3\\
i+j\geq4\end{array}$}}\ve^{i+j-4}\big(\CJ''(\QQ^{(0)})\QQ^{(i)}\big)\cdot\QQ^{(j)}\\
&+\frac{1}{3!}\sum_{\mbox{\tiny$\begin{array}{c}
i+j+k\geq4\\
\text{at least two of}~i,j,k~\text{are not zero}\end{array}$}}
\ve^{i+j+k-4}\Big(\CJ'''(\QQ^{(0)})\QQ^{(i)}\QQ^{(j)}\Big)\cdot\QQ^{(k)}\\
&+\frac{1}{4!}\CJ^{(4)}\big(\QQ^{(0)}+\theta_1\overline{\QQ}^{\ve}\big)(\overline{\QQ}^{\ve})^4+\big(\CJ''(\QQ_0+\theta_2\ve\overline{\QQ}^{\ve})\overline{\QQ}^{\ve}\big)\cdot\QQ_R\\
&+\frac{1}{2}\ve^2\big(\CJ''(\widetilde{\QQ}+\theta_3\ve^3\QQ_R)\QQ_R\big)\cdot\QQ_R,\quad\forall ~\theta_l\in(0,1),~l=1,2,3,
\end{align*}
with $\overline{\QQ}^{\ve}=\QQ^{(1)}+\ve\QQ^{(2)}+\ve^2\QQ^{(3)}$.

Since $\CM_{\QQ},\CV_{\QQ},\CN_{\QQ}$ and $\CP_{\QQ}$ are functions of $\QQ$, we have the following expansions:
\begin{align}\label{CM-CV-CN-CP-out-expansions}
\left\{
\begin{array}{l}
\CM_{\QQ^{\ve}}=\sum\limits^3_{k=0}\ve^k\CM^{(k)}+\ve^3\CM_R+\ve^4\mathfrak{R}_{\CM},\quad \CV_{\QQ^{\ve}}=\sum\limits^3_{k=0}\ve^k\CV^{(k)}+\ve^3\CV_R+\ve^4\mathfrak{R}_{\CV},\vspace{1ex}\\
\CN_{\QQ^{\ve}}=\sum\limits^3_{k=0}\ve^k\CN^{(k)}+\ve^3\CN_R+\ve^4\mathfrak{R}_{\CN},\quad \CP_{\QQ^{\ve}}=\sum\limits^3_{k=0}\ve^k\CP^{(k)}+\ve^3\CP_R+\ve^4\mathfrak{R}_{\CP},
\end{array}
\right.
\end{align}
where $\CM^{(k)},\CV^{(k)},\CN^{(k)}$ and $\CP^{(k)} (k\geq0)$ are given by, respectively,
\begin{align*}
&\CM^{(k)}=\left(
  \begin{array}{cc}
    \CM^{(k)}_{11} &\, \CM^{(k)}_{12}\vspace{0.5ex} \\
    \CM^{(k)}_{12} &\, \CM^{(k)}_{22} \\
  \end{array}
\right)=
\left(
  \begin{array}{cc}
    \Gamma_2\CR^{(k)}_{4}+\Gamma_3\CR^{(k)}_{3} &\, -\Gamma_3\CR^{(k)}_{3} \vspace{0.5ex}\\
    -\Gamma_3\CR^{(k)}_{3} &\, \Gamma_1\CR^{(k)}_{5}+\Gamma_3\CR^{(k)}_{3} \\
  \end{array}
\right), \\
&\CV^{(k)}=\left(\begin{array}{c}
    \CV^{(k)}_{Q_1}  \vspace{0.5ex}\\
    \CV^{(k)}_{Q_2}
\end{array}\right),
\quad
\CN^{(k)}=(\CN^{(k)}_{Q_1},\CN^{(k)}_{Q_2})=\big((\CV^{(k)}_{Q_1})^T, (\CV^{(k)}_{Q_2})^T\big),\\
&\CP^{(k)}=\zeta\big(I_{22}\CR^{(k)}_{1}+I_{11}\CR^{(k)}_{2}+e_1I_{11}\CR^{(k)}_{3}\big).
\end{align*}
Here, $\CM^{(0)},\CV^{(0)},\CN^{(0)},\CP^{(0)}$ are those calculated from closure approximation at $\QQ^{(0)}$. Again, $\CM^{(k)}$, $\CV^{(k)}, \CN^{(k)},\CP^{(k)}(1\leq k\leq3)$ merely depend on $\QQ^{(k)}(0\leq k\leq3)$ but are independent of $\ve$, and 
are polynomials of degree $k$ with respect to $\QQ^{(k)}(1\leq k\leq3)$, respectively. Moreover, $\CM_R,\CV_R,\CN_R$ and $\CP_R$ depend on $\QQ^{(k)}(0\leq k\leq3)$ and $\QQ_R$, and are all linear with respect to $\QQ_R$. All higher order terms with respect to $\ve$ are contained in the terms $\ve^4\mathfrak{R}_{\CM},\ve^4\mathfrak{R}_{\CV},\ve^4\mathfrak{R}_{\CN}$ and $\ve^4\mathfrak{R}_{\CP}$, respectively.

We are now in a position to write down the expansion of the system (\ref{Re-MB-Q-tensor-1})--(\ref{Re-MB-Q-tensor-3}) with the small parameter $\ve$ and collect the terms (independent of the remainder term $(\QQ_R,\vv_R)$) with the same order of $\ve$. More specifically, we have the following:

$\bullet$ {\it The $O(\ve^{-1})$ system}:
\begin{align}\label{O-1-epsion-system}
\CM^{(0)}\CJ(\QQ^{(0)})=0\Rightarrow\CJ(\QQ^{(0)})=0,
\end{align}
because $\CM^{(0)}$ is positive definite.

$\bullet$ {\it Zeroth order term in $\ve$}:
\begin{align}
\frac{\partial\QQ^{(0)}}{\partial t}+\vv^{(0)}\cdot\nabla\QQ^{(0)}=&-\CM^{(0)}\big(\CH_{\QQ^{(0)}}\QQ^{(1)}+\CG(\QQ^{(0)})\big)+\CV^{(0)}\kappa^{(0)},\label{O-0-epsion-system-1}\\
\Big(\frac{\partial\vv^{(0)}}{\partial t}+\vv^{(0)}\cdot\nabla\vv^{(0)}\Big)_i=&-\partial_ip^{(0)}+\eta\Delta v^{(0)}_i+\partial_j\big(\CP^{(0)}\kappa^{(0)}\big)_{ij}\nonumber\\
&+\partial_j\Big(\CN^{(0)}\big(\CH_{\QQ^{(0)}}\QQ^{(1)}+\CG(\QQ^{(0)})\big)_{ij}\Big)\nonumber\\
&+\big(\CH_{\QQ^{(0)}}\QQ^{(1)}+\CG(\QQ^{(0)})\big)\cdot\partial_i\QQ^{(0)},\label{O-0-epsion-system-2}\\
\nabla\cdot\vv^{(0)}=&~0.\label{O-0-epsion-system-3}
\end{align}

$\bullet$ {\it First order term in $\ve$}:
\begin{align}
\frac{\partial\QQ^{(1)}}{\partial t}+\vv^{(0)}\cdot\nabla\QQ^{(1)}=&-\CM^{(0)}\big(\CH_{\QQ^{(0)}}\QQ^{(2)}+\CG(\QQ^{(1)})+\JJ_1\big)+\CV^{(0)}\kappa^{(1)}+\FF_1,\label{O-1-epsion-system-1}\\
\Big(\frac{\partial\vv^{(1)}}{\partial t}+\vv^{(0)}\cdot\nabla\vv^{(1)}\Big)_i=&-\partial_ip^{(1)}+\eta\Delta v^{(1)}_i+\partial_j\big(\CP^{(0)}\kappa^{(1)}\big)_{ij}\nonumber\\
&+\partial_j\Big(\CN^{(0)}\big(\CH_{\QQ^{(0)}}\QQ^{(2)}+\CG(\QQ^{(1)})+\JJ_1\big)_{ij}\Big)\nonumber\\
&+\big(\CH_{\QQ^{(0)}}\QQ^{(2)}+\CG(\QQ^{(1)})+\JJ_1\big)\cdot\partial_i\QQ^{(0)}+\GG_1,\label{O-1-epsion-system-2}\\
\nabla\cdot\vv^{(1)}=&~0,\label{O-1-epsion-system-3}
\end{align}
where $\FF_1$ and $\GG_1$ are given by
\begin{align*}
\FF_1=&-\vv^{(1)}\cdot\nabla\QQ^{(0)}-\CM^{(1)}\big(\CH_{\QQ^{(0)}}\QQ^{(1)}+\CG(\QQ^{(0)})\big)+\CV^{(1)}\kappa^{(0)},\\
\GG_1=&-\vv^{(1)}\cdot\nabla\vv^{(0)}+\partial_j\big(\CP^{(1)}\kappa^{(0)}\big)_{ij}+\partial_j\Big(\CN^{(1)}\big(\CH_{\QQ^{(0)}}\QQ^{(1)}+\CG(\QQ^{(0)})\big)_{ij}\Big)\\
&+\big(\CH_{\QQ^{(0)}}\QQ^{(1)}+\CG(\QQ^{(0)})\big)\cdot\partial_i\QQ^{(1)}.
\end{align*}

$\bullet$ {\it Second order term in $\ve$}:
\begin{align}
\frac{\partial\QQ^{(2)}}{\partial t}+\vv^{(0)}\cdot\nabla\QQ^{(2)}=&-\CM^{(0)}\big(\CH_{\QQ^{(0)}}\QQ^{(3)}+\CG(\QQ^{(2)})+\JJ_2\big)+\CV^{(0)}\kappa^{(2)}+\FF_2,\label{O-2-epsion-system-1}\\
\Big(\frac{\partial\vv^{(2)}}{\partial t}+\vv^{(0)}\cdot\nabla\vv^{(2)}\Big)_i=&-\partial_ip^{(2)}+\eta\Delta v^{(2)}_i+\partial_j\big(\CP^{(0)}\kappa^{(2)}\big)_{ij}\nonumber\\
&+\partial_j\Big(\CN^{(0)}\big(\CH_{\QQ^{(0)}}\QQ^{(3)}+\CG(\QQ^{(2)})+\JJ_2\big)_{ij}\Big)\nonumber\\
&+\big(\CH_{\QQ^{(0)}}\QQ^{(3)}+\CG(\QQ^{(2)})+\JJ_2\big)\cdot\partial_i\QQ^{(0)}+\GG_2,\label{O-2-epsion-system-2}\\
\nabla\cdot\vv^{(2)}=&~0,\label{O-2-epsion-system-3}
\end{align}
where $\FF_2$ and $\GG_2$ are given by
\begin{align*}
\FF_2=&-\vv^{(2)}\cdot\nabla\QQ^{(0)}-\vv^{(1)}\cdot\nabla\QQ^{(1)}-\CM^{(2)}\big(\CH_{\QQ^{(0)}}\QQ^{(1)}+\CG(\QQ^{(0)})\big)\\
&-\CM^{(1)}\big(\CH_{\QQ^{(0)}}\QQ^{(2)}+\CG(\QQ^{(1)})+\JJ_1\big)+\CV^{(2)}\kappa^{(0)}+\CV^{(1)}\kappa^{(1)},\\
\GG_2=&-\vv^{(2)}\cdot\nabla\vv^{(0)}-\vv^{(1)}\cdot\nabla\vv^{(1)}+\partial_j\big(\CP^{(2)}\kappa^{(0)}\big)_{ij}+\partial_j\big(\CP^{(1)}\kappa^{(1)}\big)_{ij}\\
&+\partial_j\Big(\CN^{(2)}\big(\CH_{\QQ^{(0)}}\QQ^{(1)}+\CG(\QQ^{(0)})\big)_{ij}\Big)+\partial_j\Big(\CN^{(1)}\big(\CH_{\QQ^{(0)}}\QQ^{(2)}+\CG(\QQ^{(1)})+\JJ_1\big)_{ij}\Big)\\
&+\big(\CH_{\QQ^{(0)}}\QQ^{(2)}+\CG(\QQ^{(1)})+\JJ_1\big)\cdot\partial_i\QQ^{(1)}+\big(\CH_{\QQ^{(0)}}\QQ^{(1)}+\CG(\QQ^{(0)})\big)\cdot\partial_i\QQ^{(2)}.
\end{align*}

The problem becomes how to solve $(\QQ^{(k)},\vv^{(k)})(0\leq k\leq2)$ and $\QQ^{(3)}$ from the above system (\ref{O-1-epsion-system})--(\ref{O-2-epsion-system-3}).
First of all, the $O(\ve^{-1})$ system in (\ref{O-1-epsion-system}) implies that $\CJ(\QQ^{(0)})=0$,
which will be taken as the biaxial global minimum with the following form:
\begin{align}\label{biaxial-global-minimum}
Q^{(0)}_i(\xx,t)=s_i\Big(\nn^2_1(\xx,t)-\frac{1}{3}\Fi\Big)+b_i\big(\nn^2_2(\xx,t)-\nn^2_3(\xx,t)\big),\quad i=1,2,
\end{align}
for some orthornomal frame $\Fp=(\nn_1,\nn_2,\nn_3)\in SO(3)$.

It can be easily observed that the system of order $O(\ve^k)(k=0,1,2)$ are not closed, since the evolution equations of the leading terms $\QQ^{(k)}(k=0,1,2)$ contain the corresponding non-leading terms $\QQ^{(k+1)}$. However, the zero-eigenvalue subspace ${\rm Ker} \CH_{\QQ^{(0)}}$ of the Hessian of the bulk energy can be utilized to cancel the non-leading terms, and thus closing the system of the leading order.
%
%
In particular, the evolution equations of the frame $\Fp=(\nn_1,\nn_2,\nn_3)$ are determined by the $O(1)$ system (\ref{O-0-epsion-system-1})--(\ref{O-0-epsion-system-3}).
To accompolish this, we take dot product with $(\CM^{(0)})^{-1}\bxi_j$ on \eqref{O-0-epsion-system-1}.
Since $\bxi_j\in {\rm Ker} \CH_{\QQ^{(0)}}$, we obtain
\begin{align*}
    0=\bxi_j\cdot\Big[\big(\CM^{(0)}\big)^{-1}\big(\dot{\QQ}^{(0)}-\CV^{(0)}\kappa^{(0)}\big)+\CG(\QQ^{(0)})\Big].
\end{align*}
By letting $\QQ^{(0)}$ take \eqref{biaxial-global-minimum}, the biaxial frame hydrodynamics is deduced.
This is exactly what has been done in \cite{LX}.

On the other hand, assume that we have derived the equations for $(\Fp,\vv)$ as above.
It implies that $\big(\CM^{(0)}\big)^{-1}\big(\dot{\QQ}^{(0)}-\CV^{(0)}\kappa^{(0)}\big)+\CG(\QQ^{(0)})$ belongs to $({\rm span}\{\bxi_1,\bxi_2,\bxi_3\})^{\perp}$.
By Assumption 1, this term belongs to $({\rm Ker} \CH_{\QQ^{(0)}})^{\perp}$.
Therefore, there exists $\QQ^{(1)}$ satisfying
\begin{align}
-\CH_{\QQ^{(0)}}\QQ^{(1)}=\big(\CM^{(0)}\big)^{-1}\big(\dot{\QQ}^{(0)}-\CV^{(0)}\kappa^{(0)}\big)+\CG(\QQ^{(0)}). \label{QQ-1-out}
\end{align}
This observation is crucial for the construction of $\QQ^{(1)}$, which we will clarify later.

\subsection{Existence of the Hilbert expansion} \label{existen-hilbert-sub}

In this subsection, we show the existence of the Hilbert expansion. In other words, we will show how to
solve $(\QQ^{(k)}, \vv^{(k)}),\,(1 \leq k \leq 2)$ and $\QQ^{(3)}$ from the system (\ref{O-1-epsion-system-1})--(\ref{O-2-epsion-system-3}) and derive the corresponding estimates.
To be more specific, we prove the following proposition.

\begin{proposition}\label{existence-Hilbert-expansion-prop}
Let $(\Fp,\vv^{(0)})$ be a smooth solution to the biaxial frame system  \eqref{frame-equation-n1}--\eqref{yuan-incompressible-v} obtained from the system \eqref{O-0-epsion-system-1}--\eqref{O-0-epsion-system-3} on $[0,T]$, satisfying
\begin{align*}
(\nabla\Fp,\vv^{(0)})\in C([0,T];H^{\ell})\quad \text{for}~\ell\geq20.
\end{align*}
Then, there exists the solution $(\QQ^{(k)},\vv^{(k)})(k=0,1,2)$ and $\QQ^{(3)}\in ({\rm Ker}\CH_{\QQ^{(0)}})^{\perp}$ of the system \eqref{O-1-epsion-system-1}--\eqref{O-2-epsion-system-3} satisfying
\begin{align*}
(\nabla\QQ^{(k)},\vv^{(k)})\in C([0,T];H^{\ell-4k}) (k=0,1,2),\quad \QQ^{(3)}\in C([0,T];H^{\ell-11}).
\end{align*}
\end{proposition}

Let us decompose $\QQ^{(1)}$ according to ${\rm Ker}\CH_{\QQ^{(0)}}$, i.e. $\QQ^{(1)}=\QQ_{\top}^{(1)}+\QQ_{\bot}^{(1)}$ with $\QQ_{\top}^{(1)}\in{\rm Ker}\CH_{\QQ^{(0)}}$ and $\QQ_{\bot}^{(1)}\in({\rm Ker}\CH_{\QQ^{(0)}})^{\perp}$.
Assume that we already have a smooth solution $(\QQ^{(0)}(\Fp),\vv)$.
Before showing Proposition \ref{existence-Hilbert-expansion-prop}, we present a lemma about the material derivative of $\QQ^{(1)}$. In what follows, $L(\cdot)$ represents a linear function with the coefficients belonging to $C([0,T];H^{\ell-1})$ and $R\in C([0,T];H^{\ell-3})$ some function depending only on $\Fp, \vv^{(0)}$ and $\QQ_{\bot}^{(1)}$.

\begin{lemma}\label{lem:proj-marterial-deriva-QQ}
It holds
\begin{align*}
\MP^{\rm out}(\dot{\QQ}^{(1)})=&L(\QQ_{\top}^{(1)})+R,\quad
\MP^{\rm in}(\dot{\QQ}^{(1)})=\dot{\QQ}_{\top}^{(1)}+L(\QQ_{\top}^{(1)})+R,
\end{align*}
where $\dot{\QQ}_{\top}^{(1)}=(\partial_t+\vv^{(0)}\cdot\nabla)\QQ_{\top}^{(1)}$.
\end{lemma}
\begin{proof}
Recall that
\begin{align*}
\bxi_k(\Fp)=\ML_k\QQ^{(0)}(\Fp),\quad k=1,2,3.
\end{align*}
Then, for $\QQ_{\top}^{(1)}\in{\rm Ker}\CH_{\QQ^{(0)}}$, we have
$
\QQ_{\top}^{(1)}=\sum\limits^3_{k=1}a_k(t)\bxi_k(\Fp),
$
from which we obtain
\begin{align*}
\dot{\QQ}_{\top}^{(1)}=&(\partial_t+\vv^{(0)}\cdot\nabla)\QQ_{\top}^{(1)}\\
=&\sum\limits^3_{k=1}a'_k(t)\bxi_k(\Fp)+\sum\limits^3_{k=1}a_k(t)(\partial_t+\vv^{(0)}\cdot\nabla)\bxi_k(\Fp),
\end{align*}
which leads to
\begin{align*}
\MP^{\rm out}(\dot{\QQ}_{\top}^{(1)})=&\sum\limits^3_{k=1}a_k(t)\MP^{\rm out}\Big((\partial_t+\vv^{(0)}\cdot\nabla)\bxi_k(\Fp)\Big)
\eqdefa L(\QQ^{(1)}_{\top}).
\end{align*}
The reason why we can regard $\MP^{\rm out}(\dot{\QQ}_{\top}^{(1)})$ as a linear term of $\QQ^{(1)}_{\top}$ is that they are both linear with respect to the coefficients $a_k(t)$.
Using $\QQ^{(1)}=\QQ_{\top}^{(1)}+\QQ_{\bot}^{(1)}$, we deduce that
\begin{align*}
\dot{\QQ}^{(1)}=&\MP^{\rm in}(\dot{\QQ}^{(1)})+\MP^{\rm out}(\dot{\QQ}^{(1)})\\
=&\MP^{\rm in}(\dot{\QQ}^{(1)}_{\top}+\dot{\QQ}^{(1)}_{\perp})+\MP^{\rm out}(\dot{\QQ}^{(1)}_{\top}+\dot{\QQ}^{(1)}_{\perp})\\
=&\MP^{\rm in}(\dot{\QQ}^{(1)}_{\top})+L(\QQ^{(1)}_{\top})+R.
\end{align*}
Thus, taking the projections $\MP^{\rm out}$ and $\MP^{\rm in}$, respectively, we have
\begin{align*}
\MP^{\rm out}(\dot{\QQ}^{(1)})=&L(\QQ_{\top}^{(1)})+R,\\
\MP^{\rm in}(\dot{\QQ}^{(1)})=&\dot{\QQ}^{(1)}-\MP^{\rm out}(\dot{\QQ}^{(1)})=\dot{\QQ}_{\top}^{(1)}+L(\QQ_{\top}^{(1)})+R.
\end{align*}
\end{proof}

{\it Proof of Proposition \ref{existence-Hilbert-expansion-prop}}.
Suppose that $(\Fp,\vv^{(0)})$ is a solution to the biaxial frame system \eqref{frame-equation-n1}--\eqref{yuan-incompressible-v} on $[0,T]$ and satisfy
\begin{align*}
(\nabla\Fp,\vv^{(0)})\in C([0,T];H^{\ell})\quad \text{for}~\ell\geq20.
\end{align*}
Since $\QQ^{(0)}=\QQ^{(0)}(\Fp)$ is a function of the frame $\Fp$ and takes the form (\ref{biaxial-global-minimum}), it follows that $\QQ^{(0)}\in C([0,T];H^{\ell+1})$.

We solve $\QQ_{\bot}^{(1)}$ by rewriting \eqref{QQ-1-out} as
\begin{align*}
\QQ_{\bot}^{(1)}=-\CH^{-1}_{\QQ^{(0)}}\Big(\big(\CM^{(0)}\big)^{-1}\big(\dot{\QQ}^{(0)}-\CV^{(0)}\kappa^{(0)}\big)+\CG(\QQ^{(0)})\Big)\in C([0,T];H^{\ell-1}).
\end{align*}
Here, the inverse $\CH^{-1}_{\QQ^{(0)}}$ is well-defined within $({\rm Ker}\CH_{\QQ^{(0)}})^{\perp}$ due to Proposition \ref{prop-linearized-operator}.
Thus, the existence of $(\QQ^{(1)}, \vv^{(1)})$ can be reduced to solving $(\QQ^{(1)}_{\top}, \vv^{(1)})$.
The key observation is that $(\QQ^{(1)}_{\top}, \vv^{(1)})$ satisfies a linear dissipative system, which we derive below, although the system seems nonlinear at first glance due to the term $\JJ_1$.

Note that $\CM^{(1)},\CV^{(1)}$ are linear functions of $\QQ^{(1)}$. Consequently, using \eqref{QQ-1-out}, the term $\FF_1$ can be expressed by
\begin{align*}
\FF_1(\QQ^{(1)})=L(\QQ^{(1)}_{\top},\vv^{(1)})+R.
\end{align*}
Since $\JJ_1$ is a quadratic function of $\QQ^{(1)}$ and $\QQ^{(1)}=\QQ_{\top}^{(1)}+\QQ_{\bot}^{(1)}$, there has
\begin{align*}
\JJ_1(\QQ^{(1)})=\JJ_1(\QQ^{(1)}_{\top})+L(\QQ^{(1)}_{\top})+R.
\end{align*}
We claim that $\JJ_1(\QQ^{(1)}_{\top})\in({\rm Ker}\CH_{\QQ^{(0)}})^{\perp}$. In fact, it suffices to prove that
\begin{align*}
  \JJ_1(\QQ^{(1)}_{\top})\cdot\ML_{\alpha}\QQ^{(0)}=0,\quad\alpha=1,2,3.
\end{align*}
By the definition of $\JJ_1$ and $\QQ_{\top}^{(1)}\in{\rm Ker}\CH_{\QQ^{(0)}}$, we deduce that
\begin{align*}
\JJ_1(\QQ^{(1)}_{\top})\cdot\ML_{\alpha}\QQ^{(0)}=&\frac{1}{2}\CJ''(\QQ^{(0)})_{ijklpq}(\QQ^{(1)}_{\top})_{kl}(\QQ^{(1)}_{\top})_{pq}\ML_{\alpha}\QQ^{(0)}_{ij}\\
=&\frac{1}{2}\ML_{\alpha}\big[\CJ'(\QQ^{(0)})_{klpq}(\QQ^{(1)}_{\top})_{kl}\big](\QQ^{(1)}_{\top})_{pq}-\frac{1}{2}\CJ'(\QQ^{(0)})_{klpq}\ML_{\alpha}(\QQ^{(1)}_{\top})_{kl}(\QQ^{(1)}_{\top})_{pq}\\
=&0.
\end{align*}

We are now in a position to derive the system for $(\QQ^{(1)}_{\top},\vv^{(1)})$.
By observing the equation (\ref{O-1-epsion-system-1}), we find that $\CH_{\QQ^{(0)}}\QQ^{(2)}+\JJ_1(\QQ^{(1)}_{\top})\in ({\rm Ker}\CH_{\QQ^{(0)}})^{\perp}$.
This motivates us to take use of the projection operator $\MP^{\rm in}$.
However, we can not cancel non-leading terms by taking the projection $\MP^{\rm in}$ on the equation (\ref{O-1-epsion-system-1}) directly because of $\MP^{\rm in}\CM^{(0)}\neq\CM^{(0)}\MP^{\rm in}$.
Therefore, to obtain the closed linear system of leading terms, we need to investigate the intrinsic structure of the equations (\ref{O-1-epsion-system-1})--(\ref{O-1-epsion-system-3}).

Let us denote
\begin{align*}
&\dot{\QQ}^{(1)}_{\top}=(\partial_t+\vv^{(0)}\cdot\nabla)\QQ^{(1)}_{\top},\quad \dot{\vv}^{(1)}=(\partial_t+\vv^{(0)}\cdot\nabla)\vv^{(1)},\\
&\MA_1=\MP^{\rm in}(\CG(\QQ^{(1)}_{\top})),\quad \MB_1=\MP^{\rm in}\big[(\CM^{(0)})^{-1}\CV^{(0)}\kappa^{(1)}\big],\\
&\MA_2=\MP^{\rm out}(\CG(\QQ^{(1)}_{\top})),\quad \MB_2=\MP^{\rm out}\big[(\CM^{(0)})^{-1}\CV^{(0)}\kappa^{(1)}\big].
\end{align*}
Acting $(\CM^{(0)})^{-1}$ on (\ref{O-1-epsion-system-1}), we derive that
\begin{align}\label{ker-QQ1-eqaution-first}
(\CM^{(0)})^{-1}\dot{\QQ}^{(1)}_{\top}=&-\big(\CH_{\QQ^{(0)}}\QQ^{(2)}+\CG(\QQ^{(1)}_{\top})+\JJ_1(\QQ^{(1)}_{\top})\big)+(\CM^{(0)})^{-1}\CV^{(0)}\kappa^{(1)}\nonumber\\
&+L(\QQ^{(1)}_{\top},\vv^{(1)})+R.
\end{align}
Now we impose the projection $\MP^{\rm in}$ on (\ref{ker-QQ1-eqaution-first}).
Since $\CH_{\QQ^{(0)}}\QQ^{(2)}+\JJ_1(\QQ^{(1)}_{\top})\in ({\rm Ker}\CH_{\QQ^{(0)}})^{\perp}$, we derive from Lemma \ref{lem:proj-marterial-deriva-QQ} that
\begin{align}\label{dotQQ-top1-equ}
\MP^{\rm in}\dot{\QQ}^{(1)}_{\top}=\MP^{\rm in}(\widetilde{\CM}^{(0)})^{-1}(-\MA_1+\MB_1)+L(\QQ^{(1)}_{\top},\vv^{(1)})+R,
\end{align}
where $\widetilde{\CM}^{(0)}\eqdefa\MP^{\rm in}(\CM^{(0)})^{-1}\MP^{\rm in}$ is symmetric positive definite.
On the other hand, we impose the projection $\MP^{\rm out}$ on (\ref{ker-QQ1-eqaution-first}) to obtain
\begin{align*}
\MP^{\rm out}\big((\CM^{(0)})^{-1}\dot{\QQ}^{(1)}_{\top}\big)=-\big(\CH_{\QQ^{(0)}}\QQ^{(2)}+\JJ_1(\QQ^{(1)}_{\top})\big)-\MA_2+\MB_2+L(\QQ^{(1)}_{\top},\vv^{(1)})+R,
\end{align*}
which, together with $\MP^{\rm out}(\dot{\QQ}^{(1)}_{\top})=L(\QQ^{(1)}_{\top})$ and (\ref{dotQQ-top1-equ}), implies that
\begin{align}\label{CH-QQ0+CG+JJ1}
\widetilde{\mu}_{\QQ^{(1)}_{\top}}&\eqdefa\CH_{\QQ^{(0)}}\QQ^{(2)}+\CG(\QQ^{(1)}_{\top})+\JJ_1(\QQ^{(1)}_{\top})\nonumber\\
&=\MA_1+\MB_2-\MP^{\rm out}(\CM^{(0)})^{-1}\MP^{\rm in}\dot{\QQ}^{(1)}_{\top}+L(\QQ^{(1)}_{\top},\vv^{(1)})+R\nonumber\\
&=\MA_1+\MB_2-\MP^{\rm out}(\CM^{(0)})^{-1}\MP^{\rm in}(\widetilde{\CM}^{(0)})^{-1}(-\MA_1+\MB_1)+L(\QQ^{(1)}_{\top},\vv^{(1)})+R.
\end{align}

Thus, substituting (\ref{CH-QQ0+CG+JJ1}) into the equation (\ref{O-1-epsion-system-2}), together with (\ref{dotQQ-top1-equ}), we obtain the following closed linear system for $(\QQ^{(1)}_{\top},\vv^{(1)})$:
\begin{align}
\dot{\QQ}^{(1)}_{\top}=&-\CM^{(0)}\widetilde{\mu}_{\QQ^{(1)}_{\top}}+\CV^{(0)}\kappa^{(1)}+L(\QQ^{(1)}_{\top},\vv^{(1)})+R,\label{QQ1-top-equation-1}\\
\dot{v}^{(1)}_i=&-\partial_ip^{(1)}+\eta\Delta v^{(1)}_i+\partial_j\big(\CP^{(0)}\kappa^{(1)}\big)_{ij}\nonumber\\
&+\partial_j\Big(\CN^{(0)}\big(\widetilde{\mu}_{\QQ^{(1)}_{\top}}+L(\QQ^{(1)}_{\top},\vv^{(1)})+R\big)\Big)_{ij}\nonumber\\
&+\widetilde{\mu}_{\QQ^{(1)}_{\top}}\cdot\partial_i\QQ^{(0)}+L(\QQ^{(1)}_{\top},\vv^{(1)})+R,\label{vv1-ker-equation-2}\\
\nabla\cdot\vv^{(1)}=&~0.\label{vv1-incomp-equation-3}
\end{align}

In order to prove the unique solvability of the linear system (\ref{QQ1-top-equation-1})--(\ref{vv1-incomp-equation-3}), we establish an a priori estimate for the energy
\begin{align*}
\CE_{\ell}(t)\eqdefa\sum^{\ell-4}_{k=0}\Big(\|\partial^k\vv^{(1)}\|^2_{L^2}+\big\langle\partial^k\QQ^{(1)}_{\top},\CG(\partial^k\QQ^{(1)}_{\top})\big\rangle\Big)+\|\QQ^{(1)}_{\top}\|^2_{L^2}.
\end{align*}
Specifically, we show that there exists a positive constant $C$ such that
\begin{align}
\frac{\ud}{\ud t}\CE_{\ell}(t)\leq C\big(\CE_{\ell}(t)+\|R(t)\|_{H^{\ell-3}}\big),
\end{align}
where the solution $(\QQ^{(1)}_{\top},\vv^{(1)})$ satisfies $(\nabla\QQ^{(1)}_{\top},\vv^{(1)})\in C([0,T];H^{\ell-4})$.
It suffices to prove the case of $k=0$ due to similar proof for the general case. The corresponding energy functional is defined as
\begin{align*}
\CE_1(t)=\|\vv^{(1)}\|^2_{L^2}+\|\QQ^{(1)}_{\top}\|^2_{L^2}+\langle\QQ^{(1)}_{\top},\CG(\QQ^{(1)}_{\top})\rangle.
\end{align*}

To begin with, it follows from (\ref{QQ1-top-equation-1}) and (\ref{CH-QQ0+CG+JJ1}) that
\begin{align}\label{QQ1-top-L2}
\frac{1}{2}\frac{\ud}{\ud t}\|\QQ^{(1)}_{\top}\|^2_{L^2}=&\big\langle-\CM^{(0)}\widetilde{\mu}_{\QQ^{(1)}_{\top}}+\CV^{(0)}\kappa^{(1)},\QQ^{(1)}_{\top}\big\rangle+\langle L(\QQ^{(1)}_{\top},\vv^{(1)})+R,\QQ^{(1)}_{\top}\rangle\nonumber\\
\leq&\delta\|\nabla\vv^{(1)}\|^2_{L^2}+C_{\delta}\|\QQ^{(1)}_{\top}\|^2_{H^1}+C(\|\vv^{(1)}\|^2_{L^2}+\|R\|^2_{L^2}).
\end{align}
Taking the inner product on (\ref{QQ1-top-equation-1}) with $\widetilde{\mu}_{\QQ^{(1)}_{\top}}$, and on (\ref{vv1-ker-equation-2}) with $\vv^{(1)}$, respectively, we deduce that
\begin{align}\label{vv1-QQ1-deriv-L2}
&\frac{\ud}{\ud t}\Big(\langle\QQ^{(1)}_{\top},\CG(\QQ^{(1)}_{\top})\rangle+\frac{1}{2}\|\vv^{(1)}\|^2_{L^2}\Big)=\langle\partial_t\QQ^{(1)}_{\top},\CG(\QQ^{(1)}_{\top})\rangle+\langle\partial_t\vv^{(1)},\vv^{(1)}\rangle\nonumber\\
&\quad=\underbrace{-\big\langle\dot{\QQ}^{(1)}_{\top},\CH_{\QQ^{(0)}}\QQ^{(2)}+\JJ_1(\QQ^{(1)}_{\top})\big\rangle}_{J_1}\underbrace{-\langle\vv^{(0)}\cdot\nabla\QQ^{(1)}_{\top},\CG(\QQ^{(1)}_{\top})\rangle}_{J_2}\nonumber\\
&\qquad-\big\langle\CM^{(0)}\widetilde{\mu}_{\QQ^{(1)}_{\top}},\widetilde{\mu}_{\QQ^{(1)}_{\top}}\big\rangle+\underbrace{\big\langle\CV^{(0)}\kappa^{(1)},\widetilde{\mu}_{\QQ^{(1)}_{\top}}\big\rangle}_{J_3}+\underbrace{\big\langle L(\QQ^{(1)}_{\top},\vv^{(1)})+R,\widetilde{\mu}_{\QQ^{(1)}_{\top}}\big\rangle}_{J_4}\nonumber\\
&\qquad-\eta\|\nabla\vv^{(1)}\|^2_{L^2}-\langle\CP^{(0)}\kappa^{(1)},\kappa^{(1)}\rangle\underbrace{-\big\langle\CN^{(0)}\widetilde{\mu}_{\QQ^{(1)}_{\top}},\kappa^{(1)}\big\rangle}_{J_5}\nonumber\\
&\qquad\underbrace{-\big\langle \CN^{(0)}\big(L(\QQ^{(1)}_{\top},\vv^{(1)})+R\big),\kappa^{(1)}\big\rangle+\langle\widetilde{\mu}_{\QQ^{(1)}_{\top}}\cdot\partial_i\QQ^{(0)},\vv^{(1)}\rangle}_{J_6}\nonumber\\
&\qquad+\underbrace{\big\langle L(\QQ^{(1)}_{\top},\vv^{(1)})+R,\vv^{(1)}\big\rangle}_{J_7}.
\end{align}
It can be easily seen that $J_3+J_5=0$.
Using the definition of $\widetilde{\mu}_{\QQ^{(1)}_{\top}}$, we obtain
\begin{align*}
J_2+J_4+J_6+J_7\leq\delta\|\nabla\vv^{(1)}\|^2_{L^2}+C_{\delta}(\|\vv^{(1)}\|^2_{L^2}+\|\QQ^{(1)}_{\top}\|^2_{H^1}+\|R\|^2_{H^1}).
\end{align*}
where we have used the following fact: for any $\QQ=(Q_1,Q_2)^T\in\mathbb{Q}$, it holds
\begin{align*}
&-\langle\vv^{(0)}\cdot\nabla\QQ,\CG(\QQ)\rangle\\
&\quad=-\int_{\mathbb{R}^3}v^{(0)}_{\alpha}\partial_{\alpha}\QQ_{kl}\Big(D_1\Delta\QQ_{kl}+D_2\MS(\partial_k\partial_m\QQ_{lm})\Big)\ud\xx\\
&\quad=\int_{\mathbb{R}^3}\Big(-\partial_mv^{(0)}_{\alpha}\partial_{\alpha}\QQ_{jk}D_1\partial_m\QQ_{jk}-D_2\big(\partial_lv^{(0)}_{\alpha}\partial_{\alpha}\QQ_{kl}\partial_m\QQ_{km}+\partial_kv^{(0)}_{\alpha}\partial_{\alpha}\QQ_{kl}\partial_m\QQ_{lm}\big)\Big)\ud\xx\\
&\quad\leq C\|\QQ\|^2_{H^1}.
\end{align*}
It remains to estimate the term $J_1$. Noting that $\CH_{\QQ^{(0)}}\QQ^{(2)}+\JJ_1(\QQ^{(1)}_{\top})\in ({\rm Ker}\CH_{\QQ^{(0)}})^{\perp}$ and using Lemma \ref{lem:proj-marterial-deriva-QQ}, we derive that
\begin{align*}
J_1=&-\big\langle\MP^{\rm out}(\dot{\QQ}^{(1)}_{\top}),\CH_{\QQ^{(0)}}\QQ^{(2)}+\JJ_1(\QQ^{(1)}_{\top})\big\rangle\\
=&-\big\langle L(\QQ^{(1)}_{\top}),\widetilde{\mu}_{\QQ^{(1)}_{\top}}-\CG(\QQ^{(1)}_{\top})\big\rangle\\
\leq&\delta\|\nabla\vv^{(1)}\|^2_{L^2}+C_{\delta}(\|\vv^{(1)}\|^2_{L^2}+\|\QQ^{(1)}_{\top}\|^2_{H^1}+\|R\|^2_{L^2}).
\end{align*}
Thus, plugging the above terms $J_i(i=1,\cdots,7)$ into (\ref{vv1-QQ1-deriv-L2}) and using (\ref{QQ1-top-L2}), we arrive at
\begin{align*}
\frac{\ud}{\ud t}\CE_1(t)\leq C(\CE_1(t)+\|R\|^2_{H^1}),
\end{align*}
which implies the existence of the solution $(\QQ^{(1)}_{\top},\vv^{(1)})$.

Hence, the solution $(\QQ^{(1)},\vv^{(1)})$ can be uniquely
determined. In a similar argument, we can solve $(\QQ^{(2)},\vv^{(2)})$ and $\QQ^{(3)}$ by the system (\ref{O-2-epsion-system-1})--(\ref{O-2-epsion-system-3}).
Here we omit the details.

\subsection{System for the remainder and uniform estimates}\label{remainder-subsection}

This subsection will be devoted to deriving the remainder system and the uniform estimates for the remainder.
Proposition \ref{existence-Hilbert-expansion-prop} tells us that $(\nabla \QQ^{(k)},\vv^{(k)})\in C([0, T]; H^{\ell-4k})$ for $k=0,1,2$ and $\QQ^{(3)}\in C([0, T]; H^{\ell-11})$.
Hence, in what follows, $\QQ^{(k)}$ and $\vv^{(k)}$ will be treated as known functions.
We denote by $C$ a constant depending on $\sum ^2_{k=0}\sup_{t\in [0,T]}\|\vv^{(k)}(t)\|_{H^{\ell-4k}}$ and
$\sum^3_{k=0}\sup_{t\in [0,T]}\| \QQ^{(k)}(t)\|_{H^{\ell+1-4k}}$, and independent of $\ve$.

Recall the definitions (\ref{QQ-Hilbert-expansion}) and (\ref{vv-Hilbert-expansion}), the remainder is written as
\begin{align}\label{QR-vR}
\QQ_R=\frac{1}{\ve}(\QQ^{\ve}-\widetilde{\QQ}),\quad \vv_R=\frac{1}{\ve}(\vv^{\ve}-\widetilde{\vv}),
\end{align}
where $\QQ_R$ and $\vv_R$ depend on $\ve$.
In order to derive the evolution equations of the
remainder of $(\QQ_R,\vv_R)$, we express the system of $(\QQ^{\ve},\vv^{\ve})$ in the abstract form below:
\begin{align}
\partial_t\QQ^{\ve}=&-\frac{1}{\ve}\CM_{\QQ^{\ve}}\CJ(\QQ^{\ve})+\VV(\QQ^{\ve},\vv^{\ve}),\label{abstract-equ-1}\\
\partial_tv^{\ve}_i=&\mathbb{P}_{\rm div}\Big[\partial_j\Big(\frac{1}{\ve}\CN_{\QQ^{\ve}}\CJ(\QQ^{\ve})+\NN(\QQ^{\ve},\vv^{\ve})\Big)_{ij}+\KK(\QQ^{\ve})_i\Big],
\end{align}
where $\mathbb{P}_{\rm div}$ is a projection operator mapping a vector field into its solenoidal part, and $\VV,\NN, \KK$ are given by
\begin{align*}
&\VV(\QQ,\vv)=-\CM_{\QQ}\CG(\QQ)+\CV_{\QQ}\kappa-\vv\cdot\nabla\QQ,\quad \KK(\QQ)_i=\mu_{\QQ}\cdot\partial_i\QQ,\\
&\NN(\QQ,\vv)=\CN_{\QQ}\CG(\QQ)-\vv\otimes\vv+2\eta\A+\CP_{\QQ}\kappa.
\end{align*}
Consequently, we deduce that
\begin{align}
\partial_t\QQ_R=&-\frac{1}{\ve^4}\Big(\CM_{\QQ^{\ve}}\CJ(\QQ^{\ve})-\CM_{\widetilde{\QQ}}\CJ(\widetilde{\QQ})\Big)+\frac{1}{\ve^3}\Big(\VV(\QQ^{\ve},\vv^{\ve})-\VV(\widetilde{\QQ},\widetilde{\vv})\Big)+\RR_1(\widetilde{\QQ},\widetilde{\vv}),\label{QQR-equation-chu}\\
(\partial_t\vv_R)_i=&~\mathbb{P}_{\rm div}\partial_j\Big[\frac{1}{\ve^4}\Big(\CN_{\QQ^{\ve}}\CJ(\QQ^{\ve})-\CN_{\widetilde{\QQ}}\CJ(\widetilde{\QQ})\Big)+\frac{1}{\ve^3}\Big(\NN(\QQ^{\ve},\vv^{\ve})-\NN(\widetilde{\QQ},\widetilde{\vv})\Big)\Big]_{ij}\nonumber\\
&+\mathbb{P}_{\rm div}\frac{1}{\ve^3}\Big(\KK(\QQ^{\ve})-\KK(\widetilde{\QQ})\Big)_i+\RR_2(\widetilde{\QQ},\widetilde{\vv}),\label{vvR-equation-chu}
\end{align}
where $\RR_i(\widetilde{\QQ},\widetilde{\vv})(i=1,2)$ are expressed by
\begin{align*}
\RR_1(\widetilde{\QQ},\widetilde{\vv})=&\frac{1}{\ve^3}\Big(-\frac{1}{\ve}\CM_{\widetilde{\QQ}}\CJ(\widetilde{\QQ})+\VV(\widetilde{\QQ},\widetilde{\vv})-\partial_t\widetilde{\QQ}\Big),\\
\RR_2(\widetilde{\QQ},\widetilde{\vv})=&\mathbb{P}_{\rm div}\frac{1}{\ve^3}\Big[\partial_j\Big(\frac{1}{\ve}\CN_{\widetilde{\QQ}}\CJ(\widetilde{\QQ})+\NN(\widetilde{\QQ},\widetilde{\vv})\Big)_{ij}+\KK(\widetilde{\QQ})-\partial_t\widetilde{v}_i\Big].
\end{align*}

There is no doubt that this is a rather tedious task if we want to precisely express the right-hand
terms of the system (\ref{QQR-equation-chu})--(\ref{vvR-equation-chu}) due to its high nonlinearity derived from the closure approximation.
To simplify the presentation, we introduce a notation $\mathfrak{R}$, called ${\it good~ terms}$, to stand for terms satisfying
\begin{align}\label{good-terms}
\|\mathfrak{R}\|_{L^2}+\ve\|\nabla\mathfrak{R}\|_{L^2}+\ve^2\|\Delta\mathfrak{R}\|_{L^2} \leq C(\ve E)(1+E+\ve F)+\ve f(E),
\end{align}
where $C(\cdot)$ and $f(\cdot): \mathbb{R}^+\cup\{0\}\rightarrow\mathbb{R}^+\cup\{0\}$ are increasing functions.
They may depend on $\|\QQ^{(k)}\|(k=0,\cdots,3)$ and parameters of the system but is independent of $\ve$. Here, $E$ and $F$ in (\ref{good-terms}) are defined as, respectively,
\begin{align}
E\eqdefa&\|\QQ_R\|_{H^1}+\ve\|\Delta\QQ_R\|_{L^2}+\ve^2\|\nabla\Delta\QQ_R\|_{L^2}+\|\vv_R\|_{L^2}+\ve\|\nabla\vv_R\|_{L^2}+\ve^2\|\Delta\vv_R\|_{L^2},\label{E-norm}\\
F\eqdefa&\ve\|\nabla\CG(\QQ_R)\|_{L^2}+\ve^2\|\Delta\CG(\QQ_R)\|_{L^2}+\ve^2\|\Delta\nabla\vv_R\|_{L^2}.\label{F-norm}
\end{align}
Let us give some examples of good terms that can be absorbed into $\mathfrak{R}$.
Using the Sobolev embedding theorem, for $k=0,1,2$ and some positive constant $C$, it follows that
\begin{align*}
&\ve^k\|\QQ_R\|_{H^k}+\ve^k\|\vv_R\|_{H^k}\leq E,\quad \ve\|\QQ_R\|_{L^{\infty}}+\ve^2\|\vv_R\|_{L^{\infty}}\leq CE,\\
&\ve^{k+1}\|\CG(\QQ_R)\|_{H^k}+\ve^3\|\nabla\vv_R\|_{L^{\infty}}\leq C(E+\ve F).
\end{align*}
In addition, since $\|\QQ^{(0)}-\QQ^{*}\|_{H^k}, \|\QQ^{(i)}\|_{H^k}(k\leq3,1\leq i\leq3)$ can be all controlled by a constant independent of $\ve$, there holds
\begin{align*}
\|\QQ^{\ve}-\QQ^{*}\|_{H^k}\leq C+\ve^3\|\QQ_R\|_{H^k}\leq C(\ve E),\quad \|\vv^{\ve}\|_{H^k}\leq C(\ve E).
\end{align*}
It should be emphasized that
the key feature of the good terms $\mathfrak{R}$ lies in the right-hand side being controlled
by $C(1+E)$ as $\ve\rightarrow0$. Thus, we can deduce a closed energy estimate uniformly in $\ve$ (see
Proposition \ref{uniform-bound-prop} below).

Armed with the definition of good term $\mathfrak{R}$, let us derive the right-hand terms of the system (\ref{QQR-equation-chu})--(\ref{vvR-equation-chu}).
First of all, by means of the choices of $\QQ^{(k)}(0\leq k\leq3), \vv^{(l)}(0\leq l\leq 2)$ by its regularity in Proposition \ref{existence-Hilbert-expansion-prop}, it can be seen that $\|\RR_k(\widetilde{\QQ},\widetilde{\vv})\|_{H^2}(k=1,2)$ are all controlled by a constant uniformly in $\ve$, thus can be absorbed into
the good terms $\mathfrak{R}$.

For the remaining terms in the
system (\ref{QQR-equation-chu})--(\ref{vvR-equation-chu}), we have the following two lemmas:

\begin{lemma}\label{singular-remainder-le}
It holds
\begin{align}
\CM_{\QQ^{\ve}}\CJ(\QQ^{\ve})-\CM_{\widetilde{\QQ}}\CJ(\widetilde{\QQ})=&\ve^3\CM_{\QQ^{(0)}}(\CH_{\QQ^{(0)}}\QQ_R)+\ve^4\mathfrak{R},\label{CMQQve-CMwidQQ}\\
\CN_{\QQ^{\ve}}\CJ(\QQ^{\ve})-\CN_{\widetilde{\QQ}}\CJ(\widetilde{\QQ})=&\ve^3\CN_{\QQ^{(0)}}(\CH_{\QQ^{(0)}}\QQ_R)+\ve^4\mathfrak{R}.\label{CNQQve-CNwidQQ}
\end{align}
\end{lemma}

\begin{proof}
By the Taylor expansion, for any $\theta\in(0,1)$, we have
\begin{align*}
&\|\CJ(\QQ^{\ve})-\CJ(\widetilde{\QQ})-\ve^3\CJ'(\widetilde{\QQ})\QQ_R\|_{H^k}\\
&\quad=\Big\|\ve^6\int^1_0\theta\CJ''(\widetilde{\QQ}+\theta\ve^3\QQ_R)\QQ_R\cdot\QQ_R\ud\theta\Big\|_{H^k}\\
&\quad\leq\ve^6C(\ve^3\|\QQ_R\|_{H^2})\|\QQ_R\|_{H^k}\|\QQ_R\|_{H^2}\leq \ve^5C(\ve E)E\|\QQ_R\|_{H^k},
\end{align*}
which implies
$\CJ(\QQ^{\ve})-\CJ(\widetilde{\QQ})-\ve^3\CH_{\widetilde{\QQ}}\QQ_R=\ve^4\mathfrak{R}$. Consequently, we derive that
\begin{align*}
\CM_{\widetilde{\QQ}}\CJ(\QQ^{\ve})-\CM_{\widetilde{\QQ}}\CJ(\widetilde{\QQ})=&\ve^3\CM_{\widetilde{\QQ}}\big(\CH_{\widetilde{\QQ}}\QQ_R\big)+\ve^4\mathfrak{R}\\
=&\ve^3\CM_{\QQ^{(0)}}\big(\CH_{\QQ^{(0)}}\QQ_R\big)+\ve^4\mathfrak{R}.
\end{align*}
On the one hand, we infer from Lemma \ref{defference-YY-UU-lemma} that
\begin{align*}
&\|\CM_{\QQ^{\ve}}\CJ(\widetilde{\QQ})-\CM_{\widetilde{\QQ}}\CJ(\widetilde{\QQ})\|_{H^k}\\
&\quad\leq C(\|\ve^3\QQ_R\|_{H^2})\|\ve^3\QQ_R\|_{H^1}\big\|\CJ(\QQ^{\ve})-\CJ(\widetilde{\QQ})\big\|_{H^2}\\
&\quad\leq C(\|\ve^3\QQ_R\|_{H^2})\ve^4\|\QQ_R\|_{H^k}(1+\|\ve^2\QQ_R\|_{H^2})\leq \ve^4C(\ve E)\|\QQ_R\|_{H^k}(1+\ve E),
\end{align*}
which implies $\CM_{\QQ^{\ve}}\CJ(\widetilde{\QQ})-\CM_{\widetilde{\QQ}}\CJ(\widetilde{\QQ})\in\mathfrak{R}$. The other identity (\ref{CNQQve-CNwidQQ}) is obtained with a similar argument.
\end{proof}

\begin{lemma}\label{VVNNKK-remainder-le}
For the terms $\VV,\NN$ and $\KK$, it holds
\begin{align}
\VV(\QQ^{\ve},\vv^{\ve})-\VV(\widetilde{\QQ},\widetilde{\vv})=&\ve^3\big(-\CM_{\QQ^{(0)}}\CG(\QQ_R)+\CV_{\QQ^{(0)}}\kappa_R\big)+\ve^3\mathfrak{R},\label{VV-remainder}\\
\NN(\QQ^{\ve},\vv^{\ve})-\NN(\widetilde{\QQ},\widetilde{\vv})=&\ve^3\big(\CN_{\QQ^{(0)}}\CG(\QQ_R)+2\eta\A_R+\CP_{\QQ^{(0)}}\kappa_R\big)+\ve^3\mathfrak{R},\label{NN-remainder}\\
\KK(\QQ^{\ve})-\KK(\widetilde{\QQ})=&\ve^3\nabla\cdot\mathfrak{R},\label{KK-remainder}
\end{align}
where $\kappa_R=(\nabla\vv_R)^T$ and $\A_R=\frac{1}{2}(\kappa_R+\kappa^T_R)$.
\end{lemma}

\begin{proof}
To begin with, for $0\leq k\leq 2$, we have
\begin{align*}
\ve^k\|\vv^{\ve}\cdot\nabla\QQ^{\ve}-\widetilde{\vv}\cdot\nabla\widetilde{\QQ}\|_{H^k}\leq& \ve^k\|\vv^{\ve}\cdot\nabla\QQ^{\ve}-\vv^{\ve}\cdot\nabla\widetilde{\QQ}\|_{H^k}+\ve^k\|\vv^{\ve}\cdot\nabla\widetilde{\QQ}-\widetilde{\vv}\cdot\nabla\widetilde{\QQ}\|_{H^k}\\
\leq&\ve^3\|\ve^k\QQ_R\|_{H^k}\|\vv^{\ve}\|_{L^{\infty}}+\ve^3\|\ve^k\vv_R\|_{H^k}\|\nabla\widetilde{\QQ}\|_{L^{\infty}}\\
\leq&\ve^3C(1+\ve E)E.
\end{align*}
According to (\ref{QR-vR}), we write
\begin{align*}
\CM_{\QQ^{\ve}}\CG(\QQ^{\ve})-\CM_{\widetilde{\QQ}}\CG(\widetilde{\QQ})=&\underbrace{\CM_{\QQ^{\ve}}\CG(\widetilde{\QQ})-\CM_{\widetilde{\QQ}}\CG(\widetilde{\QQ})}_{\MM_1}+\underbrace{\ve^3\big(\CM_{\QQ^{\ve}}\CG(\QQ_R)-\CM_{\widetilde{\QQ}}\CG(\QQ_R)\big)}_{\MM_2}\\
&+\underbrace{\ve^3\big(\CM_{\widetilde{\QQ}}\CG(\QQ_R)-\CM_{\QQ^{(0)}}\CG(\QQ_R)\big)}_{\MM_3}+\ve^3\CM_{\QQ^{(0)}}\CG(\QQ_R).
\end{align*}
Then, for $0\leq k\leq 2$, we derive from Lemma \ref{defference-YY-UU-lemma} that
\begin{align*}
\|\MM_1\|_{H^k}\leq& C(\|\QQ^{\ve}\|_{L^{\infty}},\|\widetilde{\QQ}\|_{L^{\infty}},\|\QQ^{\ve}-\QQ^*\|_{H^k},\|\widetilde{\QQ}-\QQ^*\|_{H^k})\|\ve^3\QQ_R\|_{H^k}\|\CG(\widetilde{\QQ})\|_{H^{k+2}},\\
\|\MM_2\|_{H^k}\leq&\ve^3C(\|\QQ^{\ve}\|_{L^{\infty}},\|\widetilde{\QQ}\|_{L^{\infty}})\|\CG(\QQ_R)\|_{H^k}\|\ve^3\QQ_R\|_{L^{\infty}}\\
&+\ve^3C(\|\QQ^{\ve}\|_{L^{\infty}},\|\widetilde{\QQ}\|_{L^{\infty}},\|\QQ^{\ve}-\QQ^*\|_{H^k},\|\widetilde{\QQ}-\QQ^*\|_{H^k})\|\CG(\QQ_R)\|_{L^{\infty}}\|\ve^3\QQ_R\|_{H^k},\\
\|\MM_3\|_{H^k}\leq &\ve^3C\big(\|\QQ^{(0)}\|_{L^{\infty}},\|\widetilde{\QQ}\|_{L^{\infty}},\|\QQ^{(0)}-\QQ^*\|_{H^{k+2}},\|\widetilde{\QQ}-\QQ^*\|_{H^{k+2}}\big)\\
&\cdot\|\widetilde{\QQ}-\QQ^{(0)}\|_{H^{k+2}}\|\CG(\QQ_R)\|_{H^k}.
\end{align*}
Noting the following simple fact:
\begin{align*}
\|\widetilde{\QQ}-\QQ^{(0)}\|_{H^{k+2}}=\ve\|\QQ^{(1)}+\ve\QQ^{(2)}+\ve^2\QQ^{(3)}\|_{H^{k+2}}\leq CE,
\end{align*}
we deduce that
\begin{align*}
&\ve^k\|\MM_1+\MM_2\|_{H^k}\leq C(\ve E)\ve^3E+C(\ve E)\ve^4E(E+F),\\
&\ve^{k}\|\MM_3\|_{H^k}\leq \ve^3C(E+\ve F).
\end{align*}
Thus, we obtain
\begin{align*}
-\big(\CM_{\QQ^{\ve}}\CG(\QQ^{\ve})-\CM_{\widetilde{\QQ}}\CG(\widetilde{\QQ})\big)=-\ve^3\CM_{\QQ^{(0)}}\CG(\QQ_R)+\ve^3\mathfrak{R}.
\end{align*}
Those terms containing the operators $\CV_{\QQ}, \CN_{\QQ}$ and $\CP_{\QQ}$ in (\ref{VV-remainder})--(\ref{NN-remainder}) can be handled in similar arguments, since these operators share the same properties with $\CM_{\QQ}$. Hence, we have shown (\ref{VV-remainder}) and (\ref{NN-remainder}).

Recalling (\ref{additional-pressure-term}), we have
\begin{align*}
\KK(\QQ)_i=\partial_j\Sigma_{ij}(\QQ,\QQ)-\partial_i\widetilde{p},\quad \Sigma_{ij}(\QQ,\widehat{\QQ})\eqdefa-\frac{\partial F_e(\nabla\QQ)}{\partial(\partial_j\QQ)}\cdot\partial_i\widehat{\QQ}.
\end{align*}
We complete the proof of the lemma by
\begin{align*}
\ve^k\|\Sigma(\QQ^{\ve},\QQ^{\ve})-\Sigma(\widetilde{\QQ},\widetilde{\QQ})\|_{H^k}=&\ve^{k+3}\|\Sigma(\QQ^{\ve},\QQ_R)+\Sigma(\QQ_R,\widetilde{\QQ})\|_{H^k}\\
\leq&\ve^3C\|\ve^k\nabla\QQ_R\|_{H^k}(1+\|\ve^3\nabla\QQ_R\|_{L^{\infty}})\\
\leq&\ve^3C(1+\ve E)E.
\end{align*}
\end{proof}

Using Lemma \ref{singular-remainder-le} and Lemma \ref{VVNNKK-remainder-le}, we arrive at the following remainder
system:
\begin{align}
\partial_t\QQ_R=&-\CM_{\QQ^{(0)}}\Big(\frac{1}{\ve}\CH_{\QQ^{(0)}}\QQ_R+\CG(\QQ_R)\Big)+\CV_{\QQ^{(0)}}\kappa_R+\mathfrak{R},\label{QvR-remainder-equation-1}\\
\partial_t\vv_R=&-\nabla p_R+\eta\Delta\vv_R+\nabla\cdot(\CP_{\QQ^{(0)}}\kappa_R)\nonumber\\
&+\nabla\cdot\Big(\CN_{\QQ^{(0)}}\Big(\frac{1}{\ve}\CH_{\QQ^{(0)}}\QQ_R+\CG(\QQ_R)\Big)\Big)+\nabla\cdot\mathfrak{R}+\mathfrak{R},\label{QvR-remainder-equation-2}\\
\nabla\cdot\vv_R=&~0.\label{QvR-remainder-equation-3}
\end{align}
It can be observed that the remainder system (\ref{QvR-remainder-equation-1})--(\ref{QvR-remainder-equation-3}) involves the singular term $\ve^{-1}\CH_{\QQ^{(0)}}\QQ_R$ in $\ve$. Therefore, as shown in \cite{WZZ3,LWZ,LW}, in order to obtain the uniform energy estimates,
we construct the following energy functionals:
\begin{align}
\mathfrak{E}(t)\eqdefa&\frac{1}{2}\int_{\mathbb{R}^3}\Big[\Big(|\vv_R|^2+\big(\CM^{-1}_{\QQ^{(0)}}\QQ_R\big)\cdot\QQ_R+\frac{1}{\ve}\big(\CH^{\ve}_{\QQ^{(0)}}\QQ_R\big)\cdot\QQ_R\Big)\nonumber\\
&+\ve^2\Big(|\nabla\vv_R|^2+\frac{1}{\ve}\big(\CH^{\ve}_{\QQ^{(0)}}\partial_i\QQ_R\big)\cdot\partial_i\QQ_R\Big)\nonumber\\
&+\ve^4\Big(|\Delta\vv_R|^2+\frac{1}{\ve}\big(\CH^{\ve}_{\QQ^{(0)}}\Delta\QQ_R\big)\cdot\Delta\QQ_R\Big)\Big]\ud\xx,\label{Ef-t-functional}\\
\mathfrak{F}(t)\eqdefa&\int_{\mathbb{R}^3}\Big[\Big(\eta|\nabla\vv_R|^2+\frac{1}{\ve}\CM_{\QQ^{(0)}}\big(\CH^{\ve}_{\QQ^{(0)}}\QQ_R\big)\cdot\big(\CH^{\ve}_{\QQ^{(0)}}\QQ_R\big)\Big)\nonumber\\
&+\ve^2\Big(\eta|\Delta\vv_R|^2+\frac{1}{\ve}\CM_{\QQ^{(0)}}\big(\CH^{\ve}_{\QQ^{(0)}}\partial_i\QQ_R\big)\cdot\big(\CH^{\ve}_{\QQ^{(0)}}\partial_i\QQ_R\big)\Big)\nonumber\\
&+\ve^4\Big(\eta|\nabla\Delta\vv_R|^2+\frac{1}{\ve}\CM_{\QQ^{(0)}}\big(\CH^{\ve}_{\QQ^{(0)}}\Delta\QQ_R\big)\cdot\big(\CH^{\ve}_{\QQ^{(0)}}\Delta\QQ_R\big)\Big)\Big]\ud\xx,\label{Ff-t-functional}
\end{align}
where $\eta>0$ and $\CH^{\ve}_{\QQ^{(0)}}\QQ_R=\CH_{\QQ^{(0)}}\QQ_R+\ve\CG(\QQ_R)$.

Using the the fact that $\CH_{\QQ^{(0)}}$ is positive semi-definite and $\CM^{-1}_{\QQ^{(0)}}$ is positive definite, we immediately obtain the following lemma (cf. \cite{WZZ3,LWZ}).
\begin{lemma}\label{basic-lemma-remainder}
There exists a positive constant $C$, such that
\begin{align*}
\|\QQ_R\|_{H^1}+\|(\ve\nabla^2\QQ_R,\ve^2\nabla^3\QQ_R)\|_{L^2}+\|(\vv_R,\ve\nabla\vv_R,\ve^2\nabla^2\vv_R)\|_{L^2}\leq&C\mathfrak{E}^{\frac{1}{2}},\\
\big\|\big(\ve^{-1}\CH^{\ve}_{\QQ^{(0)}}\QQ_R,\nabla\CH^{\ve}_{\QQ^{(0)}}\QQ_R,\ve\Delta\CH^{\ve}_{\QQ^{(0)}}\QQ_R\big)\big\|_{L^2}\leq&C(\mathfrak{E}+\mathfrak{F})^{\frac{1}{2}},\\
\big\|(\ve\nabla\CG(\QQ_R),\ve^2\Delta\CG(\QQ_R))\big\|_{L^2}+\big\|(\nabla\vv_R,\ve\nabla^2\vv_R,\ve^2\nabla^3\vv_R)\big\|_{L^2}\leq& C(\mathfrak{E}+\mathfrak{F})^{\frac{1}{2}}.
\end{align*}
\end{lemma}

\begin{corollary}\label{EF-corollary}
Let $E$ and $F$ be defined by \eqref{E-norm} and \eqref{F-norm}, respectively. Then it follows that
\begin{align*}
E\leq C\mathfrak{E}^{\frac{1}{2}},\quad F\leq C(\mathfrak{E}+\mathfrak{F})^{\frac{1}{2}}.
\end{align*}
\end{corollary}

The a priori estimate for the remainder $(\QQ_R,\vv_R)$ is stated as follows.

\begin{proposition}\label{uniform-bound-prop}
There exist two functions $C$ and $f$ depending on $(\QQ^{(k)},\vv^{(k)})$ and the parameters of the
system (but independent of $\ve$), such that if $(\QQ_R, \vv_R)$ be a smooth solution to the system \eqref{QvR-remainder-equation-1}--\eqref{QvR-remainder-equation-3}
on $[0, T]$, then for any $t\in[0, T]$, it satisfies
\begin{align*}
\frac{\ud}{\ud t}\mathfrak{E}(t)+\mathfrak{F}(t)\leq C(\ve E)(1+\mathfrak{E})+\ve f(\mathfrak{E})+C(\ve E)\ve\mathfrak{F}.
\end{align*}
\end{proposition}

To ensure the closure of energy estimates, we need to control the singular term $\frac{1}{\ve}\big\langle\QQ_R,\partial_t\big(\CH_{\QQ^{(0)}}\big)\QQ_R\big\rangle$, which comes from the $L^2$-inner product $\frac{1}{\ve}\frac{\ud}{\ud t}\big\langle\QQ_R,\CH^{\ve}_{\QQ^{(0)}}\QQ_R\big\rangle$.
The singular term possesses a good upper bound as a result of the eigen-decomposition \eqref{eigen-decomp}, given in the following lemma.

\begin{lemma}\label{key-lemma}
For any $\delta>0$, there exists a positive constant $C=C(\delta,\|\nabla_{t,\xx}\Fp\|_{L^{\infty}},\|\nabla\partial_t\Fp\|_{L^{\infty}})$ such that for any $\QQ\in\mathbb{Q}$, it follows that
\begin{align*}
\frac{1}{\ve}\big\langle\QQ,\partial_t\big(\CH_{\QQ^{(0)}}\big)\QQ\big\rangle\leq\delta\Big\|\frac{1}{\ve}\CH^{\ve}_{\QQ^{(0)}}\QQ\Big\|^2_{L^2}+C_{\delta}\Big(\frac{1}{\ve}\big\langle\CH^{\ve}_{\QQ^{(0)}}\QQ,\QQ\big\rangle+\|\QQ\|^2_{L^2}\Big),
\end{align*}
where $\CH^{\ve}_{\QQ^{(0)}}\QQ=\CH_{\QQ^{(0)}}\QQ+\ve\CG(\QQ)$.
\end{lemma}

\begin{proof}
Recall the decomposition $\QQ=\QQ_{\top}+\QQ_{\perp}$ with $\QQ_{\top}\in{\rm Ker}\CH_{\QQ^{(0)}}$ and $\QQ_{\perp}\in({\rm ker}\CH_{\QQ^{(0)}})^{\perp}$.
Then, we obtain
\begin{align}\label{key-lemma-decomposition}
\frac{1}{\ve}\big\langle\QQ,\partial_t\big(\CH_{\QQ^{(0)}}\big)\QQ\big\rangle=&\frac{1}{\ve}\big\langle\QQ_{\top},\partial_t\big(\CH_{\QQ^{(0)}}\big)\QQ_{\top}\big\rangle+\frac{2}{\ve}\big\langle\QQ_{\perp},\partial_t\big(\CH_{\QQ^{(0)}}\big)\QQ_{\top}\big\rangle\nonumber\\
&+\frac{1}{\ve}\big\langle\QQ_{\perp},\partial_t\big(\CH_{\QQ^{(0)}}\big)\QQ_{\perp}\big\rangle.
\end{align}
We deal with three terms on the right-hand side, respectively.
From \eqref{eigen-decomp},
we deduce that
\begin{align}
\partial_t\big(\CH_{\QQ^{(0)}}\big)=\sum_{k=1}^7\lambda_k\Big(\frac{\partial \ee_k}{\partial\Fp} \frac{\partial\Fp}{\partial t}\otimes \ee_k+\ee_k\otimes\frac{\partial \ee_k}{\partial\Fp}\frac{\partial\Fp}{\partial t}\Big). \label{dtH}
\end{align}
Due to the orthogonality $\langle \ee_k,\QQ_{\top}\rangle=0$, the first term on the right-hand side of (\ref{key-lemma-decomposition}) vanishes. For the second term, we derive from Proposition \ref{prop-linearized-operator} that
\begin{align*}
&\frac{1}{\ve}\big\langle\QQ_{\perp},\partial_t\big(\CH_{\QQ^{(0)}}\big)\QQ_{\top}\big\rangle=\Big\langle\CH^{-1}_{\QQ^{(0)}}\big(\partial_t\big(\CH_{\QQ^{(0)}}\big)\QQ_{\top}\big),\frac{1}{\ve}\CH_{\QQ^{(0)}}\QQ\Big\rangle\\
&\quad=\Big\langle\CH^{-1}_{\QQ^{(0)}}\big(\partial_t\big(\CH_{\QQ^{(0)}}\big)\QQ_{\top}\big),\frac{1}{\ve}\CH^{\ve}_{\QQ^{(0)}}\QQ\Big\rangle-\big\langle\CH^{-1}_{\QQ^{(0)}}\big(\partial_t\big(\CH_{\QQ^{(0)}}\big)\QQ_{\top}\big),\CG(\QQ)\big\rangle\\
&\quad\leq C\|\partial_t\Fp\|_{L^{\infty}}\|\QQ_{\top}\|_{L^2}\Big\|\frac{1}{\ve}\CH^{\ve}_{\QQ^{(0)}}\QQ\Big\|_{L^2}+C_1(\|\nabla\QQ\|^2_{L^2}+\|\QQ\|^2_{L^2}),
\end{align*}
where $C_1$ depends on $\|\nabla_{t,\xx}\Fp\|_{L^{\infty}}$ and $\|\nabla\partial_t\Fp\|_{L^{\infty}}$. Using Proposition \ref{prop-linearized-operator} again, the third term on the right-hand side of (\ref{key-lemma-decomposition}) can be estimated as
\begin{align*}
\frac{1}{\ve}\big\langle\QQ_{\perp},\partial_t\big(\CH_{\QQ^{(0)}}\big)\QQ_{\perp}\big\rangle\leq&C\|\partial_t\Fp\|_{L^{\infty}}\frac{1}{\ve}\|\QQ_{\perp}\|^2_{L^2}\leq C\|\partial_t\Fp\|_{L^{\infty}}\frac{1}{\ve}\big\langle\CH_{\QQ^{(0)}}\QQ,\QQ\big\rangle\\
\leq& C\|\partial_t\Fp\|_{L^{\infty}}\frac{1}{\ve}\big\langle\CH^{\ve}_{\QQ^{(0)}}\QQ,\QQ\big\rangle,
\end{align*}
where we have used the fact that $\langle\QQ,\CG(\QQ)\rangle\geq0$. This completes the proof of the lemma.
\end{proof}

{\it Proof of Proposition \ref{uniform-bound-prop}}.
The proof will be divided into four steps.

{\it Step 1. $L^2$-estimate}.  From the system of remainder (\ref{QvR-remainder-equation-1})--(\ref{QvR-remainder-equation-3}) and Lemma \ref{basic-lemma-remainder}, we deduce that
\begin{align}\label{QQR-remiander-L2-estimate}
&\big\langle\partial_t\QQ_R,\CM^{-1}_{\QQ^{(0)}}\QQ_R\big\rangle+\frac{1}{\ve}\langle(\CH^{\ve}_{\QQ^{(0)}}\QQ_R),\QQ_R\rangle\nonumber\\
&\quad=\big\langle\CM^{-1}_{\QQ^{(0)}}(\CV_{\QQ^{(0)}}\kappa_R),\QQ_R\big\rangle+\langle\mathfrak{R},\CM^{-1}_{\QQ^{(0)}}\QQ_R\rangle\nonumber\\
&\quad\leq C\|\QQ_R\|_{L^2}(\|\nabla\vv_R\|_{L^2}+\|\mathfrak{R}\|_{L^2})\leq \delta_0\mathfrak{F}+C_{\delta_0}\mathfrak{E}+C\|\mathfrak{R}\|^2_{L^2},
\end{align}
together with the following relation:
\begin{align}\label{QQ-vv-R-remiander-L2-estimate}
&\langle\partial_t\vv_R,\vv_R\rangle+\Big\langle\partial_t\QQ_R,\frac{1}{\ve}\CH^{\ve}_{\QQ^{(0)}}\QQ_R\Big\rangle\nonumber\\
&+\eta\|\nabla\vv_R\|^2_{L^2}+\frac{1}{\ve}\big\langle\CM_{\QQ^{(0)}}\big(\CH^{\ve}_{\QQ^{(0)}}\QQ_R\big),\CH^{\ve}_{\QQ^{(0)}}\QQ_R\big\rangle\nonumber\\
&\quad=-\big\langle\CP_{\QQ^{(0)}}\kappa_R,\kappa_R\big\rangle-\Big\langle\CN_{\QQ^{(0)}}\Big(\frac{1}{\ve}\CH^{\ve}_{\QQ^{(0)}}\QQ_R\Big),\kappa_R\Big\rangle+\langle\nabla\cdot\mathfrak{R}+\mathfrak{R},\vv_R\rangle\nonumber\\
&\qquad+\Big\langle\CV_{\QQ^{(0)}}\kappa_R,\frac{1}{\ve}\CH^{\ve}_{\QQ^{(0)}}\QQ_R\Big\rangle+\Big\langle\mathfrak{R},\frac{1}{\ve}\CH^{\ve}_{\QQ^{(0)}}\QQ_R\Big\rangle\nonumber\\
&\quad\leq \delta_0\mathfrak{F}+C_{\delta_0}\mathfrak{E}+C\|\mathfrak{R}\|^2_{L^2},
\end{align}
where we have used the relation $\CN_{\QQ^{(0)}}=\CV^T_{\QQ^{(0)}}$ and the positive definiteness of $\CP_{\QQ^{(0)}}$.

{\it Step 2. $H^1$-estimate}.
We apply the derivative $\partial_i$ on (\ref{QvR-remainder-equation-1}) and take the $L^2$-inner
product with $\frac{1}{\ve}\CH^{\ve}_{\QQ^{(0)}}\partial_i\QQ_R$. Again by acting $\partial_i$ on (\ref{QvR-remainder-equation-2}) and taking the $L^2$-inner product with
$\partial_i\vv_R$, we deduce that
\begin{align}\label{H1-remainder-estimate}
&\ve^2\big\langle\partial_t\partial_i\vv_R,\partial_i\vv_R\big\rangle+\ve\big\langle\partial_t\partial_i\QQ_R,\CH^{\ve}_{\QQ^{(0)}}\partial_i\QQ_R\big\rangle+\ve^2\eta\|\nabla\partial_i\vv_R\|^2_{L^2}\nonumber\\
&\quad=-\ve^2\big\langle\partial_i(\CP_{\QQ^{(0)}}\kappa_R),\partial_i\kappa_R\big\rangle-\ve\big\langle\partial_i\big(\CN_{\QQ^{(0)}}(\CH^{\ve}_{\QQ^{(0)}}\QQ_R)\big),\nabla\partial_i\vv_R\big\rangle\nonumber\\
&\qquad-\big\langle\partial_i\big(\CM_{\QQ^{(0)}}\CH^{\ve}_{\QQ^{(0)}}\QQ_R\big),\CH^{\ve}_{\QQ^{(0)}}\partial_i\QQ_R\big\rangle+\ve\big\langle\partial_i\big(\CV_{\QQ^{(0)}}\kappa_R\big),\CH^{\ve}_{\QQ^{(0)}}\partial_i\QQ_R\big\rangle\nonumber\\
&\qquad+\ve^2\langle\nabla\cdot\partial_i\mathfrak{R}+\partial_i\mathfrak{R},\partial_i\vv_R\rangle+\ve\big\langle\partial_i\mathfrak{R},\CH^{\ve}_{\QQ^{(0)}}\partial_i\QQ_R\big\rangle\nonumber\\
&\quad\eqdefa I+II+III+IV+V+VI.
\end{align}
The terms on the right-hand sides can be estimated as follows:
\begin{align*}
I\leq&-\ve^2\big\langle\CP_{\QQ^{(0)}}\partial_i\kappa_R,\partial_i\kappa_R\big\rangle+C\|\ve\nabla\vv_R\|_{L^2}\|\ve\nabla\partial_i\QQ_R\|_{L^2}
\leq\delta_0\mathfrak{F}+C\mathfrak{E},\\
II\leq&-\ve\big\langle\CN_{\QQ^{(0)}}(\CH^{\ve}_{\QQ^{(0)}}\partial_i\QQ_R),\nabla\partial_i\vv_R\big\rangle\\
&+\ve\Big\langle[\CM_{\QQ^{(0)}}\CH_{\QQ^{(0)}},\partial_i]\QQ_R+\ve[\CM_{\QQ^{(0)}}\CG,\partial_i]\QQ_R,\nabla\partial_i\vv_R\Big\rangle\\
\leq&-\ve\big\langle\CN_{\QQ^{(0)}}(\CH^{\ve}_{\QQ^{(0)}}\partial_i\QQ_R),\nabla\partial_i\vv_R\big\rangle+C(\|\QQ_R\|_{L^2}+\ve\|\QQ_R\|_{H^2})\|\ve\nabla\partial_i\vv_R\|_{L^2}\\
\leq&-\ve\big\langle\CN_{\QQ^{(0)}}(\CH^{\ve}_{\QQ^{(0)}}\partial_i\QQ_R),\nabla\partial_i\vv_R\big\rangle+\delta_0\mathfrak{F}+C\mathfrak{E},\\
III\leq&-\big\langle\CM_{\QQ^{(0)}}\big(\CH^{\ve}_{\QQ^{(0)}}\partial_i\QQ_R\big),\CH^{\ve}_{\QQ^{(0)}}\partial_i\QQ_R\big\rangle+\big\langle[\CM_{\QQ^{(0)}}\CH^{\ve}_{\QQ^{(0)}},\partial_i]\QQ_R,\CH^{\ve}_{\QQ^{(0)}}\partial_i\QQ_R\big\rangle\\
\leq&-\big\langle\CM_{\QQ^{(0)}}\big(\CH^{\ve}_{\QQ^{(0)}}\partial_i\QQ_R\big),\CH^{\ve}_{\QQ^{(0)}}\partial_i\QQ_R\big\rangle+C(\|\QQ_R\|_{L^2}+\ve\|\QQ_R\|_{H^2})\|\CH^{\ve}_{\QQ^{(0)}}\partial_i\QQ_R\|_{L^2}\\
\leq&-\big\langle\CM_{\QQ^{(0)}}\big(\CH^{\ve}_{\QQ^{(0)}}\partial_i\QQ_R\big),\CH^{\ve}_{\QQ^{(0)}}\partial_i\QQ_R\big\rangle+\delta_0\mathfrak{F}+C\mathfrak{E},\\
IV\leq&\ve\big\langle\CV_{\QQ^{(0)}}\partial_i\kappa_R,\CH^{\ve}_{\QQ^{(0)}}\partial_i\QQ_R\big\rangle+\ve\|\nabla\vv_R\|_{L^2}\|\CH^{\ve}_{\QQ^{(0)}}\partial_i\QQ_R\|_{L^2}\\
\leq&\ve\big\langle\CV_{\QQ^{(0)}}\partial_i\kappa_R,\CH^{\ve}_{\QQ^{(0)}}\partial_i\QQ_R\big\rangle+\delta_0\mathfrak{F}+C\mathfrak{E},\\
V\leq&\|\ve\partial_i\mathfrak{R}\|_{L^2}\|\ve\nabla\partial_i\vv_R\|_{L^2}
+\|\ve\partial_i\mathfrak{R}\|_{L^2}\|\ve\partial_i\vv_R\|_{L^2}\\
\leq&\delta_0 \Ff+C\Ef+C\|\ve\partial_i\mathfrak{R}\|^2_{L^2},\\
VI\leq&\|\ve\partial_i\mathfrak{R}\|_{L^2}\|\CH^{\ve}_{\QQ^{(0)}}\partial_i\QQ_R\|_{L^2}
\leq \delta_0 \Ff+C\|\ve\partial_i\mathfrak{R}\|^2_{L^2}.
\end{align*}
Consequently, substituting the above estimates into (\ref{H1-remainder-estimate}) and using the cancelation relation yields
\begin{align}\label{H1-remainder-estimate-final}
&\ve^2\big\langle\partial_t\partial_i\vv_R,\partial_i\vv_R\big\rangle+\ve\big\langle\partial_t\partial_i\QQ_R,\CH^{\ve}_{\QQ^{(0)}}\partial_i\QQ_R\big\rangle\nonumber\\
&+\ve^2\eta\|\nabla\partial_i\vv_R\|^2_{L^2}+\big\langle\CM_{\QQ^{(0)}}\big(\CH^{\ve}_{\QQ^{(0)}}\partial_i\QQ_R\big),\CH^{\ve}_{\QQ^{(0)}}\partial_i\QQ_R\big\rangle\nonumber\\
&\quad\leq \delta_0 \Ff+C\mathfrak{E}+C\|\ve\partial_i\mathfrak{R}\|^2_{L^2}.
\end{align}

{\it Step 3. $H^2$-estimate}. Similar to Step 2, we first apply the derivative operator $\Delta$ on
(\ref{QvR-remainder-equation-1}), then multiply with $\frac{1}{\ve}\CH^{\ve}_{\QQ^{(0)}}\Delta\QQ_R$ and integrate the resulting identity on $\mathbb{R}^3$.
Again applying the operator $\Delta$ on (\ref{QvR-remainder-equation-2}) and taking the $L^2$-inner product with $\Delta\vv_R$
enable us to derive the following equality:
\begin{align}\label{QQvvR-remainder-H2}
&\ve^4\big\langle\partial_t\Delta\vv_R,\Delta\vv_R\big\rangle+\ve^3\big\langle\partial_t\Delta\QQ_R,\CH^{\ve}_{\QQ^{(0)}}\Delta\QQ_R\big\rangle\nonumber\\
&\quad=-\ve^4\eta\|\nabla\Delta\vv_R\|^2_{L^2}-\ve^4\big\langle\CP_{\QQ^{(0)}}\Delta\kappa_R,\Delta\kappa_R\big\rangle-\ve^3\big\langle\CN_{\QQ^{(0)}}(\CH^{\ve}_{\QQ^{(0)}}\Delta\QQ_R),\Delta\kappa_R\big\rangle\nonumber\\
&\qquad-\ve^2\big\langle\CM_{\QQ^{(0)}}(\CH^{\ve}_{\QQ^{(0)}}\Delta\QQ_R),\CH^{\ve}_{\QQ^{(0)}}\Delta\QQ_R\big\rangle+\ve^3\big\langle\CV_{\QQ^{(0)}}\Delta\kappa
_R,\CH^{\ve}_{\QQ^{(0)}}\Delta\QQ_R\big\rangle\nonumber\\
&\qquad\underbrace{-\ve^4\big\langle[\Delta,\CP_{\QQ^{(0)}}]\kappa_R,\Delta\kappa_R\big\rangle+\ve^3\big\langle[\Delta,\CV_{\QQ^{(0)}}]\kappa_R,\CH^{\ve}_{\QQ^{(0)}}\Delta\QQ_R\big\rangle}_{\CI_1}\nonumber\\
&\qquad\underbrace{-\ve^3\big\langle[\Delta,\CN_{\QQ^{(0)}}\CH^{\ve}_{\QQ^{(0)}}]\QQ_R,\Delta\kappa_R\big\rangle
-\ve^2\big\langle[\Delta,\CM_{\QQ^{(0)}}\CH^{\ve}_{\QQ^{(0)}}]\QQ_R,\CH^{\ve}_{\QQ^{(0)}}\Delta\QQ_R\big\rangle}_{\CI_2}\nonumber\\
&\qquad+\underbrace{\ve^4\big\langle\nabla\cdot\Delta\mathfrak{R}+\Delta\mathfrak{R},\Delta\vv_R\big\rangle+\ve^3\big\langle\Delta\mathfrak{R},\CH^{\ve}_{\QQ^{(0)}}\Delta\QQ_R\big\rangle}_{\CI_3}.
\end{align}
For the terms $\CI_i(i=1,2,3)$, using Lemma \ref{commutor-YY-QQ-lemma} and Lemma \ref{basic-lemma-remainder}, we derive that
\begin{align*}
\CI_1\leq&C\|\ve^2\kappa_R\|_{H^1}\|\ve^2\Delta\kappa_R\|_{L^2}+C\|\ve^2\kappa_R\|_{H^1}\|\ve\CH^{\ve}_{\QQ^{(0)}}\Delta\QQ_R\|_{L^2}
\leq\delta_0\mathfrak{F}+C\mathfrak{E},\\
\CI_2\leq&C(\|\ve\QQ_R\|_{H^1}+\|\ve^2\CG(\QQ_R)\|_{H^1})\big(\|\ve^2\Delta\kappa_R\|_{L^2}+\|\ve\CH^{\ve}_{\QQ^{(0)}}\Delta\QQ_R\|_{L^2}\big)
\leq\delta_0\mathfrak{F}+C\mathfrak{E},\\
\CI_3\leq&C\big(\|\ve^2\nabla\Delta\vv_R\|_{L^2}+\|\ve^2\Delta\vv_R\|_{L^2}+\|\ve\CH^{\ve}_{\QQ^{(0)}}\Delta\QQ_R\|_{L^2}\big)\|\ve^2\Delta\mathfrak{R}\|_{L^2}\\
\leq&\delta_0\mathfrak{F}+C(\mathfrak{E}+\|\ve^2\Delta\mathfrak{R}\|^2_{L^2}).
\end{align*}
Combining the above estimates with (\ref{QQvvR-remainder-H2}) yields
\begin{align}\label{QQvvR-remainder-H2-final}
&\ve^4\big\langle\partial_t\Delta\vv_R,\Delta\vv_R\big\rangle+\ve^3\big\langle\partial_t\Delta\QQ_R,\CH^{\ve}_{\QQ^{(0)}}\Delta\QQ_R\big\rangle\nonumber\\
&\quad\leq-\ve^4\eta\|\nabla\Delta\vv_R\|^2_{L^2}-\ve^2\big\langle\CM_{\QQ^{(0)}}(\CH^{\ve}_{\QQ^{(0)}}\Delta\QQ_R),\CH^{\ve}_{\QQ^{(0)}}\Delta\QQ_R\big\rangle\nonumber\\
&\qquad+\delta_0\mathfrak{F}+C(\mathfrak{E}+\|\ve^2\Delta\mathfrak{R}\|^2_{L^2}).
\end{align}

{\it Step 4. Closure of error estimates}. Recall the definition of $\CH^{\ve}_{\QQ^{(0)}}\QQ_R$ and Lemma \ref{key-lemma}.
It follows that
\begin{align*}
\frac{1}{\ve}\frac{\ud}{\ud t}\big\langle\QQ_R,\CH^{\ve}_{\QQ^{(0)}}\QQ_R\big\rangle=&\frac{2}{\ve}\big\langle\partial_t\QQ_R,\CH^{\ve}_{\QQ^{(0)}}\QQ_R\big\rangle+\frac{1}{\ve}\big\langle\QQ_R,\partial_t\big(\CH_{\QQ^{(0)}}\big)\QQ_R\big\rangle\\
\leq&\frac{2}{\ve}\big\langle\partial_t\QQ_R,\CH^{\ve}_{\QQ^{(0)}}\QQ_R\big\rangle+\delta\Big\|\frac{1}{\ve}\CH^{\ve}_{\QQ^{(0)}}\QQ_R\Big\|^2_{L^2}\\
&+C_{\delta}\Big(\frac{1}{\ve}\big\langle\CH^{\ve}_{\QQ^{(0)}}\QQ_R,\QQ_R\big\rangle+\|\QQ_R\|^2_{L^2}\Big),
\end{align*}
which further implies that
\begin{align*}
\frac{1}{2\ve}\frac{\ud}{\ud t}\big\langle\QQ_R,\CH^{\ve}_{\QQ^{(0)}}\QQ_R\big\rangle\leq \frac{1}{\ve}\big\langle\partial_t\QQ_R,\CH^{\ve}_{\QQ^{(0)}}\QQ_R\big\rangle+\delta\mathfrak{F}+C\mathfrak{E}.
\end{align*}
In a similar argument, we obtain the following inequalities:
\begin{align*}
 \frac{\ve}{2}\frac{\ud}{\ud t}\big\langle \partial_i\QQ_R,\CH^{\ve}_{\QQ^{(0)}}\partial_i\QQ_R\big\rangle
 \leq&\ve\big\langle\partial_t\partial_i\QQ_R,\CH^{\ve}_{\QQ^{(0)}}\partial_i\QQ_R\big\rangle+\delta\Ff+C\Ef,\\
 \frac{\ve^3}{2}\frac{\ud}{\ud t}\big\langle \Delta \QQ_R,\CH^{\ve}_{\QQ^{(0)}}\Delta\QQ_R\big\rangle
 \leq&\ve^3\big\langle\partial_t\Delta Q_R,\CH^{\ve}_{\QQ^{(0)}}\Delta\QQ_R\big\rangle+\delta\Ff+C\Ef.
\end{align*}
Therefore, together with (\ref{QQR-remiander-L2-estimate})--(\ref{QQ-vv-R-remiander-L2-estimate}), (\ref{H1-remainder-estimate-final}) and (\ref{QQvvR-remainder-H2-final}), by using the property of {\it good terms} $\mathfrak{R}$ and Corollary \ref{EF-corollary}, we arrive at
\begin{align*}
 \frac{1}{2}\frac{\ud}{\ud t}\Ef(t)+\Ff(t)\leq& \delta\Ff+C_\delta\Ef+\|\mathfrak{R}\|_{L^2}^2
 +\|\ve\nabla\mathfrak{R}\|_{L^2}^2+\|\ve^2\Delta\mathfrak{R}\|_{L^2}^2\\
 \leq&\delta\Ff+C_\delta\Ef+C(\ve E)(1+E+\ve F)+\ve f(E)\\
 \leq&\delta\Ff+C_\delta\Ef+C(\ve^2\Ef)(1+\Ef+\ve^2\Ff)+\ve^2f(\Ef),
\end{align*}
which concludes the proof of Proposition \ref{uniform-bound-prop} by taking $\delta$ enough small.

\subsection{Proof of Theorem \ref{biaixal-limit-theorem}}
Given the initial data $(\QQ^{\ve}(\xx,0),\vv^{\ve}(\xx,0))\in H^3\times H^2$, Theorem \ref{locall-posedness-theorem} tells us
that there exists a maximal time $T_\ve>0$ and a unique solution $(\QQ^{\ve},\vv^{\ve})$ to the system (\ref{Re-MB-Q-tensor-1})--(\ref{Re-MB-Q-tensor-3}) such that
\begin{align*}
\QQ^{\ve}\in C([0,T_{\ve});H^3),\quad\vv^{\ve}\in C([0,T_{\ve});H^2)\cap L^2(0,T_{\ve};H^3).
\end{align*}
From Proposition \ref{uniform-bound-prop}, we have
\begin{align*}
\frac{\ud }{\ud t}\Ef(t)+\Ff(t)\leq
 C(\ve \Ef)\big(1+\Ef\big)+\ve f(\Ef)+C(\ve \Ef)\ve\Ff,
\end{align*}
for any $t\in [0,T_{\ve}]$. Under the assumptions of Theorem \ref{locall-posedness-theorem}, it follows that
\begin{align*}
\Ef(0)\leq C_1\Big(\|\vv_{I,R}^\ve\|_{H^2}+\|\QQ_{I,R}^\ve\|_{H^3}+\ve^{-1}\|\MP^{\rm out}(\QQ^{\ve}_{I,R})\|_{L^2}\Big)\le C_1 E_0.
\end{align*}
Let $E_1=(2+C_1E_0){e}^{T}-2>\Ef(0)$, and
\begin{align*}
T_1=\sup\{t\in[0,T_\ve]: \Ef(t)\leq E_1\}.
\end{align*}
If we take $\ve_0$ small enough such that
\begin{align*}
 C(\ve_0 E_1)\leq 1,\quad\ve_0f(E_1)\leq 1,\quad\ve_0\leq\frac{1}{2}.
\end{align*}
Then, for $t\leq T_1$, it holds $\frac{\ud}{\ud t}\Ef(t)\leq 2+\Ef$.
Hence, by means of a continuous argument we conclude that $T\leq T_\ve$ and $\Ef(t)\leq E_1$ for $t\in[0,T]$. This completes the proof of Theorem \ref{biaixal-limit-theorem}.

\section{Rigorous uniaxial limit of two-tensor hydrodynamics}
\label{genappr}

We have mentioned that the minimizer of the bulk energy $F_b$ may be uniaxial.
In this case, under the limit $\ve\to 0$ the two-tensor hydrodynamics is reduced to the Ericksen--Leslie model, for which the formal derivation is given in \cite{LX}.
Following the procedure in this work, the rigorous uniaixal limit can also be established.
In what follows, we point out the main differences from the biaxial limit, with comparison with previous works on the uniaxial limit of one-tensor models.




The uniaxial minimizer has the following form:
\begin{align}\label{uniaxial-forms}
Q^{(0)}_i=s_i\Big(\nn^2_1-\frac{\Fi}{3}\Big),\quad i=1,2,
\end{align}
where $s_i(i=1,2)$ are two scalars, and $\nn_1$ may take any unit vector.
Thus, the uniaxial minimizer (\ref{uniaxial-forms}) determines a two-dimensional manifold.
Let us denote $\CH_{\nn_1}\eqdefa\CH_{\QQ^{(0)}}=\CJ'(\QQ^{(0)})$.
Analogous to \eqref{xiker}, the tangential space of the manifold at $\nn_1$ gives belongs to the zero-eigenvector space of the Hessian of $\CH_{\nn_1}$.
Here, we still assume that the tangential space and the zero-eigenvalue space are identical (cf. Assumption \ref{HQ-3}).
In other words, ${\rm Ker}\CH_{\nn_1}$ is a two-dimensional subspace of $\mathbb{Q}$.
It is worth noting that ${\rm Ker}\CH_{\nn_1}$ depends on $s_i$.
This is different from one-tensor models, where such kernel is independent of the scalar, given by
$
\{\nn_1\otimes\nn_1'+\nn_1'\otimes\nn_1:\nn_1'\perp \nn_1\}.
$


The kernel space ${\rm Ker}\CH_{\nn_1}$ is used to cancel the non-leading terms in the Hilbert expansion (cf. the discussion below \eqref{biaxial-global-minimum}).
Then, the uniaxial version of the $O(1)$ system  \eqref{O-0-epsion-system-1}--\eqref{O-0-epsion-system-3} can be reduced to the Ericksen--Leslie model.

The solvability of the $\QQ^{(k)}(k=1,2,3)$ and $\vv^{(k)}(k=1,2)$ is also handled by projecting $\QQ^{(k)}$ into the two spaces ${\rm Ker}\CH_{\nn_1}$ and $({\rm Ker}\CH_{\nn_1})^{\perp}$ by $\MP^{\rm in}$ and $\MP^{\rm out}$, respectively.
Taking the solvability of the $O(\ve)$ system (\ref{O-1-epsion-system-1})--(\ref{O-1-epsion-system-3}) as an example, the key ingredient is to derive a closed system for $(\QQ^{(1)}_{\top},\vv^{(1)})$ and to show that such a system is linear and has a closed energy estimate.
However, in the uniaxial case, we still suffer from the difficulty that the projection $\MP^{\rm in}$ and the operator $\CM^{(0)}$ are noncommutative.
Actually, the commutativity in the previous works \cite{WZZ3,LWZ,LW} for one-tensor models is a special result from the fact that their kernel is independent of scalar order parameters.
This implies that the method by directly taking the projection $\MP^{\rm in}$ on the $O(\ve)$ system will no longer be available, since $\CM^{(0)}\CH_{\QQ^{(0)}}\QQ^{(2)}\notin{\rm Ker}\CH_{\QQ^{(0)}}$ even if $\CH_{\QQ^{(0)}}\QQ^{(2)}\in{\rm Ker}\CH_{\QQ^{(0)}}$. Thus, we here overcome the difficulty by the same way in Subsection \ref{existen-hilbert-sub}. The only difference is that the kernel space ${\rm Ker}\CH_{\QQ^{(0)}}$ is a two-dimensional but not a three-dimensional subspace.



For the system for the remainder, we also need to deal with the singular term $\frac{1}{\ve}\big\langle\QQ_R,\partial_t\big(\CH_{\QQ^{(0)}}\big)\QQ_R\big\rangle$ in $\ve$.
In the uniaxial case, we still have the eigen-decomposition in the form \eqref{dtH}, but the number of the $\ee_k$ is now eight instead of seven.
The other estimates are established identically.

\appendix

\section{Biaxial minimizer and its Hessian}
We give one numerical example to illustrate why we can claim Assumption 1.
In the bulk energy $F_b$, we take the entropy term as the quasi-entropy $\nu\Xi_2$, and the coefficients as
$$
c_{02}=-35\nu,\quad c_{03}=-20\nu,\quad c_{04}=-20\nu.
$$
Numerical experiments indicate that the minimizer is biaxial with
$$
s_1\approx 0.6263,\quad b_1\approx-0.0526,\quad s_2\approx -0.2377,\quad b_2\approx 0.2890.
$$
At this minimizer, let us look into the eigenvalues of the Hessian.
There are three eigenvalues very close to zero (with absolute values $<10^{-8}$).
The other eigenvalues are all positive, with the smallest one taking the value $\approx 8.4870$.

\section{Some basic estimates in Sobolev spaces}

The following product estimates and commutator estimates are  well-known, see \cite{Triebel,BCD} for example, and frequently used in this paper.

\begin{lemma}\label{lem:product}
Let $s\geq 0$. Then for any multi-index $\alpha,\beta,\gamma,\delta$, it follows that
\begin{align*}
\|\partial^\alpha f\partial^\beta g\|_{H^s}\leq &~ C\big(\|f\|_{L^\infty}\|g\|_{H^{s+|\alpha|+|\beta|}}
+\|g\|_{L^\infty}\|f\|_{H^{s+|\alpha|+|\beta|}}\big),\\
\|\partial^\alpha f\partial^\beta g\|_{H^s}\leq &~ C\|f\|_{H^{s+|\alpha|+|\gamma|}}\|g\|_{H^{s+|\beta|+|\delta|}},\quad\text{ if } s+|\gamma|+|\delta|\geq 2.
\end{align*}
In particular, there holds
\begin{align*}
\|fg\|_{H^s}\leq &~ C\big(\|f\|_{L^\infty}\|g\|_{H^{s}}+\|g\|_{L^\infty}\|f\|_{H^{s}}\big); \\
\|fg\|_{H^s}\leq &~ C\|f\|_{H^s}\|g\|_{H^{s}},\quad\text{ if } s\geq 2;\\
\|fg\|_{H^k}\leq &~ C\min\{ \|f\|_{H^k}\|g\|_{H^{2}}, \|f\|_{H^2}\|g\|_{H^{k}}\},\quad\text{ if } 0\leq k\leq 2.
\end{align*}
\end{lemma}
\begin{lemma}\label{lem:composition}
Let $s\geq 0$ and $F(\cdot)\in C^\infty(\mathbb{R}^d)$ with $F(0)=0$. Then we have
\begin{align*}
 \|F(f)\|_{H^s}\le C(\|f\|_{L^\infty})\|f\|_{H^s}.
\end{align*}

\end{lemma}

\begin{lemma}\label{lem:commutator}
Assume that $\alpha$ is a multiple index. Then it follows that
\begin{align*}
\|[\partial^{\alpha}, g]f\|_{L^2}\leq C\big(\|\nabla g\|_{L^\infty}\|f\|_{H^{|\alpha|-1}}
+\|\nabla g\|_{H^{|a|-1}}\|f\|_{L^\infty}\big).
\end{align*}
In particular, if $|\alpha|\geq 2$, there holds
\begin{align*}
\|[\partial^{\alpha}, g]f\|_{L^2}\leq C\|g\|_{H^{|\alpha|+1}}\|f\|_{H^{|\alpha|-1}},\quad \|[\partial^{\alpha+1}, g]f\|_{L^2}\leq C\|g\|_{H^{|\alpha|+1}}\|f\|_{H^{|\alpha|}}.
\end{align*}

\end{lemma}

\begin{lemma}\label{lem:difference}
Let $\Omega$ be a convex domain in $\mathbb{R}^d$ and $k\ge 0$ be an integer. Assume
$F(\cdot)\in C^\infty(\Omega)$ and $k'=\max\{k,2\}$. Then it follows that
\begin{align*}
\|F(u)-F(v)\|_{H^k}\leq C(\|u\|_{L^\infty},\|v\|_{L^\infty})(1+\|u\|_{H^{k'}}+\|v\|_{H^{k'}} )\|u-v\|_{H^k}.
\end{align*}

\end{lemma}
\begin{proof} We may as well assume $F'(0)=0$, otherwise, we consider $G(u)=F(u)-u\cdot F'(0)$.
Choose any two points $u,v\in\mathbb{R}^d, u\neq v$. We have
\begin{align*}
F(u)-F(v)=&\int^1_0\frac{\ud}{\ud t}F(v+t(u-v))\ud t\\
=&(u-v)\cdot\int_{0}^1F'(v+t(u-v))\ud t.
\end{align*}
Consequently, from the above equation, we can derive that
\begin{align*}
\|F(u)-F(v)\|_{L^2}&~\le \|u-v\|_{L^2}\sup_{t\in[0,1]}\|F'(v+t(u-v))\|_{L^\infty}\\
&~ \le C(\|u\|_{L^\infty},\|v\|_{L^\infty})\|u-v\|_{L^2} ,\\
\|\nabla(F(u)-F(v))\|_{L^2}&~\le \|\nabla(u-v)\|_{L^2}\sup_{t\in[0,1]}\|F'(v+t(u-v))\|_{L^\infty}\\
&~\quad+\|u-v\|_{H^1}\sup_{t\in[0,1]}\|\nabla(F'(v+t(u-v)))\|_{H^1} \\
&~\le C(\|u\|_{L^\infty},\|v\|_{L^\infty})(\|u\|_{H^2}+\|v\|_{H^2})\|u-v\|_{H^1}.
\end{align*}
Further, for $k\geq 2$, we have
\begin{align*}
\|F(u)-F(v)\|_{H^k}\le&~ C\Big(\|u-v\|_{L^\infty}\sup_{t\in[0,1]}\|F'(v+t(u-v))\|_{H^k}\\
&\qquad+\|u-v\|_{H^k}\sup_{t\in[0,1]}\|F'(v+t(u-v))\|_{L^\infty}\Big)\\
\le &~C(\|u\|_{L^\infty},\|v\|_{L^\infty})(1+\|u\|_{H^k}+\|v\|_{H^k})\|u-v\|_{H^k}.
\end{align*}
In the above derivation, we have used the following estimate:
\begin{align*}
\|F'(v+t(u-v))\|_{H^k}&~\le C(\|v+t(u-v)\|_{L^\infty})\|v+t(u-v)\|_{H^k} \\
&~\le C(\|u\|_{L^\infty},\|v\|_{L^\infty})(\|u\|_{H^k}+\|v\|_{H^k}),
\end{align*}
which can be induced by Lemma \ref{lem:composition}. Thus, we conclude the proof of the lemma.
\end{proof}

\bigskip
\noindent{\bf Acknowledgments.} Sirui Li is partially supported by the NSFC under grant No. 12061019 and by the Growth Foundation for Youth Science and Technology Talent of Educational Commission of Guizhou Province of China under grant No. [2021]087. Jie Xu is partially supported by the NSFC under grant Nos. 12288201 and 12001524.

\end{document}